\documentclass{article}
\usepackage[utf8]{inputenc}
\usepackage{amsfonts}
\usepackage{mathtools}
\usepackage{amsmath}
\usepackage{amssymb}
\usepackage{authblk}
\usepackage{natbib}
\usepackage{babel}
\usepackage{graphicx}
\usepackage{graphics}
\usepackage[font=small,labelfont=bf]{caption}
\usepackage{xcolor}
\usepackage{bbm}
\usepackage[margin = 1.2in]{geometry}
\usepackage{hyperref}
\usepackage{pgfplots}
\usepackage{amsthm}
\pgfplotsset{width = 5cm, compat = 1.17}

\title{On High Dimensional Behaviour of Some\\ Two-Sample Tests Based on Ball Divergence}
\author{Bilol Banerjee}
\author{Anil K. Ghosh}
\affil{Theoretical Statistics and Mathematics Unit\\ Indian Statistical Institute, Kolkata}
\date{\null}

\renewcommand{\P}{\mathbb{P}}

\newtheorem{prop}{Proposition}[section]
\newtheorem{thm}{Theorem}[section]
\newtheorem{lemma}{Lemma}[section]

\newtheorem{rem}{Remark}

\newtheorem{lemmaA}{Lemma A.\ignorespaces}

\newcommand{\muvec}{\mbox{\boldmath $\mu$}}
\newcommand{\sigmat}{\mbox{\boldmath $\Sigma$}}

\renewcommand{\P}{\mathbb{P}}

\newcommand{\R}{\mathbb{R}}
\newcommand{\E}{\mathbb{E}}

\newcommand{\tikzcircle}[2][red,fill=red]{\tikz[baseline=-0.5ex]\draw[#1,radius=#2] (0,0) circle ;}
\newcommand{\tikzcirclev}[2][violet,fill=violet]{\tikz[baseline=-0.5ex]\draw[#1,radius=#2] (0,0) circle ;}
\newcommand{\tikzcirclep}[2][purple,fill=purple]{\tikz[baseline=-0.5ex]\draw[#1,radius=#2] (0,0) circle ;}
\newcommand{\tikzcirclem}[2][magenta,fill=magenta]{\tikz[baseline=-0.5ex]\draw[#1,radius=#2] (0,0) circle ;}

\begin{document}

\maketitle
\begin{abstract}
  In this article, we propose some two-sample tests based on \text{ball divergence} and investigate their high dimensional behavior. First, we study their behavior for High Dimension, Low Sample Size (HDLSS) data, and under appropriate regularity conditions, we establish their consistency in the HDLSS regime, where the dimension of the data grows to infinity while the sample sizes from the two distributions remain fixed. Further, we show that these conditions can be relaxed when the sample sizes also increase with the dimension, and in such cases, consistency can be proved even for shrinking alternatives. We use a simple example involving two normal distributions to prove that even when there are no consistent tests in the HDLSS regime, the powers of the proposed tests can converge to unity if the sample sizes increase with the dimension at an appropriate rate. This rate is obtained by establishing the minimax rate optimality of our tests over a certain class of alternatives.  Several simulated and benchmark data sets are analyzed to compare the performance of these proposed tests with the state-of-the-art methods that can be used for testing the equality of two high-dimensional probability distributions. 
\end{abstract}

\section{Introduction}

In a two-sample problem, we test for the equality of two $d$-dimensional probability distributions $F$ and $G$ based on $n$ independent copies ${\bf X}_1,{\bf X}_2,\ldots,{\bf X}_n$ of ${\bf X} \sim F$ and $m$ independent
copies ${\bf Y}_1,{\bf Y}_2,\ldots,{\bf Y}_m$  of ${\bf Y} \sim G$. This problem is well investigated in the literature, and several tests are available for it. In the parametric regime, we often assume $F$ and $G$ to be Gaussian, and test for the equality of their location and/or scale parameters. Several nonparametric tests are also available, especially for $d=1$. While the Wilcoxon-Mann-Whitney test is used for the univariate two-sample location problem, the Wald-Wolfowitz run test, the Kolmogorov-Smirnov test, and the Camer-von-Mises test \citep[see. e.g.,][]{hollander2013nonparametric,gibbons2014nonparametric} are applicable to general two-sample problems. Using the idea of a minimum spanning tree,
\cite{friedman1979multivariate} generalized the Wald-Wolfowitz run test and the Kolmogorov-Smirnov test to higher dimensions.  \cite{baringhaus2004new} proposed a test based on inter-point distances, which can be viewed as a multivariate generalization of the Cramer-von-Mises test through projection averaging.
\cite{szekely2004testing} and \cite{aslan2005new} also used inter-point distances to come up with tests based on  energy statistics. \cite{schilling1986multivariate} and \cite{henze1988multivariate}   developed  multivariate two-sample tests based on nearest-neighbor type coincidences. \cite{rosenbaum2005exact} proposed a distribution-free test based on optimal non-bipartite matching. \cite{gretton2012kernel} used the notion of maximum mean discrepancy (MMD) to construct a multivariate two-sample test based on kernel mean embedding of two probability distributions. 
These multivariate two-sample tests are consistent in the classical asymptotic regime. For any fixed dimension $d$, the powers of these tests converge to unity as the sample sizes $n$ and $m$ increase. Since these tests are based on pairwise distances among the observations, they can be conveniently used for high-dimensional data even when the dimension is much larger than the combined sample size. But, most of them  often perform poorly in the high dimension, low sample size (HDLSS) situations, especially when the scale difference between $F$ and $G$ dominates their location difference \citep[see, e.g.,][]{biswas2014nonparametric}. 

Following the seminal paper by \cite{hall2005geometric}, the HDLSS regime has received increasing attention. Over the last ten years, several two-sample tests have been proposed for HDLSS data.
\cite{wei2016direction,ghosh2016distribution,srivastava2016raptt} proposed some tests based on linear projections, which are mainly useful for two-sample location problems. \cite{biswas2014nonparametric} and \cite{tsukada2019high} proposed some general two-sample tests based on averages of inter-point distances. 
Under some appropriate assumptions, these two tests turn out to be consistent in both classical and HDLSS asymptotic regimes but nothing is known about their asymptotic behavior when the sample sizes increase with the dimension.
Moreover, they are not robust against outliers generated from heavy-tailed distributions.  \cite{kim2020robust} developed a robust multivariate test based on projection averaging but this test is applicable only when the distances between the observations are measured using the Euclidean metric. 
Some graph-based two-sample tests have also been proposed in the literature.  \cite{mondal2015high} developed a high-dimensional two-sample test based on nearest neighbors. 
\cite{biswas2014distribution} constructed a multivariate run test based on the shortest Hamiltonian path (SHP). \cite{liu2011triangle} proposed a test based on triangles formed by the observations. Under some assumptions on the underlying distributions, these tests have consistency in the HDLSS asymptotic regime. But in the classical asymptotic regime, these graph-based tests usually have poor powers against local alternatives \citep[see, e.g.,][]{bhattacharya2019general}.
Even the large sample consistency of the SHP-based run test and the triangle test is yet to be proved. Also, it is not known how these tests perform when the dimension and the sample sizes grow simultaneously. 
This type of asymptotic behavior has been studied for some  two-sample tests for location \citep[see, e.g.,][]{bai1996effect,chen2010two,srivastava2013two} and scale \citep[see,e.g.,][]{li2012two,cai2013two,ishii2019equality,zheng2020testing}, but for the general two-sample test, the literature is scarce. 

In this article, we propose some two-sample tests based on ball divergence and study their high dimensional behavior not only under HDLSS setup but also in situations, where the dimension and sample sizes grow simultaneously. In the process, we also establish the minimax rate optimality of the proposed tests over a certain class of alternatives associated with ball divergence. Extensive simulation studies are carried out to compare the performance of our tests with some state-of-the-art methods. 

The article is organized as follows.
In Section 2, we construct a test statistic using $\ell_2$ distances among the observations and propose a test based on the permutation principle. The large-sample consistency of this permutation test has also been proved. In section 3, we study the performance of this test in the HDLSS regime. We observe that this test based on the $\ell_2$ distance may fail to discriminate between two high dimensional distributions differing outside the first two moments. 
To take care of this problem, we propose tests based on other appropriate distance functions and prove their high dimensional consistency for a broader class of alternatives. In Section 4, we establish the minimax rate optimality of the proposed tests and investigate their asymptotic behavior when the sample sizes also increase with the dimension. In this setup, we prove the consistency of our tests under shrinking alternatives. Some simulated and real data sets are analyzed in Section 5 to evaluate the empirical performance of our tests. Finally, Section 6 contains a brief summary of the work, and ends with a discussion on some possible directions for future research. All proofs and mathematical details are deferred to the Appendix.


\section{The proposed test based on ball divergence}
\label{The proposed test}


Let ${\bf X} \sim F$ and ${\bf Y}\sim G$ be two $d$-dimensional random variables taking values on a separable metric space $(\mathbb{R}^d,\rho)$, for $\rho$ being the metric on ${\mathbb R}^d$. We know that $F$ and $G$ differ if and only if there exists a ball ${\mathbb B}({\bf u},\epsilon): = \{{\bf v}\in \mathbb{R}^d\mid \rho({\bf v},{\bf u})\leq\epsilon\}$ such that $F({\mathbb B}({\bf u},\epsilon))\not = G({\mathbb B}({\bf u},\epsilon))$. Therefore, for any ${\bf u}\in {\mathbb R}^d$ and $\epsilon>0$, $|F({\mathbb B}({\bf u},\epsilon)) - G({\mathbb B}({\bf u},\epsilon))|$ gives a measure of difference between $F$ and $G$ in a neighborhood of ${\bf u}$. So, when we have two sets $\mathcal{X} = \{{\bf X}_1,{\bf X}_2,\ldots,{\bf X}_n\}$ and $\mathcal{Y} = \{{\bf Y}_1,{\bf Y}_2,\ldots,{\bf Y}_m\}$ of independent realizations of ${\bf X}$ and ${\bf Y}$, respectively, we can choose ${\bf U}_i$ and ${\bf U}_j$ ($i=1,2,\ldots,N$) from the pooled sample $\mathcal{U} = \{{\bf U}_1={\bf X}_1,\ldots,{\bf U}_n={\bf X}_n,{\bf U}_{n+1}={\bf Y}_1,\ldots, {\bf U}_{N}={\bf Y}_m\}=\mathcal{X}\cup\mathcal{Y}$ of size $N=m+n$ to construct the balls ${\mathbb B}_{ij}: = {\mathbb B}({\bf U}_i,\rho({\bf U}_j,{\bf U}_i))$ and compute the differences $D_{ij}=\big|\hat{F}_{ij}({\mathbb B}_{ij})-\hat{G}_{ij}({\mathbb B}_{ij})\big|$. Here $\hat{F}_{ij}$ and $\hat{G}_{ij}$ are the empirical analogs of $F$ and $G$, obtained from ${\mathcal U}$ after deleting the observations ${\bf U}_i$ and ${\bf U}_j$ from the respective samples. One can use these differences to construct a statistic 
$T= {[N(N-1)]}^{-1} \sum_{i \neq j} D_{ij}^2$ and reject the null hypothesis $H_0:F=G$ for higher values of it. However, to reduce the computing cost, here we consider only those cases, where ${\bf U}_i$ and ${\bf U}_j$ come from the same distribution. The resulting test statistic can be expressed as
\begin{align*}
T_{n,m}^\rho :=& \frac{1}{n(n-1)}\sum_{1\leq i\not= j\leq n}\Big\{\frac{1}{n-2}\sum_{k=1, k\not= i,j}^n\delta({\bf X}_k,{\bf X}_j,{\bf X}_i)-\frac{1}{m}\sum_{k=1}^{m}\delta({\bf Y}_{k},{\bf X}_j,{\bf X}_i)\Big\}^2\\
        &+ \frac{1}{m(m-1)}\sum_{1\leq i\not= j\leq m}\Big\{\frac{1}{n}\sum_{k=1}^n\delta({\bf X}_k,{\bf Y}_{j},{\bf Y}_{i})-\frac{1}{m-2}\sum_{k=1, k\not = i,j}^{m}\delta({\bf Y}_{k},{\bf Y}_{j},{\bf Y}_{i})\Big\}^2,
\end{align*}
where $\delta({\bf s},{\bf u},{\bf v}) = \mathbbm{1}\{\rho({\bf s},{\bf v})\leq \rho({\bf u},{\bf v})\}$, and $\mathbbm{1}\{\cdot\}$ is the indicator function. \cite{pan2018ball} proposed a similar test statistic, where they also considered the case $i=j$ (where ${\mathbb B}_{ij}$ has radius $0$). Moreover, ${\bf U}_i$ and ${\bf U}_j$ were not removed from ${\cal U}$ for computing  ${\hat F}({\mathbb B}_{ij})$ and
${\hat G}({\mathbb B}_{ij})$. One can show that $T_{n,m}^{\rho}$ is a consistent estimator (follows from Lemmas A.1 and A.2) of 
$$\Theta_\rho^2(F,G)= \int \int \{F({\mathbb B}({\bf u},\rho({\bf v},{\bf u}))-G({\mathbb B}({\bf u},\rho({\bf v},{\bf u}))\}^2 [dF({\bf u})dF({\bf v})+dG({\bf u})dG({\bf v})],$$ a measure of ball divergence between $F$ and $G$ defined in \cite{pan2018ball}.
Clearly, a large value of $T_{n,m}^{\rho}$ gives an evidence against $H_0:F=G$, and we reject $H_0$ when $T_{n,m}^{\rho}$ exceeds the critical value. For a given level of significance $\alpha~ (0<\alpha<1)$, this critical value (cut-off) is computed using the permutation method, which is described below. 

\begin{itemize}
    \item Consider a permutation $\pi$ of $\{1,2,\ldots, N\}$ and define $\mathcal{U}_\pi = \{{\bf U}_{\pi(1)},{\bf U}_{\pi(2)},\ldots, {\bf U}_{\pi(N)}\}$.
    
    \item Use $\mathcal{X}_{n,\pi} = \{{\bf U}_{\pi(1)},{\bf U}_{\pi(2)},\ldots, {\bf U}_{\pi(n)}\}$ and $\mathcal{Y}_{m,\pi} = \{{\bf U}_{\pi(n+1)},{\bf U}_{\pi(n+2)},\ldots,{\bf U}_{\pi(n+m)}\}$ as the two samples to calculate $T_{n,m,\pi}^\rho$, the permutation analog of $T_{n,m}^\rho$.
    
    \item Repeat this method for all possible permutations. If $\mathcal{S}_N$ denotes the set of all permutations of $\{1,2,\ldots,N\}$, the critical value is given by
    $$c_{1-\alpha} = \inf\{t\in\R: \frac{1}{N!}\sum_{\pi\in \mathcal{S}_N}\mathbbm{1}[T_{n,m,\pi}^\rho\leq t]\geq 1-\alpha\}.$$
  \end{itemize}
We reject $H_0$ if $T_{n,m}^\rho$ is larger than $c_{1-\alpha}$ or the corresponding p-value $\frac{1}{N!}\sum_{\pi\in \mathcal{S}_N}\mathbbm{1}[T_{n,m,\pi}^\rho \ge T_{n,m}^{\rho}]< \alpha$. The following lemma shows that $c_{1-\alpha}$  can be upper bounded by a function of $n$ and $m$ that converges to zero as both $n$ and $m$ diverge to infinity. 

\vspace{-0.04in}
\begin{lemma}
Let $\mathcal{X} = \{{\bf X}_1,{\bf X}_2,\ldots,{\bf X}_n\}$ and $\mathcal{Y} = \{{\bf Y}_1,{\bf Y}_2,\ldots, {\bf Y}_m\}$ be two sets of independent random vectors from two $d$-dimensional distributions $F$ and $G$, respectively. For any $\alpha$ $(0<\alpha<1)$, the inequality
$$
\vspace{-0.05in}
0<c_{1-\alpha}\leq \frac{2}{3\alpha(\min\{n,m\}-2)}
\vspace{-0.05in}$$
holds with probability one.
\label{permutation-cut-off}
\end{lemma}

It is interesting to observe that this upper bound of $c_{1-\alpha}$ does not depend on $d$. So, irrespective of the dimension of the data, the cut-off is of the order $O_P(1/(\min\{n,m\}))$, and it converges to $0$ as $\min\{n,m\}$ diverges to infinity. Therefore, under the alternative hypothesis $H_1:F \neq G$, if $T_{m,n}^\rho$ converges to a positive constant, the power of the test converges to one. This large sample consistency of  the permutation test is asserted by the following theorem.

\begin{thm}
If $\Theta^2_{\rho}(F,G)>0$, for any given $\alpha$ $(0<\alpha<1)$, the power of the level $\alpha$ test based on $T_{m,n}^\rho$
converges to $1$ as $\min\{n,m\}$ increases to infinity.
\label{constlargesam}
\end{thm}

Note that if $(R^d,\rho)$ is a finite dimensional separable metric space, $\Theta^2_{\rho}(F,G)=0$ if and if only $F=G$. 
So, under $H_1:F \neq G$, we have 
$\Theta^2_{\rho}(F,G)>0$. 
For computing $c_{1-\alpha}$, instead of considering all $N!$ permutations of $\{1,2,\ldots,N\}$, it is enough to consider all possible subsets of ${\mathcal U}$ of size $n$. But, if $n$ and $m$ are moderately large, it may not computationally feasible to consider all $\binom{N}{n}$ subsets or all $N!$ permutations. In such cases, we  generate $B$ random permutations $\pi_1,\pi_2,\ldots,\pi_B$ and reject $H_0$ if  the p-value of the test
$$p_{n,m,B} = \frac{1}{B+1}\bigg\{\sum_{i=1}^B\mathbbm{1}[T_{n,m,\pi_i}^\rho\geq T_{n,m}^\rho]+1\bigg\}$$
is smaller than the significance level $\alpha$. Notice that the use of all $N!$ permutations leads to the p-value
$$p_{n,m} = \frac{1}{N!}\bigg\{\sum_{\pi\in \mathcal{S}_N}\mathbbm{1}[T_{n,m,\pi_i}^\rho\geq T_{n,m}^\rho]\bigg\}.$$
As $B$ increases, the difference between $p_{n,m, B}$ and $p_{n,m}$ converges to $0$ almost surely
(see Lemma~\ref{randper}). 
This justifies the implementation of the test based on random permutations. 

\begin{lemma}
\label{randper}
Given the pooled sample ${\cal U}$, $|p_{n,m, B} - p_{n,m}| \stackrel{a.s.}{\rightarrow}0$  as $B$ grows to infinity.
\end{lemma}

Though \cite{pan2018ball} also suggested implementing their test using the permutation method, they proved the consistency of their test based on the large sample distribution of the test statistic. The large sample consistency of the permutation test was missing. Moreover, they did not investigate the high-dimensional behavior of their test. 
Recently \cite{kim2022minimax} provided a unified framework for analyzing the minimax rate optimality of the permutation test. Their emphasis was on tests based on U-statistic in different scenarios. Though $T_{n,m}^{\rho}$ is not a U-statistic, 
it can be shown that the permutation test based on $T_{n,m}^{\rho}$ is also minimax rate optimal over a certain class of alternatives. We discuss this in Section \ref{High Dimensional Performance under Shrinking Alternatives}, where we investigate the asymptotic behavior of our test when the sample sizes increase with the dimension. 
Before that, in the next section, we study its behavior in the HDLSS asymptotic regime. 

\section{Behavior of the proposed test in the HDLSS setup}
\label{High Dimensional Performance}
In this section, we study the high dimensional behavior of the test when the dimension of the data $d$ grows to infinity while the sample sizes $n$ and $m$ remain fixed. The behavior of the test may depend on the metric $\rho$. Since the $\ell_2$ distance (i.e., the Euclidean metric) is arguably the most popular choice as 
the distance function on ${\mathbb R}^d$, we first consider the test based on the $\ell_2$ distance. 

\subsection{Test based on the $\ell_2$ distance}
\label{HDLSS Asymptotics}
For ${\bf X}_1,{\bf X}_2,\ldots,{\bf X}_n \stackrel{iid}{\sim}F$ and ${\bf Y}_1,{\bf Y}_2,\ldots,{\bf Y}_m \stackrel{iid}{\sim}G$, the test statistic based on the $\ell_2$ distance can be expressed as
\begingroup
\addtolength{\jot}{-1em}
\begin{align*}
T_{n,m}^{\ell_2}= &\frac{1}{n(n-1)}\sum_{1 \le i \neq j \le n}  \hspace{-0.1in}\Big(\frac{1}{n-2}\hspace{-0.05in}\sum_{k (\neq i,j)=1}^{n} \hspace{-0.1in}{\mathbbm{1} }\{\|{\bf X}_k-{\bf X}_i\|\le \|{\bf X}_j-{\bf X}_i\|\} \\ 
& \hspace{2in}-\frac{1}{m}\sum_{\ell=1}^{m} {\mathbbm{1} }\{\|{\bf Y}_{\ell}-{\bf X}_i\|\le \|{\bf X}_j-{\bf X}_i\|\}\Big)^2\\ 
&+\frac{1}{m(m-1)}\sum_{1 \le i \neq j \le m}  \Big(\frac{1}{n}\sum_{k=1}^{n} {\mathbbm{1}}\{\|{\bf X}_k-{\bf Y}_i\|\le \|{\bf Y}_j-{\bf Y}_i\|\}\\
& \hspace{2in}-\frac{1}{m-2}\hspace{-0.05in}\sum_{\ell(\neq i,j)=1}^{m} \hspace{-0.1in} {\mathbbm{1}}\{\|{\bf Y}_{\ell}-{\bf Y}_i\|\le \|{\bf Y}_j-{\bf Y}_i\|\}\Big)^2,
\end{align*}
\endgroup
where  $\|\cdot\|$ denotes the Euclidean norm on ${\mathbb R}^d$. For investigating the high dimensional behavior of the test based on $T_{n,m}^{\ell_2}$, here we consider the following assumptions.
\begin{itemize}
    \item[(A1)] For ${\bf Z} = (Z^{(1)},\ldots,Z^{(d)}) \sim F$ or $G$, fourth moments of $Z^{(q)}$s are
    uniformly bounded.
    \vspace{0.05in}
    \item[(A2)] If ${\bf X}_1,{\bf X}_2 \stackrel{iid}{\sim} F$ and ${\bf Y}_1,{\bf Y}_2 \stackrel{iid}{\sim} G$ are independent, for ${\bf W} = {\bf X}_1-{\bf X}_2, {\bf X}_1-{\bf Y}_1$ and ${\bf Y}_1-{\bf Y}_2$, the sum
    $\sum_{1\leq q\not= q'\leq d}Cov\{(W^{(q)})^2, (W^{(q')})^2\}$ is of order $o(d^2)$.
    \vspace{0.05in}
    \item[(A3)] There exist constants $\nu^2$, $\sigma_F^2$ and $\sigma_G^2$ such that for two independent random vectors ${\bf X}\sim F$ and ${\bf Y}\sim G$, $d^{-1}\|E({\bf X})-E({\bf Y})\|^2$, $d^{-1}\sum_{q=1}^d Var(X^{(q)})$ and $d^{-1}\sum_{q=1}^d Var(Y^{(q)})$ converge to $\nu^2$, $\sigma_F^2$ and $\sigma_G^2$, respectively, as $d$ increases to infinity.
\end{itemize}
Assumptions (A1)-(A3) are quite common in the HDLSS literature. \cite{hall2005geometric} considered similar assumptions to investigate the high dimensional behavior of some popular classifiers. Similar assumptions were also made by \cite{ahn2010maximal}, \cite{biswas2014distribution} and \cite{sarkar2019perfect} in the context of high dimensional classification, hypothesis testing, and cluster analysis, respectively.  
Assumptions (A1) and (A2) ensure that the weak law of large numbers (WLLN) holds for the sequence $\{({\bf W}^{(q)})^2:~q\ge1\}$, i.e., $d^{-1}|\|{\bf W}\|^2-E\|{\bf W}\|^2| \stackrel{P}{\longrightarrow} 0$  as $d \rightarrow \infty$. Again, depending on whether ${\bf W}={\bf X}_1-{\bf X}_2, {\bf Y}_1-{\bf Y}_2$ or ${\bf X}_1-{\bf Y}_1$ the limiting value of $d^{-1}E\|{\bf W}\|^2$ and hence that of $d^{-1}\|{\bf W}\|^2$ can be obtained under Assumption (A3). As a result, we have the  following lemma.

\begin{lemma}
Suppose that ${\bf X}_1,{\bf X}_2\stackrel{iid}{\sim} F$ and ${\bf Y}_1,{\bf Y}_2\stackrel{iid}{\sim} G$ are independent. If $F$ and $G$ satisfy Assumptions (A1)-(A3), then as $d$ grows to infinity, $d^{-1/2}\|{\bf X}_1-{\bf X}_2\| \stackrel{P}{\rightarrow} \sigma_F\sqrt{2}$, $d^{-1/2}\|{\bf Y}_1-{\bf Y}_2\| \stackrel{P}{\rightarrow} \sigma_G\sqrt{2}$ and
$d^{-1/2}\|{\bf X}_1-{\bf Y}_1\| \stackrel{P}{\rightarrow}\sqrt{\sigma_G^2+\sigma_F^2+\nu^2}$.
\label{inter-point-dist-accu}
\end{lemma}

Lemma~\ref{inter-point-dist-accu} helps us to study the high dimensional behavior of the indicators in $T_{n,m}^{\ell_2}$ and that of the test statistic $T_{n,m}^{\ell_2}$ as a whole. 
Lemma~\ref{more-than-1/3} uses these results to show that if $\nu^2>0$ or $\sigma^2_F\neq \sigma^2_G$,  $T_{n,m}^{\ell_2}$ takes values higher than $1/3$ with probability tending to $1$ as the dimension grows to infinity. 

\begin{lemma}
Assume that the two distributions $F$ and $G$ satisfy Assumptions (A1)-(A3). If $\nu^2+(\sigma_F-\sigma_G)^2 > 0$, we have $\lim_{d\to\infty}P\{T_{n,m}^{\ell_2}> 1/3\} = 1$.
\label{more-than-1/3}
\end{lemma}

In Lemma~\ref{permutation-cut-off}, we have already seen that the critical value of the permutation test $c_{1-\alpha}$ is smaller than $\frac{2}{3\alpha} \big(\frac{1}{\min\{n,m\}-2}\big)$ with probability one. So, the resulting test has the high-dimensional consistency if $\min\{n,m\} -2 > 2/\alpha$. This result is  stated as Theorem~\ref{HDLSS-L2} below.

\begin{thm}
Assume that $F$ and $G$ satisfy Assumptions (A1)-(A3). If $\nu^2+(\sigma_F-\sigma_G)^2> 0$ and $\min\{n, m\} \geq 2+2/\alpha$, the power of the level $\alpha$  $(0<\alpha<1)$ test based on $T_{n,m}^{\ell_2}$ increases to one as $d$ increases to infinity.
\label{HDLSS-L2}
\end{thm}

This theorem gives a sufficient condition for the high dimensional consistency of the proposed test.
If two distributions $F$ and $G$ differ in their locations ($\nu^2>0$) and/or scales ($\sigma_F^2\not=\sigma_G^2$), the test based on $T_{n,m}^{\ell_2}$ turns out to be consistent in the HDLSS regime.

Now, we consider three examples involving multivariate normal distributions to study the empirical performance of the proposed test for high-dimensional data.  

\vspace{0.075in}
\textbf{Example 1}  Two distributions $F = \mathcal{N}_{d}({\bf 0}_d,{\bf I}_d)$ and $G = \mathcal{N}_d(0.15\,{\bf 1}_d, {\bf I}_d)$ have different means but the same dispersion matrix. Here ${\bf 0}_d=(0,0,\ldots,0)^{\top}$ and ${\bf 1}_d=(1,1,\ldots,1)^{\top}$ are $d$-dimensional vectors with all elements equal to $0$ and $1$, respectively, ${\bf I}_d$ denotes the $d \times d$ identity matrix, and ${\mathcal N}_d(\muvec,\sigmat)$ stands for the $d$-variate normal distribution with the mean vector $\muvec$ and the dispersion matrix $\sigmat$.

\vspace{0.075in}
\textbf{Example 2}  Two distributions $F = \mathcal{N}_{d}({\bf 0}_d,{\bf I}_d)$  and $G = \mathcal{N}_d({\bf 0}_d,1.1 {\bf I}_d)$ have the same location but they differ in their scales. 

\vspace{0.075in}
\textbf{Example 3}  Both $F = \mathcal{N}_{d}({\bf 0}_d,\sigmat_{1,d})$ and $G = \mathcal{N}_d({\bf 0}_d,\sigmat_{2,d})$ have the same mean  and diagonal dispersion matrices. The first $d/2$ diagonal elements of $\sigmat_{1,d}$ are $1$ and the rest are $2$. On the contrary, $\sigmat_{2,d}$ has the first $d/2$ diagonal elements equal to $2$ and the rest equal to $1$.       

\vspace{0.075in}
For each example, we considered $10$ different choices of $d$ ($d=2^i$ for $i=1,\ldots,10$). In each case, we used the test based on $100$ observations  ($50$ from each distribution). This process was repeated 500 times to estimate the power of the test by the proportion of times it rejected $H_0$. The results are reported in Figure~\ref{fig:HDLSS-l2-illustration}.

\begin{figure}[!h]
    \centering
    \begin{tikzpicture}
    \begin{axis}[xmin = 1, xmax = 10, xlabel = {$\log_2(d)$}, ylabel = {Power Estimates}, title = {\bf Example 1}]

    \addplot[color = red, mark = *, step = 1cm,very thin]coordinates{(1,0.098)(2,0.09)(3,0.088)(4,0.116)(5,0.128)(6,0.154)(7,0.202)(8,0.356)(9,0.696)(10,0.99)};

    \end{axis}
    \end{tikzpicture}
    \begin{tikzpicture}
    \begin{axis}[xmin = 1, xmax = 10, xlabel = {$\log_2(d)$}, ylabel = {Power Estimates}, title = {\bf Example 2}]

    \addplot[color = red,   mark = *, step = 1cm,very thin]coordinates{(1,0.062)(2,0.092)(3,0.122)(4,0.23)(5,0.436)(6,0.736)(7,0.96)(8,1)(9,1)(10,1)};

    \end{axis}
    \end{tikzpicture}
    \begin{tikzpicture}
     \begin{axis}[xmin = 1, xmax = 10, ymin = 0, ymax = 1, xlabel = {$\log_2(d)$}, ylabel = {Power Estimates}, title = {\bf Example 3}]

     \addplot[color = red,   mark = *, step = 1cm,very thin]coordinates{(1,0.074)(2,0.072)(3,0.034)(4,0.042)(5,0.034)(6,0.042)(7,0.044)(8,0.068)(9,0.04)(10,0.06)};

    \end{axis}
    \end{tikzpicture}

    \caption{Power of the permutation test based on $T_{n,m}^{\ell_2}$ in Examples 1-3.}
    \label{fig:HDLSS-l2-illustration}
\end{figure}
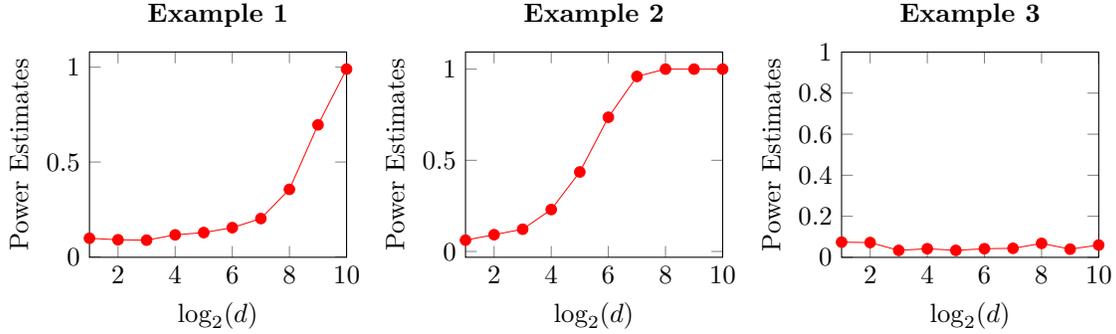

In Examples 1 and 2, we have $\nu^2+(\sigma_F-\sigma_G)^2>0$. So, as one would expect in view of Theorem~\ref{HDLSS-L2}, in these two examples, the power of the proposed test increased with the dimension. However,
in Example 3, we have $\nu^2=0$ and $\sigma_F=\sigma_G$. So, the sufficient conditions for consistency (see Theorem \ref{HDLSS-L2}) do not hold. In this example, the proposed test had a poor performance. To overcome this limitation of the test based on the $\ell_2$ distance, in the next subsection, we use different distance functions to construct the test statistic and study the high dimensional performance of the resulting tests.

\subsection{Tests based on generalized distances}
\label{HDLSS Asymptotics Under Generalized Distances}

Instead of using the $\ell_2$ distance, here we consider the generalized distance function proposed by \cite{sarkar2018some}. The generalized distance between two $d$-dimensional observations ${\bf x}=(x^{(1)},\ldots,x^{(d)})$ and ${\bf y}=(y^{(1)},\ldots,y^{(d)})$ is given by 
$\varphi_{h,\psi}({\bf x}, {\bf y}) = h\big\{\frac{1}{d}\sum_{q = 1}^d \psi(|x^{(q)}-y^{(q)}|^2)\big\},$
where $h:\R_+\to\R_+$ and $\psi:\R_+\to\R_+$ are continuous, monotonically increasing functions with $h(0) = \psi(0) = 0$. Note that all $\ell_p$ distances ($p \ge 1$) are special cases  of $\varphi_{h,\psi}$ (up to a multiplicative constant). We replace the $\ell_2$ metric by $\varphi_{h,\psi}$ to get the 
generalized test statistic 

\begingroup
\addtolength{\jot}{-1em}
{ \begin{align*}
T^{h,\psi}_{n,m}&=\frac{1}{n(n-1)}\sum_{1 \le i \neq j \le n}  \hspace{-0.1in}\Big(\frac{1}{n-2}\hspace{-0.05in}\sum_{k (\neq i,j)=1}^{n} \hspace{-0.1in}{\mathbbm{1} }\{\varphi_{h,\phi}({\bf X}_k,{\bf X}_i)\le \varphi_{h,\phi}({\bf X}_j,{\bf X}_i)\}\\
& \hspace{2in}-\frac{1}{m}\sum_{\ell=1}^{m} {\mathbbm{1} }\{\varphi_{h,\phi}({\bf Y}_{\ell},{\bf X}_i)\le \varphi_{h,\phi}({\bf X}_j,{\bf X}_i)\}\Big)^2\\
&+\frac{1}{m(m-1)}\sum_{1 \le i \neq j \le m}  \Big(\frac{1}{n}\sum_{k=1}^{n} {\mathbbm{1}}\{\varphi_{h,\phi}({\bf X}_k,{\bf Y}_i)\le \varphi_{h,\phi}({\bf Y}_j,{\bf Y}_i)\}\\
& \hspace{2in}-\frac{1}{m-2}\hspace{-0.05in}\sum_{\ell(\neq i,j)=1}^{m} \hspace{-0.1in} {\mathbbm{1}}\{\varphi_{h,\phi}({\bf Y}_{\ell},{\bf Y}_i)\le \varphi_{h,\phi}({\bf Y}_j,{\bf Y}_i)\}\Big)^2.
\end{align*}}
\endgroup

\noindent
We reject the null hypothesis $H_0:F=G$  for large values of $T^{h,\psi}_{n,m}$, where the cut-off is chosen using the  permutation method as before. 

High dimensional behavior of the test based on $T_{n,m}^{h,\psi}$ can be investigated under assumptions similar to (A1)-(A3). Recall that (A1) and (A2) were needed to have WLLN for the sequence of possibly dependent and non-identically distributed random variables $\{(W^{(q)})^2:~q\ge 1\}$ for ${\bf W}= {\bf X}_1-{\bf X}_2, {\bf X}_1-{\bf Y}_1$ and ${\bf Y}_1-{\bf Y}_2$, while (A3) gave us information about the limiting value of $d^{-1}\sum_{q=1}^d E((W^{(q)})^2)$.
The following assumption takes care of these issues for the sequence of random variables $\{\psi(W^{(q)})^2:~q\ge 1\}$.

\begin{itemize}
    \item[(A4)] If ${\bf X}_1,{\bf X}_2 \stackrel{iid}{\sim} F$ and ${\bf Y}_1,{\bf Y}_2\stackrel{iid}{\sim} G$ are independent, for $W = {\bf X}_1-{\bf X}_2$, ${\bf X}_1-{\bf Y}_1$ and ${\bf Y}_1-{\bf Y}_2$, ($i$) $\limsup_{d\to\infty}d^{-1}\sum_{q=1}^d E\psi(|W^{(q)}|^2)<\infty$ and ($ii$) $d^{-1}\sum_{q=1}^d \{\psi(|W^{(q)}|^2)-E\psi(|W^{(q)}|^2)\}$ converges to zero in probability as $d$ grows to infinity.  
\end{itemize}

Now, let us define $\varphi_{h,\psi}^*(F,F) = h\{d^{-1}\sum_{q=1}^d  E\psi(|X_1^{(q)}-X_2^{(q)}|^2)\}$, $\varphi_{h,\psi}^*(G,G) = h\{d^{-1}\sum_{q=1}^d E\psi(|Y_1^{(q)}-Y_2^{(q)}|^2)\}$ and $\varphi_{h,\psi}^*(F,G) = h\{d^{-1}\sum_{q=1}^d E\psi(| X_1^{(q)}-Y_1^{(q)}|^2)\}$. There is an interesting lemma due to \cite{sarkar2018some} (Lemma 1) involving these three quantities. The lemma is stated below.

\begin{lemma}
Suppose that $h$ is concave and $\psi$ has a non-constant completely monotone derivative. Then for any fixed $d$, we have $e_{h,\psi}(F,G) = 2\varphi_{h,\psi}^*(F,G)-\varphi_{h,\psi}^*(F,F)-\varphi_{h,\psi}^*(G,G)\geq 0,$ where the equality holds if and only if $F$ and $G$ have the same one-dimensional marginal distributions.
\label{part-wise-gdist}
\end{lemma}

Note that $e_{h,\psi}(F,G)$  
can be viewed as an energy distance \citep[see, e.g.,][]{aslan2005new} between $F$ and $G$. 
In view of Lemma~\ref{part-wise-gdist}, for appropriate choices of $h$ and $\psi$, it is somewhat reasonable to assume that under the alternative $H_1:F\neq G$,  the limiting energy distance between $F$ and $G$ remains bounded away from $0$ (i.e., $\liminf_{d\to\infty} e_{h,\psi}(F,G)>0$). Under this assumption, we can establish the consistency of the test based on $T_{m,n}^{h,\psi}$ in the HDLSS asymptotic regime. This result is given by the following theorem. 

\begin{thm}
Assume that $F$ and $G$ satisfy Assumption (A4) and $\liminf_{d\to\infty} e_{h,\psi}(F,G)$ $>0$. If $\min\{n,m\} \geq 2+2/\alpha$, the power of the level $\alpha$ $(0<\alpha<1)$ based on $T^{h,\psi}_{n,m}$ increases to one as $d$ increases to infinity.
\label{HDLSS-gdist}
\end{thm}

For $h(t)=\sqrt{t}$ and $\psi(t)=t$, $\varphi_{h,\psi}$ turns out to be the $\ell_2$ metric (up to a multiplicative constant), but this choice of $\psi$ does not have a non-constant completely monotone derivative \citep[see, e.g.,][]{baringhaus2010rigid} as mentioned in Lemma~\ref{part-wise-gdist}. But there are several other choices of $\psi$
that satisfy this property. For instance, one can use $\psi_1(t) = \sqrt{t}$, $\psi_2(t) = 1-e^{-t/2}$ or $\psi_3(t) = \log(1+t)$ with $h(t)=t$ in all three cases. Note that for the first two choices of $\psi$, $\varphi_{h,\psi}$ turns
out to be a distance function (the first one leads to a scalar multiple of the $\ell_1$ distance), but that is not the case for $\psi_3$. In that case, we can call it a dissimilarity measure. Note that for the $\ell_2$ distance (i.e., $\psi(t)=t$ and 
$h(t)=\sqrt{t}$), under (A1)-(A3), we have $\liminf_{d\to\infty} e_{h,\psi}(F,G) = 2\sqrt{\nu^2+\sigma_F^2+\sigma_G^2} -\sigma_F\sqrt{2}-\sigma_G\sqrt{2}$, which turns out to be positive if and only if either $\nu^2>0$ and/or $\sigma_F \neq \sigma_G$, i.e. $\nu^2+(\sigma_F-\sigma_G)^2>0$, the condition mentioned in Theorem~\ref{HDLSS-L2}. In Example 3, we have  $\nu^2+(\sigma_F-\sigma_G)^2=0$ but $\liminf_{d\to\infty} e_{h,\psi}(F,G)>0$  for $\psi=\psi_1,\psi_2,\psi_3$ with $h(t)=t$. In this example, while the test based on $T_{n,m}^{\ell_2}$ had poor performance, those based on $T_{n,m}^{h,\psi}$ with other three choices of $\psi$ (henceforth referred to as $T_{n,m}^{\ell_1}$, $T_{n,m}^{\exp}$ and $T_{n,m}^{\log}$, respectively) performed well (see Figure 2). To carry out further investigation on the high dimensional behavior of these tests, we consider two other examples. 

\vspace{0.1in}
\textbf{Example 4}  Both $F$ and $G$ have  independent and identically distributed coordinate variables. While in $F$, they follow the standard Cauchy distribution, in $G$ they have a location shift of one unit. 

\vspace{0.1in}
\textbf{Example 5}  Two normal distributions $F = \mathcal{N}_d(d^{-1/2}\,{\bf 1}_d,{\bf I}_d)$ and $G = \mathcal{N}_d(-d^{-1/2}\,{\bf 1}_d,{\bf I}_d)$ differ only in their means, where the Euclidean distance between the two means remains constant as the dimension varies.

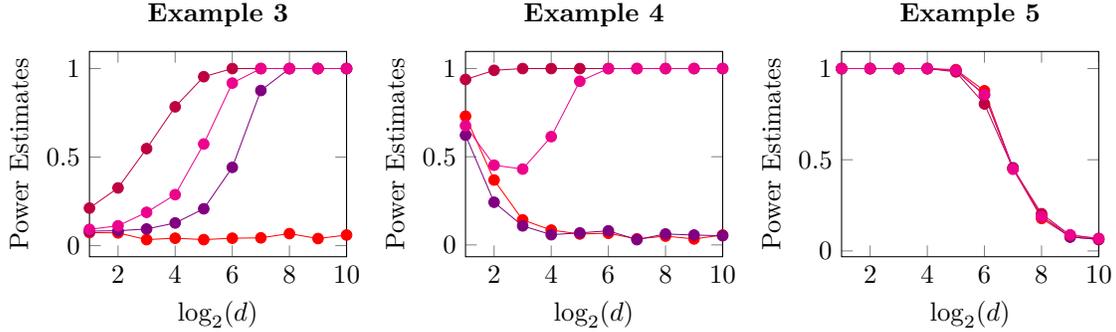
\begin{figure}[!h]
    \centering
    \begin{tikzpicture}
    \begin{axis}[xmin = 1, xmax = 10, xlabel = {$\log_2(d)$}, ylabel = {Power Estimates}, title = {\bf Example 3}]
      \addplot[color = red,   mark = *, step = 1cm,very thin]coordinates{(1,0.074)(2,0.072)(3,0.034)(4,0.042)(5,0.034)(6,0.042)(7,0.044)(8,0.068)(9,0.04)(10,0.06)};

      \addplot[color = violet, mark = *, step = 1cm,very thin]coordinates{(1,0.082)(2,0.084)(3,0.094)(4,0.128)(5,0.208)(6,0.442)(7,0.876)(8,1)(9,1)(10,1)};

      \addplot[color = purple, mark = *, step = 1cm,very thin]coordinates{(1,0.212)(2,0.326)(3,0.548)(4,0.784)(5,0.954)(6,1)(7,1)(8,1)(9,1)(10,1)};

      \addplot[color = magenta, mark = *, step = 1cm,very thin]coordinates{(1,0.092)(2,0.112)(3,0.188)(4,0.288)(5,0.574)(6,0.918)(7,1)(8,1)(9,1)(10,1)};

    \end{axis}
    \end{tikzpicture}
    \begin{tikzpicture}
    \begin{axis}[xmin = 1, xmax = 10, xlabel = {$\log_2(d)$}, ylabel = {Power Estimates}, title = {\bf Example 4}]
\addplot[color = red, mark = *, step = 1cm,very thin]coordinates{(1,0.73)(2,0.368)(3,0.142)(4,0.084)(5,0.062)(6,0.066)(7,0.034)(8,0.05)(9,0.034)(10,0.056)};

\addplot[color = violet, mark = *, step = 1cm,very thin]coordinates{(1,0.622)(2,0.242)(3,0.108)(4,0.058)(5,0.068)(6,0.08)(7,0.03)(8,0.062)(9,0.056)(10,0.052)};

\addplot[color = purple, mark = *, step = 1cm,very thin]coordinates{(1,0.938)(2,0.99)(3,1)(4,1)(5,1)(6,1)(7,1)(8,1)(9,1)(10,1)};

\addplot[color = magenta, mark = *, step = 1cm,very thin]coordinates{(1,0.676)(2,0.452)(3,0.43)(4,0.614)(5,0.928)(6,1)(7,1)(8,1)(9,1)(10,1)};

    \end{axis}
    \end{tikzpicture}
    \begin{tikzpicture}
     \begin{axis}[xmin = 1, xmax = 10, xlabel = {$\log_2(d)$}, ylabel = {Power Estimates}, title = {\bf Example 5}]
\addplot[color = red, mark = *, step = 1cm,very thin]coordinates{(1,1)(2,1)(3,1)(4,1)(5,0.994)(6,0.878)(7,0.452)(8,0.178)(9,0.078)(10,0.062)};

\addplot[color = violet, mark = *, step = 1cm,very thin]coordinates{(1,1)(2,1)(3,1)(4,1)(5,0.988)(6,0.856)(7,0.456)(8,0.19)(9,0.076)(10,0.064)};

\addplot[color = purple, mark = *, step = 1cm,very thin]coordinates{(1,1)(2,1)(3,1)(4,1)(5,0.982)(6,0.806)(7,0.452)(8,0.204)(9,0.086)(10,0.068)};

\addplot[color = magenta, mark = *, step = 1cm,very thin]coordinates{(1,1)(2,1)(3,1)(4,1)(5,0.99)(6,0.857)(7,0.448)(8,0.188)(9,0.088)(10,0.068)};

    \end{axis}
    \end{tikzpicture}

    \caption{Power of BD-$\ell_2$ (\textcolor{red}{$\tikzcircle{2pt}$}),  BD-$\ell_1$ ({$\tikzcirclev{2pt}$}), BD-$\exp$ ({$\tikzcirclep{2pt}$}) and BD-$\log$ ({$\tikzcirclem{2pt}$}) tests in Examples 3-5.}
    \label{fig:HDLSS-gdist-illustration}
\end{figure}

Here also, we generated  50 observations from each distribution, and the process was repeated 500 times to estimate the power of the tests. The results are reported in Figure \ref{fig:HDLSS-gdist-illustration}.

In Example 4, the ball divergence tests based on $T_{n,m}^{\ell_1}$ and $T_{n,m}^{\ell_2}$ (henceforth referred to as BD-$\ell_2$ and BD-$\ell_1$ tests) did not work well, but those based on $T_{n,m}^{\exp}$ and $T_{n,m}^{\log}$ (henceforth referred to as BD-$\exp$ and BD-$\log$ tests) had excellent performance. Among them, BD-$\exp$ had an edge. Note that in this example, the coordinate variables in $F$ and $G$ do not have finite moments. That affected the performance of BD-$\ell_1$ and BD-$\ell_2$. Since $\psi_2(t)=1-e^{-t/2}$ is a bounded function, we do not have such problems for BD-$\exp$. Note that Assumption (A4) holds for this choice of $\psi$. It holds for $\psi_3(t)=\log(1+t)$ as well. This was the reason behind the good performance of the tests based on these two choices of $\psi$. 

In Example 5, the Mahalanobis distance between the two distributions remains the same as the dimension increases. This can be viewed as having a constant difference along the first coordinate, while the other coordinate variables serve as noise.
Clearly, the noise increases with the dimension, and we have $\lim_{d \rightarrow \infty} e_{h,\psi}(F, G) = 0$ for all choices of $h$ and $\psi$ considered here. So, as expected, the powers of all tests gradually dropped as the dimension increased. In this situation, if we want the powers of the tests to increase with the dimension, the only possible option is to increase the sample sizes appropriately with the dimension. We consider this type of situation in the next section.

\section{Behavior of the proposed tests when the sample sizes increase with the dimension of the data}
\label{High Dimensional Performance under Shrinking Alternatives}

In this section, we deal with the cases, where the two distributions $F$ and $G$ gradually become close as the dimension increases. We call these types of alternatives as shrinking alternatives. In such cases, the power of any test based on fixed sample sizes is expected to go down as the dimension increases. Therefore, to achieve better performance in high-dimension, one needs to increase the sample sizes with the dimension. Now, one may be curious to know whether it is possible to increase the sample sizes with the dimension at an appropriate rate such that one can construct a valid level $\alpha$ ($0<\alpha<1$) test with power converging to $1$ as the dimension increases. We shall show that this is possible for our tests as long as the ball divergence measure between $F$ and $G$ shirks to $0$
at an appropriately slower rate. This is obtained by establishing the minimax rate optimality of our tests. We  discuss it in the following sub-section.  


\subsection{Minimax rate optimality}
\label{Minimax Rate Optimality}

Consider a testing problem between a pair of hypotheses $H_0':\Theta_{\ell_2}^2(F, G)=0$ and $H_1':\Theta_{\ell_2}^2(F, G)>\epsilon$ for some $\epsilon>0$. Define $\P_{F,G}^{(n,m)}$ as the joint distribution of ${\bf X}_1,{\bf X}_2,\ldots,{\bf X}_n,{\bf Y_1},{\bf Y}_2,\ldots,{\bf Y}_m$, where ${\bf X}_1,{\bf X}_2,\ldots,{\bf X}_n \stackrel{iid}{\sim}F$ and ${\bf Y}_1,{\bf Y}_2,\ldots,{\bf Y}_m \stackrel{iid}{\sim}F$. Let 
$\mathcal{F}(\epsilon):=\{(F, G)\mid \Theta_{\ell_2}^2(F, G)>\epsilon\}$ denote the class of alternatives, and for a given significance level $\alpha\in(0,1)$, $\mathbb{T}_{n,m,d}(\alpha)$ denote the class of all level $\alpha$ test functions $\phi:\mathcal{U}\to \{0,1\}$. The minimax type II error rate for this class is defined as
$$R_{n,m,d}(\epsilon) = \inf_{\phi\in \mathbb{T}_{n,m,d}(\alpha)}\sup_{(F,G)\in \mathcal{F}(\epsilon)} \P^{(n,m)}_{F,G}(\phi=0).$$
Here, we want to find an $\epsilon_0=\epsilon_0(n,m,d)$, for which the following conditions hold.
\begin{enumerate}
    \item[(a)] \label{cond_a} 
    For any $0<\zeta<1-\alpha$, there exists a constant $c(\alpha,\zeta)>0$ such that for all $0 < c < c(\alpha,\zeta)$, we have 
    $\liminf\limits_{n,m,d \rightarrow \infty} R_{n,m,d}(c~\epsilon_0({n,m,d})) \geq \zeta$ .
    
    \item[(b)] \label{cond_b} There exists a level $\alpha$ test $\phi_0$ such that for any $0<\zeta<1-\alpha$, we can find a constant $C(\alpha,\zeta) > 0$ for which $\limsup\limits_{n,m,d \rightarrow \infty} \sup\limits_{(F,G)\in\mathcal{F}(c~\epsilon_0({n,m,d}))}\P_{F,G}^{(n,m)}\{\phi_0=0\}\leq \zeta$ for all $c>C(\alpha,\zeta)$, or in other words, $\limsup\limits_{n,m,d \rightarrow \infty} R_{n,m,d}(c~\epsilon_0({n,m,d})) \leq \zeta$ for all $c > C(\alpha,\zeta)$. 
\end{enumerate}

The rate $\epsilon_0({n,m,d})$ (which is unique up to a constant) is called the minimax rate of separation for this problem, and the test $\phi_0$ is called the minimax rate optimal test. Theorem 4.1 below shows that if $\epsilon$ is of smaller order than $\lambda(n,m):= (1/\sqrt{n}+1/\sqrt{m})^2$, for all level $\alpha$ tests. the maximum type II error rate is bounded away from $0$.

\begin{thm}
For $0<\zeta<1-\alpha$, there exists a constant $c_0(\alpha,\zeta)$ such that for $\lambda({n,m}) = (1/\sqrt{n}+1/\sqrt{m})^2$, the minimax type II error rate $R_{n,m,d}(c\lambda(n,m))$ is lower bounded by $\zeta$ for all $0<c<c_0(\alpha, \zeta)$.
\label{minimax-lower-bound}
\end{thm}

The above result shows that the minimax rate of separation cannot be of order smaller than $O(1/\sqrt{n}+1/\sqrt{m})^2$. We now show that for $\epsilon_0(n,m,d)=\lambda(n,m)$, the permutation test based on $T_{n,m}^{\ell_2}$ satisfies the condition (b) stated above. 

\begin{thm}
For $0<\zeta<1-\alpha$, there exists a constant $C_0(\alpha,\zeta)$ such that asymptotically the maximum type II error of the test based on $T_{n,m}^{\ell_2}$ over $\mathcal{F}(c\lambda(n,m))$  is uniformly bounded above by $\zeta$ for all $c>C_0(\alpha,\zeta)$, i.e.,
\vspace{-0.025in}
$$\limsup\limits_{n,m,d \rightarrow \infty}\sup_{(F,G)\in \mathcal{F}(c\lambda({n,m}))}P^{(n,m)}_{F,G}(T_{n,m}\leq c_{1-\alpha})\leq \zeta~~\mbox{for all } c>C_0(\alpha,\zeta).
\vspace{-0.1in}$$
\label{minimax-upper-bound}
\end{thm}

Theorems 4.1 and 4.2 together show that the minimax rate of separation  $\epsilon_0({n,m,d})=(1/\sqrt{n}+1/\sqrt{m})^2$ does not depend on the dimension, and they also establish the minimax rate optimality of the  permutation test based on $T_{n,m}^{\ell_2}$ for the class of alternatives $\mathcal{F}(\epsilon)$.


\subsection{Performance under shrinking alternatives}
\label{Performance Under Shrinking Alternatives}

Theorem \ref{minimax-upper-bound} gives us a lower bound $\lambda({n,m})$ on the rate of $\Theta_{\ell_2}^2(F,G)$ that enables us to detect the difference between $F$ and $G$ using the permutation test based on $T_{n,m}^{\ell_2}$. If we increase the sample sizes $n$ and $m$ with the dimensions $d$ such that $\lambda({n,m})$ converges to $0$ at a faster rate than $\Theta_{\ell_2}^2(F,G)$ (i.e., $\Theta_{\ell_2}^2(F,G)/\lambda({n,m}) \rightarrow \infty$  as $d \rightarrow \infty$), the test based on $T_{n,m}^{\ell_2}$ turns out to be consistent. This result is asserted by the following theorem. 

\begin{thm}
Suppose that $n$ and $m$, the sample sizes from the two distributions $F$ and $G$, grow as a function of the dimension $d$ in such a way that $\lim_{d\to\infty}\Theta_{\ell_2}^2(F,G)/\lambda({n,m}) = \infty$. Then for any fixed $\alpha$ $(0<\alpha<1)$, the power of the level $\alpha$ test based on $T_{n,m}^{\ell_2}$ converges to one as dimension increases to infinity.
\label{High-Dimension-L2}
\end{thm}

It is easy to see that if $\liminf_{d\to\infty}\Theta_{\ell_2}^2(F,G)>0$, the assumption in Theorem \ref{High-Dimension-L2} holds even if $n$ and $m$ grow very slowly 
with $d$. For such examples, one can expect to get good results even in the HDLSS setup, and we have observed the same in our numerical studies. In Section~\ref{HDLSS Asymptotics Under Generalized Distances}, we have seen that in the HDLSS setup, if $F$ and $G$ satisfy the conditions of Theorem \ref{HDLSS-L2}, we have $\liminf_{d\to\infty}\Theta_{\ell_2}^2(F, G) \geq 1/3$. As expected, in such cases we have the consistency of the test when $m$ and $n$ also increase with the dimension. 

So far, we have discussed the minimax rate of optimality of the test based on $T_{n,m}^{\ell_2}$ and established its consistency for shrinking alternatives. One can show that the results similar to Theorems  \ref{minimax-lower-bound}
-\ref{High-Dimension-L2}
 hold even when the test is constructed based on other distance functions considered in Section 3. We state the result below as Theorem \ref{High-Dimenion-gdist}, but we skip the details of the proof since it is exactly the same as in the case of the test based on the $\ell_2$ distance.

\begin{thm}
Let $h,\psi: {\mathbb R}_+ \rightarrow{\mathbb R}_+$ be continuous and monotonically increasing functions with $h(0)=\psi(0)=0$. Assume that $n$ and $m$, the sample sizes $F$ and $G$, grow as a function of the dimension $d$ in such a way that $\lim_{d\to\infty}\Theta_{\varphi_{h,\psi}}^2(F,G)/\lambda({n,m}) = \infty$. Then for any fixed $\alpha$ $(0<\alpha<1)$, the power of the level $\alpha$ test based on $T_{n,m}^{h,\psi}$ converges to one as the dimension increases to infinity.
\label{High-Dimenion-gdist}
\end{thm}



\begin{rem}
In the HDLSS setup, we need Assumptions (A1)-(A4) for the consistency of the tests. But when $n$ and $m$ grow with $d$, we do not need such assumptions. Also, unlike the HDLSS setup, here we have consistency for the test based on the $\ell_p$-distance for all $p\ge 1$. 
\end{rem}


\begin{rem}
    Theorems 4.3 and 4.4 do not tell us anything about the asymptotic behavior of the tests when $\lim_{d\to\infty}\Theta_{\ell_2}^2(F,G)/\lambda({n,m}) = c$ $($ or $\lim_{d\to\infty}\Theta_{h,\psi}^2(F,G)/\lambda({n,m}) = c)$ for some $c\in(0,\infty)$. However, in such cases, one can show that the asymptotic power of the test has a lower bound  $1 -({C_1 c+C_2})/{\left(c-\frac{1}{3\alpha}\right)^2},$
    where $C_1$ and $C_2$ are two universal constants (see the proof of Theorem~\ref{minimax-upper-bound}).
\end{rem}


Now, consider an example involving two multivariate distributions  $F = \prod_{i=1}^d \mathcal{N}_1(1/d^\beta,1)$ and $G = \prod_{i=1}^d \mathcal{N}_1(-1/d^\beta,1)$, where $\beta$ is a positive constant. Note that as $d$ grows to infinity, here we have $\nu^2+(\sigma_{F}-\sigma_{G})^2=0$ and $\lim_{d\to\infty} e_{h,\psi}(F,G) = 0$ for all $h$ and $\psi$ considered in this article. So, the conditions for the HDLSS consistency of the tests are not satisfied. Now, we study the behavior of the test based on $T_{n,m}^{\ell_2}$ when the sample sizes increase with the dimension at the rate $O(d^{\gamma})$ for some $\gamma>0$. We find out the relation between  $\gamma$ and $\beta$ that leads to the consistency of the test. Our findings are summarized in the following proposition.


\begin{prop}
Let $n$ and $m$ be the sample sizes from $F = \prod_{i=1}^d \mathcal{N}_1(1/d^\beta,1)$ and $G = \prod_{i=1}^d \mathcal{N}_1(-1/d^\beta,1)$, respectively. If $\beta>0$ and $n\asymp m\asymp d^{\gamma}$ for some $\gamma>0$, then, for the ball divergence test based on $T_{n,m}^{\ell_2}$, we have the following results.
\begin{itemize}
    \item[(a)] If $\beta\leq 1/4$, for any $\gamma>0$, the test is  consistent $($since $\Theta^2_{\ell_2}(F,G)>0)$.
    
    \item[(b)] If $1/4<\beta\leq 1/2$, the test is consistent if $\gamma>4\beta-1$.
    
    \item[(c)] If $\beta>1/2$ and $\gamma<2\beta-1$ there exist no level $\alpha$ $(0<\alpha<1)$ tests with asymptotic power more than the nominal level $\alpha$. For $\gamma>2\beta$, we have $\lim_{d\to\infty}\Theta_{\ell_2}^2(F,G)/\rho_{n,m} = \infty$, and hence our proposed test is consistent.
    \label{NSA-prop}
\end{itemize}
\end{prop}


Proposition \ref{NSA-prop}{(c)} says that if $\beta>1/2$, for the consistency of the test, one needs to increase the sample size at a rate faster than $O(d^{2\beta-1})$. Therefore, the HDLSS consistency is not possible in this case. Recall that in Example 5, we have $\beta = 1/2$. So, if we increase the sample sizes at a rate faster than $O(d)$, our test will be consistent. We observed the same in our numerical study.

 For this study, we used 3 different choices of $\beta$ ($\beta=0.2, 0.3$ and $0.5$), and in each case, 7 different choices of $\gamma$ ($\gamma=0, 0.4, 0.5, 0.6,0.9. 1$ and $1.1$) and 10 different values of $d$ ($d=2^i$ for $i=1,\ldots,10$) were considered. We took $n=m=5+ \lfloor d^\gamma \rfloor$ to ensure $n,m \ge 5$, and each experiment was repeated $500$
times to compute the power of the test based on $T_{n,m}^{\ell_2}$. The results are reported in Figure \ref{fig:Example}.
In this example, for higher values of $\beta$, $\Theta_{\ell_2}^2(F, G)$ converges to zero at a faster rate. Therefore, to discriminate between $F$ and $G$, we need to increase the sample sizes at a higher rate as well. Figure \ref{fig:Example} shows that for higher values of $\beta$, the tests corresponding to lower values of $\gamma$ performed poorly. Note that here $\gamma = 0$ represents the HDLSS scenario. We can see that for $\beta = 0.2$,
even for $m=n=6$ (i.e., $\gamma=0$), the power of our test converged to $1$ in high dimensions. This was expected in view of Proposition~\ref{NSA-prop}(a). As expected, the test had higher power for larger values of $\gamma$. For $\beta=0.3$ and $0.5$, it did not work well in the HDLSS setup, but when $m$ and $n$ increased with $d$ at an appropriate rate, the power of the test converged to unity as we would expect in view of Proposition~\ref{NSA-prop}(b)-(c).  

\begin{figure}[h]
    \centering
    \begin{tikzpicture}
    \begin{axis}[xmin = 1, xmax = 10, xlabel = {$\log_2(d)$}, ylabel = {Power Estimates}, title = {\bf $\beta = 0.2$}]
\addplot[color = red, mark = *, step = 1cm,very thin]coordinates{(1,0.848)(2,0.93)(3,0.972)(4,0.986)(5,0.998)(6,1)(7,1)(8,1)(9,1)(10,1)};

\addplot[color = purple, mark = star, step = 1cm,very thin]coordinates{(1,0.848)(2,0.93)(3,0.988)(4,0.998)(5,1)(6,1)(7,1)(8,1)(9,1)(10,1)};

\addplot[color = magenta, mark = diamond*, step = 1cm,very thin]coordinates{(1,0.848)(2,0.956)(3,0.988)(4,1)(5,1)(6,1)(7,1)(8,1)(9,1)(10,1)};

\addplot[color = green, mark = x, step = 1cm,very thin]coordinates{(1,0.848)(2,0.956)(3,0.996)(4,1)(5,1)(6,1)(7,1)(8,1)(9,1)(10,1)};

\addplot[color = blue,  mark = triangle*, step = 1cm,very thin]coordinates{(1,0.818)(2,0.982)(3,1)(4,1)(5,1)(6,1)(7,1)(8,1)(9,1)(10,1)};

\addplot[color = orange, mark = square*, step = 1cm,very thin]coordinates{(1,0.914)(2,0.994)(3,1)(4,1)(5,1)(6,1)(7,1)(8,1)(9,1)(10,1)};

\addplot[color = teal, mark = +, step = 1cm,very thin]coordinates{(1,0.914)(2,0.994)(3,1)(4,1)(5,1)(6,1)(7,1)(8,1)(9,1)(10,1)};

    \end{axis}
    \end{tikzpicture}
    \begin{tikzpicture}
    \begin{axis}[xmin = 1, xmax = 10, xlabel = {$\log_2(d)$}, ylabel = {Power Estimates}, title = {\bf $\beta = 0.3$}]
\addplot[color = red, mark = *, step = 1cm,very thin]coordinates{(1,0.78)(2,0.822)(3,0.838)(4,0.832)(5,0.82)(6,0.8)(7,0.796)(8,0.746)(9,0.666)(10,0.62)};

\addplot[color = purple, mark = star, step = 1cm,very thin]coordinates{(1,0.78)(2,0.822)(3,0.918)(4,0.956)(5,0.982)(6,0.984)(7,0.998)(8,1)(9,1)(10,1)};

\addplot[color = magenta, mark = diamond*, step = 1cm,very thin]coordinates{(1,0.78)(2,0.882)(3,0.918)(4,0.982)(5,0.994)(6,1)(7,1)(8,1)(9,1)(10,1)};

\addplot[color = green, mark = x, step = 1cm,very thin]coordinates{(1,0.78)(2,0.882)(3,0.958)(4,0.992)(5,1)(6,1)(7,1)(8,1)(9,1)(10,1)};

\addplot[color = blue,  mark = triangle*, step = 1cm,very thin]coordinates{(1,0.78)(2,0.944)(3,0.994)(4,1)(5,1)(6,1)(7,1)(8,1)(9,1)(10,1)};

\addplot[color = orange, mark = square*, step = 1cm,very thin]coordinates{(1,0.872)(2,0.968)(3,0.998)(4,1)(5,1)(6,1)(7,1)(8,1)(9,1)(10,1)};

\addplot[color = teal, mark = +, step = 1cm,very thin]coordinates{(1,0.872)(2,0.968)(3,1)(4,1)(5,1)(6,1)(7,1)(8,1)(9,1)(10,1)};

    \end{axis}
    \end{tikzpicture}
    \begin{tikzpicture}
     \begin{axis}[xmin = 1, xmax = 10, xlabel = {$\log_2(d)$}, ylabel = {Power Estimates}, title = {\bf $\beta = 0.5$}]
\addplot[color = red, mark = *, step = 1cm,very thin]coordinates{(1,0.682)(2,0.556)(3,0.348)(4,0.202)(5,0.136)(6,0.09)(7,0.09)(8,0.05)(9,0.054)(10,0.05)};

\addplot[color = purple, mark = star, step = 1cm,very thin]coordinates{(1,0.682)(2,0.556)(3,0.438)(4,0.33)(5,0.206)(6,0.144)(7,0.1)(8,0.062)(9,0.082)(10,0.058)};

\addplot[color = magenta, mark = diamond*, step = 1cm,very thin]coordinates{(1,0.682)(2,0.614)(3,0.438)(4,0.39)(5,0.234)(6,0.2)(7,0.122)(8,0.09)(9,0.078)(10,0.062)};

\addplot[color = green, mark = x, step = 1cm,very thin]coordinates{(1,0.682)(2,0.614)(3,0.526)(4,0.418)(5,0.338)(6,0.268)(7,0.15)(8,0.118)(9,0.094)(10,0.1)};

\addplot[color = blue,  mark = triangle*, step = 1cm,very thin]coordinates{(1,0.682)(2,0.714)(3,0.724)(4,0.728)(5,0.8)(6,0.83)(7,0.836)(8,0.834)(9,0.837)(10,0.836)};

\addplot[color = orange, mark = square*, step = 1cm,very thin]coordinates{(1,0.76)(2,0.768)(3,0.838)(4,0.876)(5,0.942)(6,0.992)(7,0.998)(8,1)(9,1)(10,1)};

\addplot[color = teal, mark = +, step = 1cm,very thin]coordinates{(1,0.76)(2,0.768)(3,0.864)(4,0.96)(5,0.994)(6,1)(7,1)(8,1)(9,1)(10,1)};

    \end{axis}
    \end{tikzpicture}

    \caption{Powers of the BD-$\ell_2$ test for different choice of $\beta$ (namely, $0.2,0.3$ and $0.5$) and $\gamma$ (namely, $0$ (\textcolor{red}{$\tikzcircle{2pt}$}), $0.4$ (\textcolor{purple}{$\star$}), $0.5$ (\textcolor{magenta}{$\blacklozenge$}), $0.6$ (\textcolor{green}{$\times$}), $0.9$ (\textcolor{blue}{$\blacktriangle$}), $1$ (\textcolor{orange}{$\blacksquare$}), $1.1$ (\textcolor{teal}{$+$})).}
    \label{fig:Example}
\end{figure}
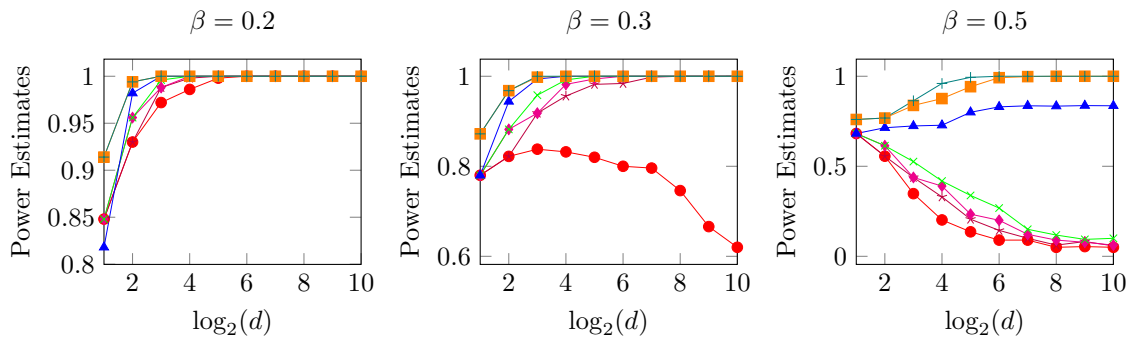

\section{Empirical performance of the proposed tests}

In this section, we investigate the empirical performance of our tests for high-dimensional data. First, we study their level properties and then compare their powers with some popular tests available in the literature. For this comparison, we consider the multivariate run tests based on minimum spanning tree  \citep{friedman1979multivariate}
and shortest Hamiltonian path \citep{biswas2014distribution}, the tests based on averages of inter-point distances proposed by \cite{baringhaus2004new} and \cite{biswas2014nonparametric}, the nearest neighbor test \citep{schilling1986multivariate, henze1988multivariate}, 
and the test based on maximum mean discrepancy \citep{gretton2012kernel}. Henceforth, we shall refer to them as the FR test, the SHP test, the BF test, the BG test, the NN test, and the MMD test, respectively. For the NN test, we consider the test based on $3$ neighbors, which has been reported to perform well in the literature \citep[see, e.g.,][]{schilling1986multivariate}. Throughout this article, all tests are considered to have the 5\% nominal level. The SHP test has the distribution-free property. For all other tests, the cut-off is computed based on 500 random permutations as before. 

\subsection{Analysis of simulated data sets}
\label{Synthetic Data Analysis}

To study the level properties of our tests, we generated two sets of independent observations from the $d$-variate standard normal distribution
and used them as observations from $F$ and $G$, respectively. This experiment was repeated 500 times, and for each test, we
computed the proportion of times it rejected $H_0$. We carried out our experiment for different sample sizes ($n=m=20$, $35$ and $50$) and dimension ($d=2^{i}$ for $i=1,2,\ldots,10$). Figure 4 clearly shows that on all occasions, the BD-${\ell_2}$ rejected $H_0$ in nearly $5\%$ of the cases. BD-${\ell_1}$, BD-${\exp}$ and BD-${\log}$ also exhibited similar level properties. To avoid repetition, they are not reported here. We also observed almost similar results when the normal distribution was replaced by other multivariate distributions. But, to save space, we decided not to report them.

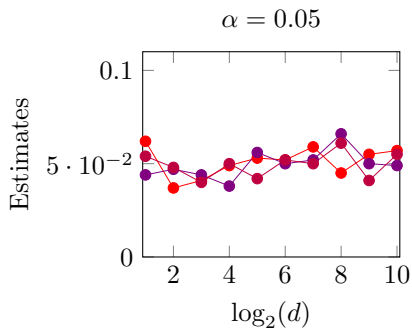
\begin{figure}[!h]
\centering
\begin{tikzpicture}


\begin{axis}[xmin = 0.9, xmax = 10.1, ymin = 0, ymax = 0.11, xlabel = {$\log_2(d)$}, ylabel = {Estimates}, title = {$\alpha=0.05$}]
\addplot[color = red,   mark = *, step = 1cm,very thin]coordinates{(1,0.062)(2,0.037)(3,0.041)(4,0.049)(5,0.053)(6,0.052)(7,0.059)(8,0.045)(9,0.055)(10,0.057)};

\addplot[color = violet, mark = *, step = 1cm,very thin]coordinates{(1,0.044)(2,0.047)(3,0.044)(4,0.038)(5,0.056)(6,0.05)(7,0.052)(8,0.066)(9,0.05)(10,0.049)};

\addplot[color = purple, mark = *, step = 1cm,very thin]coordinates{(1,0.054)(2,0.048)(3,0.040)(4,0.050)(5,0.042)(6,0.052)(7,0.05)(8,0.061)(9,0.041)(10,0.055)};

\end{axis}
\end{tikzpicture}





\caption{Observed levels of the BD-${\ell_2}$ test for $n=m=20$ (\textcolor{red}{$\tikzcircle{2pt}$}), $n=m=35$ ({$\tikzcirclev{2pt}$}) and $n=m=50$ ({$\tikzcirclep{2pt}$}).}
   \label{level}
\end{figure}

Next, we investigate the power properties of the proposed tests and compare their performance with some of the existing methods. We consider two types of examples
for this purpose. In Section~\ref{HDLSS Scenario}, we deal with examples with fixed sample sizes and study the performance of different tests as the dimension increases. In Section \ref{Shrinking Alternative Scenario}, we consider the situations, where the conditions for HDLSS consistency of the proposed tests do not hold  (i.e., we have $\nu^2+(\sigma_F-\sigma_G)^2=0$ and $\lim e_{h,\psi}(F,G) = 0$). In such cases, we investigate the performance of different tests when the sample sizes grow with the dimension.

\subsubsection{Dimension increases when the sample sizes remain fixed}
\label{HDLSS Scenario}

We begin with the four examples (Examples 1-4) discussed in Sections 3.1 and 3.2. The powers of the proposed tests and those of their competitors are reported in Figure \ref{Sim1}.

\begin{figure}[!h]
    \centering
    \begin{tikzpicture}
\begin{axis}[xmin = 1, xmax = 10, xlabel = {$\log_2(d)$}, ylabel = {Power Estimates}, title = {\bf Example 1}]
\addplot[color = red,   mark = *, step = 1cm,very thin]coordinates{(1,0.098)(2,0.09)(3,0.088)(4,0.116)(5,0.128)(6,0.154)(7,0.202)(8,0.356)(9,0.696)(10,0.99)};

\addplot[color = violet, mark = *, step = 1cm,very thin]coordinates{(1,0.104)(2,0.092)(3,0.104)(4,0.12)(5,0.146)(6,0.164)(7,0.202)(8,0.39)(9,0.672)(10,0.99)};

\addplot[color = purple, mark = *, step = 1cm,very thin]coordinates{(1,0.096)(2,0.084)(3,0.112)(4,0.122)(5,0.144)(6,0.164)(7,0.204)(8,0.384)(9,0.62)(10,0.956)};

\addplot[color = magenta, mark = *, step = 1cm,very thin]coordinates{(1,0.102)(2,0.088)(3,0.104)(4,0.124)(5,0.148)(6,0.168)(7,0.216)(8,0.382)(9,0.68)(10,0.98)};

\addplot[color = green, mark = square*, step = 1cm,very thin]coordinates{(1,0.072)(2,0.1)(3,0.098)(4,0.122)(5,0.17)(6,0.164)(7,0.366)(8,0.53)(9,0.786)(10,0.958)};

\addplot[color = blue,  mark = diamond*, step = 1cm,very thin]coordinates{(1,0.14)(2,0.15)(3,0.254)(4,0.37)(5,0.582)(6,0.84)(7,0.968)(8,0.998)(9,1)(10,1)};

\addplot[color = orange, mark = square*, step = 1cm,very thin]coordinates{(1,0.086)(2,0.092)(3,0.09)(4,0.138)(5,0.204)(6,0.306)(7,0.408)(8,0.702)(9,0.93)(10,0.998)};

\addplot[color = teal, mark = diamond*, step = 1cm,very thin]coordinates{(1,0.116)(2,0.140)(3,0.23)(4,0.35)(5,0.57)(6,0.834)(7,0.968)(8,1)(9,1)(10,1)};

\addplot[color = yellow, mark = square*, step = 1cm,very thin]coordinates{(1,0.064)(2,0.076)(3,0.084)(4,0.092)(5,0.130)(6,0.23)(7,0.28)(8,0.512)(9,0.758)(10,0.962)};

\addplot[color = black, mark = diamond*, step = 1cm,very thin]coordinates{(1,0.052)(2,0.034)(3,0.034)(4,0.056)(5,0.034)(6,0.048)(7,0.076)(8,0.182)(9,0.402)(10,0.826)};
\end{axis}
\end{tikzpicture}
\begin{tikzpicture}
\begin{axis}[xmin = 1, xmax = 10, xlabel = {$\log_2(d)$}, ylabel = {Power Estimates}, title = {\bf Example 2}]

\addplot[color = red,   mark = *, step = 1cm,very thin]coordinates{(1,0.062)(2,0.092)(3,0.122)(4,0.23)(5,0.436)(6,0.736)(7,0.96)(8,1)(9,1)(10,1)};

\addplot[color = violet, mark = *, step = 1cm,very thin]coordinates{(1,0.066)(2,0.088)(3,0.114)(4,0.224)(5,0.422)(6,0.722)(7,0.958)(8,1)(9,1)(10,1)};

\addplot[color = purple, mark = *, step = 1cm,very thin]coordinates{(1,0.062)(2,0.08)(3,0.118)(4,0.22)(5,0.412)(6,0.708)(7,0.954)(8,1)(9,1)(10,1)};

\addplot[color = magenta, mark = *, step = 1cm,very thin]coordinates{(1,0.066)(2,0.086)(3,0.118)(4,0.22)(5,0.412)(6,0.708)(7,0.954)(8,1)(9,1)(10,1)};

\addplot[color = green, mark = square*, step = 1cm,very thin]coordinates{(1,0.042)(2,0.09)(3,0.06)(4,0.074)(5,0.068)(6,0.064)(7,0.052)(8,0.06)(9,0.02)(10,0.004)};

\addplot[color = blue,  mark = diamond*, step = 1cm,very thin]coordinates{(1,0.052)(2,0.044)(3,0.042)(4,0.056)(5,0.044)(6,0.02)(7,0.074)(8,0.088)(9,0.128)(10,0.156)};

\addplot[color = orange, mark = square*, step = 1cm,very thin]coordinates{(1,0.066)(2,0.086)(3,0.052)(4,0.04)(5,0.058)(6,0.064)(7,0.052)(8,0.046)(9,0.02)(10,0.002)};

\addplot[color = teal, mark = diamond*, step = 1cm,very thin]coordinates{(1,0.06)(2,0.058)(3,0.056)(4,0.06)(5,0.06)(6,0.092)(7,0.118)(8,0.164)(9,0.248)(10,0.366)};

\addplot[color = yellow, mark = square*, step = 1cm,very thin]coordinates{(1,0.032)(2,0.052)(3,0.042)(4,0.056)(5,0.062)(6,0.09)(7,0.192)(8,0.426)(9,0.84)(10,1)};

\addplot[color = black, mark = diamond*, step = 1cm,very thin]coordinates{(1,0.062)(2,0.102)(3,0.148)(4,0.246)(5,0.454)(6,0.746)(7,0.966)(8,1)(9,1)(10,1)};

\end{axis}
\end{tikzpicture}

\begin{tikzpicture}
\begin{axis}[xmin = 1, xmax = 10, xlabel = {$\log_2(d)$}, ylabel = {Power Estimates}, title = {\bf Example 3}]

  \addplot[color = red,   mark = *, step = 1cm,very thin]coordinates{(1,0.074)(2,0.072)(3,0.034)(4,0.042)(5,0.034)(6,0.042)(7,0.044)(8,0.068)(9,0.04)(10,0.06)};

      \addplot[color = violet, mark = *, step = 1cm,very thin]coordinates{(1,0.082)(2,0.084)(3,0.094)(4,0.128)(5,0.208)(6,0.442)(7,0.876)(8,1)(9,1)(10,1)};

      \addplot[color = purple, mark = *, step = 1cm,very thin]coordinates{(1,0.212)(2,0.326)(3,0.548)(4,0.784)(5,0.954)(6,1)(7,1)(8,1)(9,1)(10,1)};

      \addplot[color = magenta, mark = *, step = 1cm,very thin]coordinates{(1,0.092)(2,0.112)(3,0.188)(4,0.288)(5,0.574)(6,0.918)(7,1)(8,1)(9,1)(10,1)};

\addplot[color = green, mark = square*, step = 1cm,very thin]coordinates{(1,0.242)(2,0.364)(3,0.458)(4,0.532)(5,0.542)(6,0.576)(7,0.56)(8,0.592)(9,0.544)(10,0.568)};

\addplot[color = blue,  mark = diamond*, step = 1cm,very thin]coordinates{(1,0.122)(2,0.088)(3,0.102)(4,0.08)(5,0.076)(6,0.074)(7,0.064)(8,0.06)(9,0.062)(10,0.058)};

\addplot[color = orange, mark = square*, step = 1cm,very thin]coordinates{(1,0.326)(2,0.482)(3,0.566)(4,0.628)(5,0.682)(6,0.646)(7,0.658)(8,0.694)(9,0.696)(10,0.69)};

\addplot[color = teal, mark = diamond*, step = 1cm,very thin]coordinates{(1,0.246)(2,0.22)(3,0.196)(4,0.172)(5,0.142)(6,0.096)(7,0.086)(8,0.078)(9,0.072)(10,0.064)};

\addplot[color = yellow, mark = square*, step = 1cm,very thin]coordinates{(1,0.18)(2,0.3)(3,0.408)(4,0.496)(5,0.526)(6,0.558)(7,0.552)(8,0.56)(9,0.576)(10,0.586)};

\addplot[color = black, mark = diamond*, step = 1cm,very thin]coordinates{(1,0.052)(2,0.044)(3,0.022)(4,0.036)(5,0.038)(6,0.042)(7,0.052)(8,0.062)(9,0.044)(10,0.064)};

\end{axis}
\end{tikzpicture}
\begin{tikzpicture}
\begin{axis}[xmin = 1, xmax = 10, xlabel = {$\log_2(d)$}, ylabel = {Power Estimates}, title = {\bf Example 4}]

\addplot[color = red, mark = *, step = 1cm,very thin]coordinates{(1,0.73)(2,0.368)(3,0.142)(4,0.084)(5,0.062)(6,0.066)(7,0.034)(8,0.05)(9,0.034)(10,0.056)};

\addplot[color = violet, mark = *, step = 1cm,very thin]coordinates{(1,0.622)(2,0.242)(3,0.108)(4,0.058)(5,0.068)(6,0.08)(7,0.03)(8,0.062)(9,0.056)(10,0.052)};

\addplot[color = purple, mark = *, step = 1cm,very thin]coordinates{(1,0.938)(2,0.99)(3,1)(4,1)(5,1)(6,1)(7,1)(8,1)(9,1)(10,1)};

\addplot[color = magenta, mark = *, step = 1cm,very thin]coordinates{(1,0.676)(2,0.452)(3,0.43)(4,0.614)(5,0.928)(6,1)(7,1)(8,1)(9,1)(10,1)};

\addplot[color = green, mark = square*, step = 1cm,very thin]coordinates{(1,0.516)(2,0.65)(3,0.566)(4,0.466)(5,0.308)(6,0.234)(7,0.176)(8,0.12)(9,0.106)(10,0.116)};

\addplot[color = blue,  mark = diamond*, step = 1cm,very thin]coordinates{(1,0.46)(2,0.328)(3,0.208)(4,0.106)(5,0.102)(6,0.066)(7,0.054)(8,0.07)(9,0.042)(10,0.042)};

\addplot[color = orange, mark = square*, step = 1cm,very thin]coordinates{(1,0.668)(2,0.754)(3,0.71)(4,0.582)(5,0.416)(6,0.21)(7,0.086)(8,0.04)(9,0.012)(10,0)};

\addplot[color = teal, mark = diamond*, step = 1cm,very thin]coordinates{(1,0.794)(2,0.676)(3,0.424)(4,0.19)(5,0.112)(6,0.074)(7,0.038)(8,0.054)(9,0.032)(10,0.052)};

\addplot[color = yellow, mark = square*, step = 1cm,very thin]coordinates{(1,0.396)(2,0.454)(3,0.372)(4,0.244)(5,0.164)(6,0.118)(7,0.088)(8,0.07)(9,0.046)(10,0.06)};

\addplot[color = black, mark = diamond*, step = 1cm,very thin]coordinates{(1,0.062)(2,0.05)(3,0.04)(4,0.036)(5,0.07)(6,0.056)(7,0.058)(8,0.07)(9,0.036)(10,0.038)};

\end{axis}
\end{tikzpicture}
\caption{Powers of  BD-${\ell_2}$ (\textcolor{red}{$\tikzcircle{2pt}$}), BD-${\ell_1}$ ({$\tikzcirclev{2pt}$}), BD-${\exp}$ ({$\tikzcirclep{2pt}$}), BD-${\log}$ ({$\tikzcirclem{2pt}$}), FR (\textcolor{green}{$\blacksquare$}), BF (\textcolor{blue}{$\blacklozenge$}), NN (\textcolor{orange}{$\blacksquare$}), MMD (\textcolor{teal}{$\blacklozenge$}), SHP (\textcolor{yellow}{$\blacksquare$}), BG (\textcolor{black}{$\blacklozenge$}) tests in Examples 1-4.}
   \label{Sim1}
\end{figure}
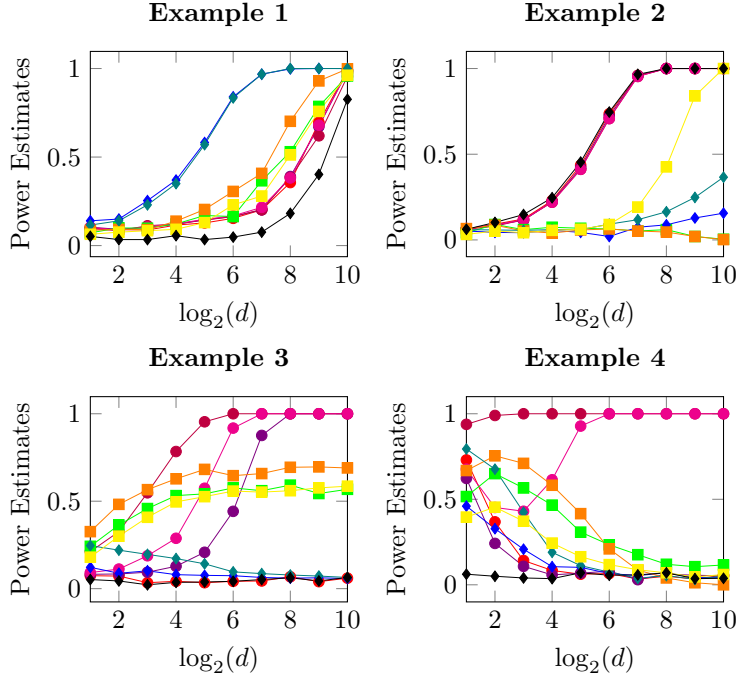

In the location problem in Example 1, BF and MMD tests outperformed all other tests considered here. However, the powers of the proposed ball divergence tests were comparable to the rest (BG, NN, FR, and SHP tests).

In Example 2, all tests based on ball divergence and the BG test had similar performance, and they performed much better than their competitors. Among the rest, the SHP test had a relatively higher power. One can notice that in high dimensions, FR and NN tests had powers close to $0$. This matches with the findings of \cite{biswas2014distribution} and \cite{mondal2015high}, where the authors explained the reasons for the poor performance of FR and NN tests in high dimensional scale problems.  

We have seen that in Example-3, $\nu^2+(\sigma_F-\sigma_G)^2=0$ but $\liminf_{d\to\infty} e_{h,\psi}(F,G)>0$ for $\psi=\psi_1,\psi_2,\psi_3$ with $h(t)=t$. So, as expected, in this example, BD-${\ell_2}$ did not have satisfactory performance, but BD-${\ell_1}$, BD-${\exp}$ and BD-${\log}$ performed well in high dimensions. Among them, BD-${\exp}$ had a clear edge. While the powers of these three tests increased with the dimension, that was not the case for 
other competing methods. Note that these competing methods are based on $\ell_2$ distances. The use of a different distance function may improve their performance. 

In the presence of heavy-tailed distributions in Example 4, all tests except BD-${\exp}$ and BD-${\log}$ had poor performance in high dimensions. Among these two tests, the one based on the bounded $\psi$-function (i.e., $\psi_2(t) = 1-e^{-t/2}$) performed better. The reason for the excellent performance of these tests has already been discussed in Section 3.2. 

As we have discussed before, in Example 5, the power of any test based on fixed sample sizes is expected to decrease as the dimension increases. We also observed the same for all tests considered here. So, we do not report those results. Instead, we consider three other examples (Examples 6-8) to investigate the empirical performance of different tests.


\vspace{0.05in}
\textbf{Example 6}  $F$ is the $d$-variate standard normal distribution while $G$ is an equal mixture of  $\mathcal{N}_d(0.5\,{\bf 1}_d, {\bf I}_d)$ and $\mathcal{N}_d(-0.5\,{\bf 1}_d, {\bf I}_d)$.

\vspace{0.05in}
\textbf{Example 7}  $F$ is same as in Example 6, but $G$ is unequal mixture of $\mathcal{N}_d({\bf 1}_d, {\bf I}_d)$ and $\mathcal{N}_d(-0.25\,{\bf 1}_d, {\bf I}_d)$ with mixing proportions $0.2$ and $0.8$, respectively.

\vspace{0.05in}
\textbf{Example 8}  Both $F$ and $G$ have independent and identically distributed coordinate variables. The coordinate variables in $F$ follow ${\cal N}_1(0,2)$ distribution, while in $G$, they follow the standard $t$ distribution with $4$ degrees of freedom. 

\vspace{0.05in}
For each example, we generated $50$ observations from each of the two distributions and considered $10$ different choices of $d$ $(d=2^i$ for $i=1,2,\ldots,10$) as before. Each experiment was repeated 500 times to estimate the power of different tests by the proportion of times they rejected $H_0$. These estimated powers are shown in Figure \ref{Sim2}. This figure clearly shows that in the examples involving mixture normal distributions, the BG test, and the ball divergence tests performed better than their competitors. In Example 8, the ball divergence tests outperformed all other competing tests considered here. Like Example 3, here also, we have $\nu^2+(\sigma_F-\sigma_G)^2=0$ but  $\liminf_{d\to\infty}e_{h,\psi}(F,G)>0$ for other three choices of $\psi$. So, as expected, BD-${\ell_1}$, BD-${\exp}$ and BD-${\log}$ performed much better than BD-${\ell_2}$.

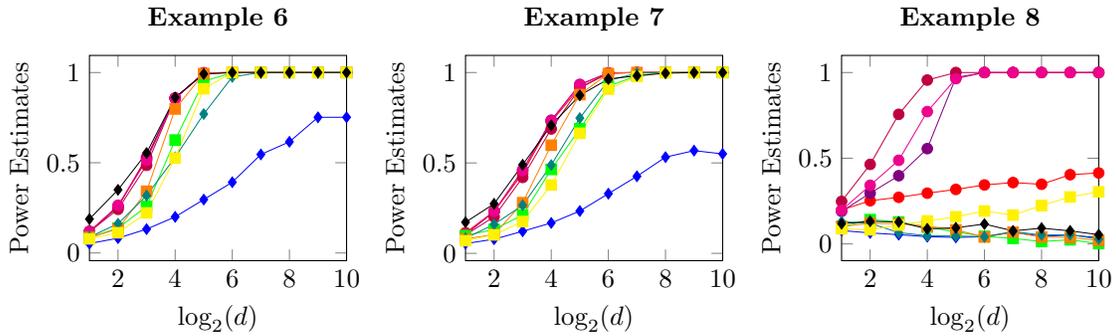
\begin{figure}[t]
    \centering
    \begin{tikzpicture}
    \begin{axis}[xmin = 1, xmax = 10, xlabel = {$\log_2(d)$}, ylabel = {Power Estimates}, title = {\bf Example 6}]
    \addplot[color = red, mark = *, step = 1cm,very thin]coordinates{(1,0.114)(2,0.258)(3,0.53)(4,0.858)(5,0.998)(6,1)(7,1)(8,1)(9,1)(10,1)};

\addplot[color = violet, mark = *, step = 1cm,very thin]coordinates{(1,0.12)(2,0.256)(3,0.524)(4,0.858)(5,0.998)(6,1)(7,1)(8,1)(9,1)(10,1)};

\addplot[color = purple, mark = *, step = 1cm,very thin]coordinates{(1,0.118)(2,0.242)(3,0.486)(4,0.858)(5,0.994)(6,1)(7,1)(8,1)(9,1)(10,1)};

\addplot[color = magenta, mark = *, step = 1cm,very thin]coordinates{(1,0.118)(2,0.264)(3,0.514)(4,0.856)(5,0.998)(6,1)(7,1)(8,1)(9,1)(10,1)};

\addplot[color = green, mark = square*, step = 1cm,very thin]coordinates{(1,0.086)(2,0.138)(3,0.258)(4,0.626)(5,0.952)(6,0.998)(7,1)(8,1)(9,1)(10,1)};

\addplot[color = blue,  mark = diamond*, step = 1cm,very thin]coordinates{(1,0.052)(2,0.082)(3,0.132)(4,0.2)(5,0.296)(6,0.392)(7,0.546)(8,0.616)(9,0.752)(10,0.752)};

\addplot[color = orange, mark = square*, step = 1cm,very thin]coordinates{(1,0.084)(2,0.134)(3,0.342)(4,0.798)(5,0.994)(6,1)(7,1)(8,1)(9,1)(10,1)};

\addplot[color = teal, mark = diamond*, step = 1cm,very thin]coordinates{(1,0.088)(2,0.162)(3,0.318)(4,0.53)(5,0.77)(6,0.976)(7,1)(8,1)(9,1)(10,1)};

\addplot[color = yellow, mark = square*, step = 1cm,very thin]coordinates{(1,0.078)(2,0.112)(3,0.222)(4,0.526)(5,0.91)(6,1)(7,1)(8,1)(9,1)(10,1)};

\addplot[color = black, mark = diamond*, step = 1cm,very thin]coordinates{(1,0.188)(2,0.35)(3,0.554)(4,0.862)(5,0.992)(6,1)(7,1)(8,1)(9,1)(10,1)};

    \end{axis}
    \end{tikzpicture}
    \begin{tikzpicture}
    \begin{axis}[xmin = 1, xmax = 10, xlabel = {$\log_2(d)$}, ylabel = {Power Estimates}, title = {\bf Example 7}]
    \addplot[color = red, mark = *, step = 1cm,very thin]coordinates{(1,0.114)(2,0.232)(3,0.442)(4,0.728)(5,0.918)(6,0.996)(7,1)(8,1)(9,1)(10,1)};

\addplot[color = violet, mark = *, step = 1cm,very thin]coordinates{(1,0.104)(2,0.23)(3,0.46)(4,0.736)(5,0.934)(6,0.998)(7,1)(8,1)(9,1)(10,1)};

\addplot[color = purple, mark = *, step = 1cm,very thin]coordinates{(1,0.102)(2,0.202)(3,0.42)(4,0.688)(5,0.914)(6,0.996)(7,1)(8,1)(9,1)(10,1)};

\addplot[color = magenta, mark = *, step = 1cm,very thin]coordinates{(1,0.106)(2,0.232)(3,0.458)(4,0.734)(5,0.932)(6,0.998)(7,1)(8,1)(9,1)(10,1)};

\addplot[color = green, mark = square*, step = 1cm,very thin]coordinates{(1,0.098)(2,0.132)(3,0.212)(4,0.464)(5,0.688)(6,0.924)(7,0.982)(8,1)(9,1)(10,1)};

\addplot[color = blue,  mark = diamond*, step = 1cm,very thin]coordinates{(1,0.054)(2,0.078)(3,0.122)(4,0.168)(5,0.234)(6,0.33)(7,0.426)(8,0.532)(9,0.568)(10,0.55)};

\addplot[color = orange, mark = square*, step = 1cm,very thin]coordinates{(1,0.08)(2,0.104)(3,0.28)(4,0.598)(5,0.878)(6,0.994)(7,1)(8,1)(9,1)(10,1)};

\addplot[color = teal, mark = diamond*, step = 1cm,very thin]coordinates{(1,0.084)(2,0.16)(3,0.266)(4,0.488)(5,0.748)(6,0.956)(7,0.996)(8,1)(9,1)(10,1)};

\addplot[color = yellow, mark = square*, step = 1cm,very thin]coordinates{(1,0.07)(2,0.102)(3,0.176)(4,0.378)(5,0.664)(6,0.908)(7,0.978)(8,1)(9,1)(10,1)};

\addplot[color = black, mark = diamond*, step = 1cm,very thin]coordinates{(1,0.172)(2,0.274)(3,0.49)(4,0.708)(5,0.874)(6,0.964)(7,0.982)(8,0.996)(9,1)(10,1)};

    \end{axis}
    \end{tikzpicture}
    \begin{tikzpicture}
    \begin{axis}[xmin = 1, xmax = 10, xlabel = {$\log_2(d)$}, ylabel = {Power Estimates}, title = {\bf Example 8}]
    \addplot[color = red, mark = *, step = 1cm,very thin]coordinates{(1,0.198)(2,0.252)(3,0.272)(4,0.296)(5,0.318)(6,0.344)(7,0.358)(8,0.348)(9,0.404)(10,0.414)};

\addplot[color = violet, mark = *, step = 1cm,very thin]coordinates{(1,0.19)(2,0.296)(3,0.398)(4,0.556)(5,0.971)(6,1)(7,1)(8,1)(9,1)(10,1)};

\addplot[color = purple, mark = *, step = 1cm,very thin]coordinates{(1,0.248)(2,0.464)(3,0.756)(4,0.956)(5,1)(6,1)(7,1)(8,1)(9,1)(10,1)};

\addplot[color = magenta, mark = *, step = 1cm,very thin]coordinates{(1,0.196)(2,0.342)(3,0.488)(4,0.772)(5,0.964)(6,1)(7,1)(8,1)(9,1)(10,1)};

\addplot[color = green, mark = square*, step = 1cm,very thin]coordinates{(1,0.098)(2,0.142)(3,0.128)(4,0.098)(5,0.07)(6,0.044)(7,0.032)(8,0.014)(9,0.024)(10,0.002)};

\addplot[color = blue,  mark = diamond*, step = 1cm,very thin]coordinates{(1,0.078)(2,0.064)(3,0.054)(4,0.042)(5,0.038)(6,0.042)(7,0.068)(8,0.044)(9,0.052)(10,0.036)};

\addplot[color = orange, mark = square*, step = 1cm,very thin]coordinates{(1,0.102)(2,0.12)(3,0.118)(4,0.104)(5,0.082)(6,0.04)(7,0.068)(8,0.04)(9,0.038)(10,0.02)};

\addplot[color = teal, mark = diamond*, step = 1cm,very thin]coordinates{(1,0.13)(2,0.12)(3,0.066)(4,0.048)(5,0.046)(6,0.042)(7,0.072)(8,0.052)(9,0.054)(10,0.028)};

\addplot[color = yellow, mark = square*, step = 1cm,very thin]coordinates{(1,0.088)(2,0.082)(3,0.11)(4,0.134)(5,0.158)(6,0.192)(7,0.168)(8,0.224)(9,0.274)(10,0.304)};

\addplot[color = black, mark = diamond*, step = 1cm,very thin]coordinates{(1,0.118)(2,0.132)(3,0.128)(4,0.088)(5,0.094)(6,0.116)(7,0.076)(8,0.092)(9,0.076)(10,0.054)};

    \end{axis}
    \end{tikzpicture}

\caption{Powers of  BD-${\ell_2}$ (\textcolor{red}{$\tikzcircle{2pt}$}), BD-${\ell_1}$ ({$\tikzcirclev{2pt}$}), BD-${\exp}$ ({$\tikzcirclep{2pt}$}), BD-${\log}$ ({$\tikzcirclem{2pt}$}), FR (\textcolor{green}{$\blacksquare$}), BF (\textcolor{blue}{$\blacklozenge$}), NN (\textcolor{orange}{$\blacksquare$}), MMD (\textcolor{teal}{$\blacklozenge$}), SHP (\textcolor{yellow}{$\blacksquare$}), BG (\textcolor{black}{$\blacklozenge$}) tests in Examples 6-8.}
    \label{Sim2}
\end{figure}


\subsubsection{Sample sizes grow with the dimension}
\label{Shrinking Alternative Scenario}

In this section, we deal with some examples, where we do not have a theoretical guarantee for the consistency of the ball divergence tests in the HDLSS regime, and investigate how the proposed tests and their competitors perform  when the sample sizes also grow with the dimension. As
before, the powers of all tests are computed based on 500 replications. 

We consider six examples (Examples 9-14) in this section. The first example is similar to Example 5. This example involving two normal distributions was discussed in Section 4.2. 

\vspace{0.05in}
{\bf Example 9} The coordinate variables in $F$ and $G$ are independent and identically as $\mathcal{N}_1(d^{-0.3},1)$ and $\mathcal{N}_1(-d^{-0.3},1)$, respectively. We consider sample sizes $n=m=5+\lfloor \sqrt{d}\rfloor$ that increase at $O(\sqrt{d})$ rate. 

\vspace{0.05in}
In this location problem, the BF test and the MMD test had the best performance (see Figure~\ref{Sim3})  closely followed by NN, BG, and ball divergence tests. The SHP test had relatively low power. Interestingly, the powers of all tests converged to unity as the dimension increased.

Next, we consider an example (Example 10), which is similar to Example 3. Recall that in Example 3, BD-${\ell_2}$ had poor performances in the HDLSS regime. Here we study its behavior when the sample sizes $m=n=d+5$ increase with the dimension. 

\vspace{0.05in}
\textbf{Example 10} Both $F = \mathcal{N}_d({\bf 0}_d,\sigmat_{1,d}^{\circ})$ and $G = \mathcal{N}_d({\bf 0}_d,\sigmat_{2,d}^{\circ})$ have the same mean ${\bf 0}_d$, but different diagonal dispersion matrices $\sigmat_{1,d}^{\circ}$ and $\sigmat_{2,d}^{\circ}$, respectively. The first $d/2$ diagonal elements of $\sigmat_{1,d}^{\circ}$  are 1, and the rest are 5. On the contrary, $\sigmat_{2,d}^{\circ}$ has the first $d/2$ diagonal elements equal to $5$ and rest equal to $1$.

\vspace{0.05in}
In this example,  NN, FR, and BD-${\exp}$ tests had better performance than their competitors, with BD-${\exp}$ having an edge (see Figure~\ref{Sim3}). Interestingly, the powers of all tests barring the BG test showed a tendency to converge to unity as the dimension increased. Note that this was not the case in Example 3 when samples of fixed sizes were used. 

Our next example deals with two multivariate normal distributions, which have the same mean and marginal variances but differ in their correlation structures. 

\vspace{0.05in}
\textbf{Example 11} $F= \mathcal{N}_d({\bf 0}_d,\sigmat_{1,d}^{\ast})$ and $G = \mathcal{N}_d({\bf 0}_d,\sigmat_{2,d}^{\ast})$ have the same mean ${\bf 0}_d$ but different dispersion matrices  $\sigmat_{1,d}^{\ast} = ((0.1^{|i-j|}))$ and $\sigmat_{2,d}^{\ast} = ((0.5^{|i-j|}))$, respectively.


\vspace{0.05in}
\cite{sarkar2018some} observed that in this example BG and BF tests do not perform well in HDLSS situations. We observed the same in our numerical study. In fact, all tests considered here had poor performance in the HDLSS setup. Note that here we have $\nu^2+(\sigma_F-\sigma_G)^2=0$ and $\lim_{d \rightarrow \infty} e_{h,\phi}(F,G)=0$ for all three choices of $\psi$ (i.e., $\psi_1$, $\psi_2$ and $\psi_3$). So, the poor performance of the ball divergence tests in the HDLSS setup was quite expected. Figure \ref{Sim3} shows the powers of different tests when the sample sizes $n=m=d+5$ increase with the dimension. Here the graph-based tests performed much better than the average distance-based tests. However, unlike BG, BF, and MMD tests, the powers of the ball divergence tests had a sharp rise in higher dimensions. 

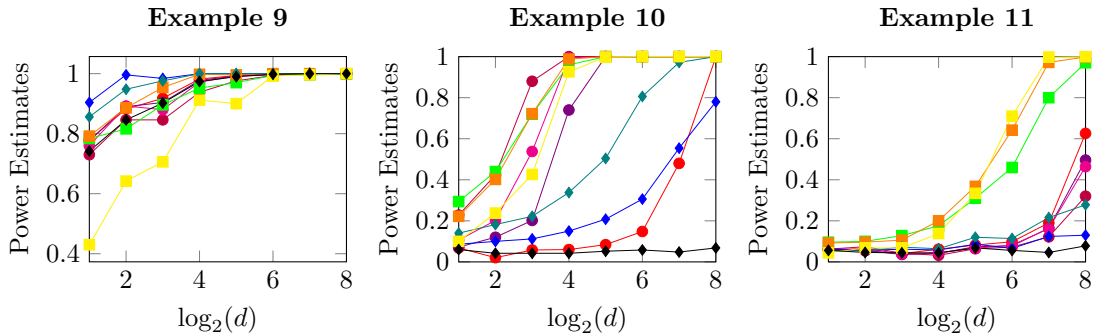
\begin{figure}[!h]
    \centering
    \begin{tikzpicture}
    \begin{axis}[xmin = 1, xmax = 8, xlabel = {$\log_2(d)$}, ylabel = {Power Estimates}, title = {\bf Example 9}]
\addplot[color = red,   mark = *, step = 1cm,very thin]coordinates{(1,0.78)(2,0.882)(3,0.918)(4,0.982)(5,0.994)(6,1)(7,1)(8,1)};

\addplot[color = violet, mark = *, step = 1cm,very thin]coordinates{(1,0.776)(2,0.89)(3,0.904)(4,0.978)(5,0.992)(6,0.998)(7,1)(8,1)};

\addplot[color = purple, mark = *, step = 1cm,very thin]coordinates{(1,0.73)(2,0.846)(3,0.846)(4,0.938)(5,0.976)(6,0.994)(7,1)(8,1)};

\addplot[color = magenta, mark = *, step = 1cm,very thin]coordinates{(1,0.76)(2,0.892)(3,0.88)(4,0.972)(5,0.992)(6,0.996)(7,1)(8,1)};

\addplot[color = green, mark = square*, step = 1cm,very thin]coordinates{(1,0.784)(2,0.816)(3,0.896)(4,0.952)(5,0.97)(6,0.994)(7,0.996)(8,0.998)};

\addplot[color = blue,  mark = diamond*, step = 1cm,very thin]coordinates{(1,0.904)(2,0.996)(3,0.984)(4,1)(5,1)(6,1)(7,1)(8,1)};

\addplot[color = orange, mark = square*, step = 1cm,very thin]coordinates{(1,0.792)(2,0.886)(3,0.954)(4,0.998)(5,0.996)(6,1)(7,1)(8,1)};

\addplot[color = teal, mark = diamond*, step = 1cm,very thin]coordinates{(1,0.856)(2,0.948)(3,0.976)(4,1)(5,1)(6,1)(7,1)(8,1)};

\addplot[color = yellow, mark = square*, step = 1cm,very thin]coordinates{(1,0.43)(2,0.642)(3,0.706)(4,0.912)(5,0.9)(6,0.992)(7,0.996)(8,1)};

\addplot[color = black, mark = diamond*, step = 1cm,very thin]coordinates{(1,0.742)(2,0.846)(3,0.902)(4,0.974)(5,0.99)(6,0.998)(7,1)(8,1)};
    \end{axis}
    \end{tikzpicture}
    \begin{tikzpicture}
     \begin{axis}[xmin = 1, xmax = 8, ymin = 0, ymax = 1,xlabel = {$\log_2(d)$}, ylabel = {Power Estimates}, title = {\bf Example 10}]
\addplot[color = red,   mark = *, step = 1cm,very thin]coordinates{(1,0.068)(2,0.0206)(3,0.058)(4,0.06)(5,0.084)(6,0.148)(7,0.48)(8,0.998)};

\addplot[color = violet, mark = *, step = 1cm,very thin]coordinates{(1,0.072)(2,0.12)(3,0.202)(4,0.74)(5,1)(6,1)(7,1)(8,1)};

\addplot[color = purple, mark = *, step = 1cm,very thin]coordinates{(1,0.228)(2,0.424)(3,0.88)(4,1)(5,1)(6,1)(7,1)(8,1)};

\addplot[color = magenta, mark = *, step = 1cm,very thin]coordinates{(1,0.106)(2,0.202)(3,0.538)(4,0.994)(5,1)(6,1)(7,1)(8,1)};

\addplot[color = green, mark = square*, step = 1cm,very thin]coordinates{(1,0.294)(2,0.44)(3,0.72)(4,0.96)(5,1)(6,1)(7,1)(8,1)};

\addplot[color = blue,  mark = diamond*, step = 1cm,very thin]coordinates{(1,0.086)(2,0.1)(3,0.112)(4,0.15)(5,0.208)(6,0.306)(7,0.554)(8,0.78)};

\addplot[color = orange, mark = square*, step = 1cm,very thin]coordinates{(1,0.222)(2,0.402)(3,0.722)(4,0.988)(5,1)(6,1)(7,1)(8,1)};

\addplot[color = teal, mark = diamond*, step = 1cm,very thin]coordinates{(1,0.14)(2,0.182)(3,0.222)(4,0.338)(5,0.504)(6,0.806)(7,0.972)(8,1)};

\addplot[color = yellow, mark = square*, step = 1cm,very thin]coordinates{(1,0.1)(2,0.238)(3,0.426)(4,0.926)(5,0.998)(6,1)(7,1)(8,1)};

\addplot[color = black, mark = diamond*, step = 1cm,very thin]coordinates{(1,0.06)(2,0.042)(3,0.042)(4,0.042)(5,0.052)(6,0.058)(7,0.048)(8,0.068)};
    \end{axis}
    \end{tikzpicture}
    \begin{tikzpicture}
     \begin{axis}[xmin = 1, xmax = 8, ymin = 0, ymax = 1,xlabel = {$\log_2(d)$}, ylabel = {Power Estimates}, title = {\bf Example 11}]
\addplot[color = red,   mark = *, step = 1cm,very thin]coordinates{(1,0.06)(2,0.054)(3,0.04)(4,0.062)(5,0.082)(6,0.098)(7,0.19)(8,0.626)};

\addplot[color = violet, mark = *, step = 1cm,very thin]coordinates{(1,0.066)(2,0.056)(3,0.036)(4,0.042)(5,0.074)(6,0.078)(7,0.164)(8,0.496)};

\addplot[color = purple, mark = *, step = 1cm,very thin]coordinates{(1,0.062)(2,0.074)(3,0.042)(4,0.032)(5,0.064)(6,0.072)(7,0.122)(8,0.32)};

\addplot[color = magenta, mark = *, step = 1cm,very thin]coordinates{(1,0.058)(2,0.066)(3,0.038)(4,0.044)(5,0.066)(6,0.082)(7,0.162)(8,0.464)};

\addplot[color = green, mark = square*, step = 1cm,very thin]coordinates{(1,0.096)(2,0.1)(3,0.128)(4,0.176)(5,0.31)(6,0.46)(7,0.8)(8,0.97)};

\addplot[color = blue,  mark = diamond*, step = 1cm,very thin]coordinates{(1,0.062)(2,0.056)(3,0.062)(4,0.058)(5,0.086)(6,0.066)(7,0.124)(8,0.13)};

\addplot[color = orange, mark = square*, step = 1cm,very thin]coordinates{(1,0.092)(2,0.096)(3,0.106)(4,0.2)(5,0.368)(6,0.642)(7,0.972)(8,0.998)};

\addplot[color = teal, mark = diamond*, step = 1cm,very thin]coordinates{(1,0.056)(2,0.054)(3,0.072)(4,0.064)(5,0.12)(6,0.114)(7,0.216)(8,0.278)};

\addplot[color = yellow, mark = square*, step = 1cm,very thin]coordinates{(1,0.044)(2,0.068)(3,0.068)(4,0.138)(5,0.334)(6,0.71)(7,0.998)(8,1)};

\addplot[color = black, mark = diamond*, step = 1cm,very thin]coordinates{(1,0.056)(2,0.046)(3,0.044)(4,0.044)(5,0.068)(6,0.056)(7,0.046)(8,0.078)};
    \end{axis}
    \end{tikzpicture}
    
    \caption{Powers of  BD-${\ell_2}$ (\textcolor{red}{$\tikzcircle{2pt}$}), BD-${\ell_1}$ ({$\tikzcirclev{2pt}$}), BD-${\exp}$ ({$\tikzcirclep{2pt}$}), BD-${\log}$ ({$\tikzcirclem{2pt}$}), FR (\textcolor{green}{$\blacksquare$}), BF (\textcolor{blue}{$\blacklozenge$}), NN (\textcolor{orange}{$\blacksquare$}), MMD (\textcolor{teal}{$\blacklozenge$}), SHP (\textcolor{yellow}{$\blacksquare$}), BG (\textcolor{black}{$\blacklozenge$}) tests in Examples 9-11.}
\label{Sim3}
\end{figure}

Finally, we consider three examples (Examples 12-14) involving sparse alternatives, where $F$ and $G$ differ only in $\lfloor d^{\beta} \rfloor$
many coordinates for $\beta \in (0,1)$. Clearly, in these examples, we have $e_{h,\psi}(F,G)\asymp d^{\beta-1}$ for all choices of $\psi$ considered in this article. For our numerical study, we use $\beta=0.7$ and $n=m=5+\lfloor \sqrt{d}\rfloor$.
\vspace{0.05in}

\textbf{Example 12} Two multivariate normal distributions $F = \mathcal{N}_d(\muvec_d,{\bf I}_d)$ and $G = \mathcal{N}_d({\bf 0}_d, {\bf I}_d)$ differ  only in their locations. The first $[d^\beta]$ coordinates of $\muvec_d$ is $2$ and the rest are zero. 
\vspace{0.05in}

\textbf{Example 13}  Two normal distributions $F = \mathcal{N}_d({\bf 0}_d,{\bf I}_d)$ and $G = \mathcal{N}_d({\bf 0}_d, \sigmat_d)$ differ only in their scales. Here $\sigmat_d$ is a diagonal matrix with first $[d^\beta]$ entries are $5$ and the rest are $1$.
\vspace{0.1in}

\textbf{Example 14}  The distribution $G$ differs from $F= \mathcal{N}_d({\bf 0}_d,2{\bf I}_d)$ only in the first $[d^\beta]$ many coordinates. These coordinate variables are independent and they follow $t$ distribution with 4 degrees of freedom.
\vspace{0.05in}

Figure \ref{Sim4} shows the powers of different tests in these three examples. In the location and scale problems in Examples 12 and 13, our findings were similar to those observed in Examples 1 and 2, respectively. In Example 12, the BF test and the MMD test had an edge over the proposed tests, but the performances of the proposed tests were competitive with the rest. The SHP test had relatively low power. In Example 13, the BG test and the ball divergence tests outperformed their competitors in higher dimensions. In this scale problem, FR and NN tests had poor performance. Example 14 can be viewed as a sparse version of Example 8, where the two distributions differ only in their shapes. In this example, all tests based on the $\ell_2$ distance failed to perform well but the ball divergence tests based on generalized distances performed better. The powers of these tests showed an upward trend with increasing dimension.

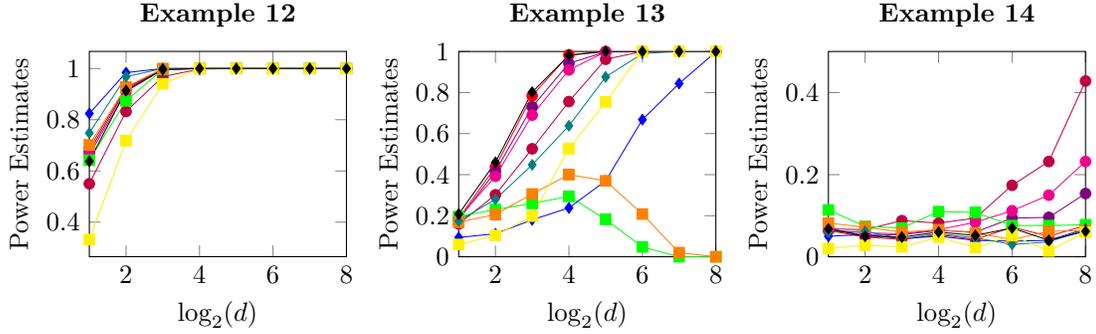
\begin{figure}[!h]
    \centering
   \begin{tikzpicture}
    \begin{axis}[xmin = 1, xmax = 8, xlabel = {$\log_2(d)$}, ylabel = {Power Estimates}, title = {\bf Example 12}]
\addplot[color = red,   mark = *, step = 1cm,very thin]coordinates{(1,0.688)(2,0.926)(3,0.996)(4,1)(5,1)(6,1)(7,1)(8,1)};

\addplot[color = violet, mark = *, step = 1cm,very thin]coordinates{(1,0.674)(2,0.92)(3,0.996)(4,1)(5,1)(6,1)(7,1)(8,1)};

\addplot[color = purple, mark = *, step = 1cm,very thin]coordinates{(1,0.55)(2,0.832)(3,0.968)(4,0.998)(5,1)(6,1)(7,1)(8,1)};

\addplot[color = magenta, mark = *, step = 1cm,very thin]coordinates{(1,0.662)(2,0.91)(3,0.996)(4,1)(5,1)(6,1)(7,1)(8,1)};

\addplot[color = green, mark = square*, step = 1cm,very thin]coordinates{(1,0.642)(2,0.874)(3,0.996)(4,1)(5,1)(6,1)(7,1)(8,1)};

\addplot[color = blue,  mark = diamond*, step = 1cm,very thin]coordinates{(1,0.824)(2,0.984)(3,1)(4,1)(5,1)(6,1)(7,1)(8,1)};

\addplot[color = orange, mark = square*, step = 1cm,very thin]coordinates{(1,0.702)(2,0.928)(3,1)(4,1)(5,1)(6,1)(7,1)(8,1)};

\addplot[color = teal, mark = diamond*, step = 1cm,very thin]coordinates{(1,0.748)(2,0.968)(3,1)(4,1)(5,1)(6,1)(7,1)(8,1)};

\addplot[color = yellow, mark = square*, step = 1cm,very thin]coordinates{(1,0.332)(2,0.718)(3,0.94)(4,1)(5,1)(6,1)(7,1)(8,1)};

\addplot[color = black, mark = diamond*, step = 1cm,very thin]coordinates{(1,0.638)(2,0.914)(3,0.996)(4,1)(5,1)(6,1)(7,1)(8,1)};
    \end{axis}
    \end{tikzpicture}
   \begin{tikzpicture}
     \begin{axis}[xmin = 1, xmax = 8, ymin = 0, ymax = 1,xlabel = {$\log_2(d)$}, ylabel = {Power Estimates}, title = {\bf Example 13}]
\addplot[color = red,   mark = *, step = 1cm,very thin]coordinates{(1,0.186)(2,0.438)(3,0.784)(4,0.984)(5,1)(6,1)(7,1)(8,1)};

\addplot[color = violet, mark = *, step = 1cm,very thin]coordinates{(1,0.18)(2,0.416)(3,0.73)(4,0.942)(5,0.998)(6,1)(7,1)(8,1)};

\addplot[color = purple, mark = *, step = 1cm,very thin]coordinates{(1,0.16)(2,0.302)(3,0.526)(4,0.756)(5,0.962)(6,1)(7,1)(8,1)};

\addplot[color = magenta, mark = *, step = 1cm,very thin]coordinates{(1,0.19)(2,0.392)(3,0.69)(4,0.91)(5,0.998)(6,1)(7,1)(8,1)};

\addplot[color = green, mark = square*, step = 1cm,very thin]coordinates{(1,0.196)(2,0.23)(3,0.26)(4,0.294)(5,0.182)(6,0.048)(7,0)(8,0)};

\addplot[color = blue,  mark = diamond*, step = 1cm,very thin]coordinates{(1,0.094)(2,0.112)(3,0.18)(4,0.238)(5,0.37)(6,0.668)(7,0.844)(8,0.998)};

\addplot[color = orange, mark = square*, step = 1cm,very thin]coordinates{(1,0.168)(2,0.206)(3,0.306)(4,0.4)(5,0.37)(6,0.208)(7,0.02)(8,0)};

\addplot[color = teal, mark = diamond*, step = 1cm,very thin]coordinates{(1,0.174)(2,0.28)(3,0.448)(4,0.638)(5,0.876)(6,0.99)(7,1)(8,1)};

\addplot[color = yellow, mark = square*, step = 1cm,very thin]coordinates{(1,0.06)(2,0.104)(3,0.198)(4,0.526)(5,0.754)(6,1)(7,1)(8,1)};

\addplot[color = black, mark = diamond*, step = 1cm,very thin]coordinates{(1,0.208)(2,0.46)(3,0.802)(4,0.982)(5,01)(6,1)(7,1)(8,1)};
    \end{axis}
    \end{tikzpicture}
    \begin{tikzpicture}
     \begin{axis}[xmin = 1, xmax = 8, ymin = 0, ymax = 0.5,xlabel = {$\log_2(d)$}, ylabel = {Power Estimates}, title = {\bf Example 14}]
\addplot[color = red,   mark = *, step = 1cm,very thin]coordinates{(1,0.064)(2,0.048)(3,0.042)(4,0.054)(5,0.042)(6,0.072)(7,0.05)(8,0.076)};

\addplot[color = violet, mark = *, step = 1cm,very thin]coordinates{(1,0.068)(2,0.052)(3,0.054)(4,0.064)(5,0.06)(6,0.094)(7,0.096)(8,0.154)};

\addplot[color = purple, mark = *, step = 1cm,very thin]coordinates{(1,0.068)(2,0.066)(3,0.088)(4,0.082)(5,0.094)(6,0.174)(7,0.232)(8,0.428)};

\addplot[color = magenta, mark = *, step = 1cm,very thin]coordinates{(1,0.066)(2,0.058)(3,0.056)(4,0.066)(5,0.084)(6,0.112)(7,0.15)(8,0.232)};

\addplot[color = green, mark = square*, step = 1cm,very thin]coordinates{(1,0.114)(2,0.074)(3,0.07)(4,0.11)(5,0.108)(6,0.076)(7,0.076)(8,0.078)};

\addplot[color = blue,  mark = diamond*, step = 1cm,very thin]coordinates{(1,0.05)(2,0.054)(3,0.044)(4,0.05)(5,0.04)(6,0.038)(7,0.04)(8,0.068)};

\addplot[color = orange, mark = square*, step = 1cm,very thin]coordinates{(1,0.082)(2,0.072)(3,0.06)(4,0.066)(5,0.058)(6,0.042)(7,0.062)(8,0.064)};

\addplot[color = teal, mark = diamond*, step = 1cm,very thin]coordinates{(1,0.06)(2,0.062)(3,0.044)(4,0.056)(5,0.046)(6,0.03)(7,0.036)(8,0.066)};

\addplot[color = yellow, mark = square*, step = 1cm,very thin]coordinates{(1,0.02)(2,0.028)(3,0.024)(4,0.048)(5,0.022)(6,0.052)(7,0.014)(8,0.06)};

\addplot[color = black, mark = diamond*, step = 1cm,very thin]coordinates{(1,0.068)(2,0.05)(3,0.048)(4,0.06)(5,0.052)(6,0.07)(7,0.04)(8,0.062)};
    \end{axis}
    \end{tikzpicture}











  \caption{Powers of  BD-${\ell_2}$ (\textcolor{red}{$\tikzcircle{2pt}$}), BD-${\ell_1}$ ({$\tikzcirclev{2pt}$}), BD-${\exp}$ ({$\tikzcirclep{2pt}$}), BD-${\log}$ ({$\tikzcirclem{2pt}$}), FR (\textcolor{green}{$\blacksquare$}), BF (\textcolor{blue}{$\blacklozenge$}), NN (\textcolor{orange}{$\blacksquare$}), MMD (\textcolor{teal}{$\blacklozenge$}), SHP (\textcolor{yellow}{$\blacksquare$}), BG (\textcolor{black}{$\blacklozenge$}) tests in Examples 12-14.}
    \label{Sim4}
\end{figure}


\vspace{-0.1in}
\subsection{Analysis of benchmark data sets}
\label{Real Data Analysis}

For further evaluation of the performance of different tests, we analyzed two real data sets, namely the Colon data and the Lightning-2 data. The Colon data set contains expression levels of 2000 genes in 40 `tumor' and 22 `normal' colon tissue samples that were analyzed with an Affymetrix oligonucleotide array. This data set is available in the R package `rda' and its description  can be found in \cite{alon1999broad}. The Lightning-2 data set is available at the UCR Time Series Classification Archive (\href{https://www.cs.ucr.edu/~eamonn/time_series_data_2018/}{\text{https://www.cs.ucr.edu/~eamonn/time\_series\_data} \text{\_2018/}}). It contains 637-dimensional observations from two populations with respective sample sizes being 48 and 73. Descriptions of this data set can be found in \cite{sarkar2020some}. These two data sets have been extensively studied in the classification literature, and it is well known that in each of these data sets, there is a reasonable separation between the two distributions. So, we can assume the alternative hypothesis $H_1:F \neq G$ to be true, and different tests can be compared based on their powers. However, when we used the full data set for testing, all tests rejected $H_0$ both in Colon and Lightning-2 data. Using that single experiment based on the full data set, it was not possible to compare different test procedures.  So, we generated random sub-samples from the entire data set, keeping the sample proportions from the two distributions approximately the same as they were in the original data. Different tests were implemented using these sub-samples, and this procedure was repeated 500 times to estimate their powers. The results for different sub-sample sizes are reported in Figure \ref{real_data_analysis}.

\begin{figure}[h]
    \centering
    \begin{tikzpicture}
    \begin{axis}[xmin = 14.5, xmax = 31.5, xlabel = {Pooled Sample Size}, ylabel = {Power Estimates}, title = {\bf Colon data}]
\addplot[color = red,   mark = *, step = 1cm,very thin]coordinates{(15,0.982)(18,1)(23,1)(26,1)(31,1)};

\addplot[color = violet, mark = *, step = 1cm,very thin]coordinates{(15,0.948)(18,0.998)(23,1)(26,1)(31,1)};

\addplot[color = purple, mark = *, step = 1cm,very thin]coordinates{(15,0.946)(18,0.998)(23,1)(26,1)(31,1)};

\addplot[color = magenta, mark = *, step = 1cm,very thin]coordinates{(15,0.954)(18,0.998)(23,1)(26,1)(31,1)};

\addplot[color = green, mark = square*, step = 1cm,very thin]coordinates{(15,0.846)(18,0.9)(23,0.954)(26,0.984)(31,1)};

\addplot[color = blue,  mark = diamond*, step = 1cm,very thin]coordinates{(15,0.992)(18,1)(23,1)(26,1)(31,1)};

\addplot[color = orange, mark = square*, step = 1cm,very thin]coordinates{(15,0.948)(18,0.998)(23,1)(26,1)(31,1)};

\addplot[color = teal, mark = diamond*, step = 1cm,very thin]coordinates{(15,0.992)(18,1)(23,1)(26,1)(31,1)};

\addplot[color = yellow, mark = square*, step = 1cm,very thin]coordinates{(15,0.654)(18,0.746)(23,0.808)(26,0.930)(31,0.978)};

\addplot[color = black, mark = diamond*, step = 1cm,very thin]coordinates{(15,0.666)(18,0.806)(23,0.954)(26,0.996)(31,1)};
    \end{axis}
    \end{tikzpicture}
    \begin{tikzpicture}
    \begin{axis}[xmin = 11, xmax = 76, xlabel = {Pooled Sample Size}, ylabel = {Power Estimates}, title = {\bf Lightning-2 data}]
\addplot[color = red,   mark = *, step = 1cm,very thin]coordinates{(12,0.114)(25,0.296)(37,0.418)(50,0.518)(62,0.680)(75,0.754)};

\addplot[color = violet, mark = *, step = 1cm,very thin]coordinates{(12,0.168)(25,0.638)(37,0.870)(50,0.970)(62,1)(75,1)};

\addplot[color = purple, mark = *, step = 1cm,very thin]coordinates{(12,0.206)(25,0.686)(37,0.916)(50,0.982)(62,1)(75,1)};

\addplot[color = magenta, mark = *, step = 1cm,very thin]coordinates{(12,0.184)(25,0.658)(37,0.888)(50,0.974)(62,1)(75,1)};

\addplot[color = green, mark = square*, step = 1cm,very thin]coordinates{(12,0.284)(25,0.314)(37,0.378)(50,0.526)(62,0.626)(75,0.804)};

\addplot[color = blue,  mark = diamond*, step = 1cm,very thin]coordinates{(12,0.16)(25,0.22)(37,0.322)(50,0.452)(62,0.574)(75,0.716)};

\addplot[color = orange, mark = square*, step = 1cm,very thin]coordinates{(12,0.146)(25,0.264)(37,0.418)(50,0.604)(62,0.736)(75,0.860)};

\addplot[color = teal, mark = diamond*, step = 1cm,very thin]coordinates{(12,0.148)(25,0.212)(37,0.324)(50,0.44)(62,0.552)(75,0.706)};

\addplot[color = yellow, mark = square*, step = 1cm,very thin]coordinates{(12,0.08)(25,0.242)(37,0.304)(50,0.482)(62,0.682)(75,0.72)};

\addplot[color = black, mark = diamond*, step = 1cm,very thin]coordinates{(12,0.112)(25,0.126)(37,0.112)(50,0.094)(62,0.114)(75,0.092)};
    \end{axis}
    \end{tikzpicture}
      \caption{Powers of  BD-${\ell_2}$ (\textcolor{red}{$\tikzcircle{2pt}$}), BD-${\ell_1}$ ({$\tikzcirclev{2pt}$}), BD-${\exp}$ ({$\tikzcirclep{2pt}$}), BD-${\log}$ ({$\tikzcirclem{2pt}$}), FR (\textcolor{green}{$\blacksquare$}), BF (\textcolor{blue}{$\blacklozenge$}), NN (\textcolor{orange}{$\blacksquare$}), MMD (\textcolor{teal}{$\blacklozenge$}), SHP (\textcolor{yellow}{$\blacksquare$}), BG (\textcolor{black}{$\blacklozenge$}) tests in Colon and Lightning-2 data sets.}
    \label{real_data_analysis}
\end{figure}
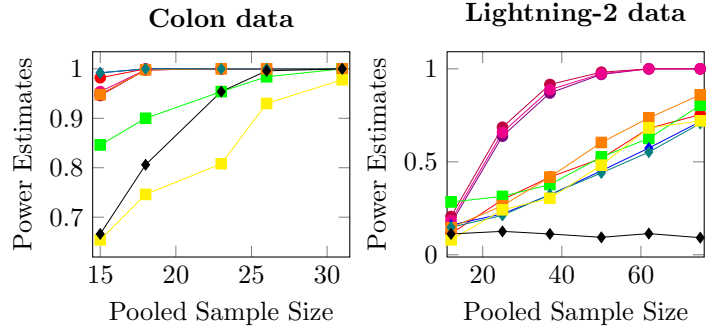


In the case of Colon data, BF, MMD, and BD-${\ell_2}$ tests had very high power even when the pooled sample size was $15$. These three tests had comparable performance, and they performed better than all other tests considered here. The ball divergence tests based on generalized distances also had  competitive performances. Like BF, MMD, and NN tests, they also had unit power for samples with a pooled sample size of 18 or higher. FR, BG, and SHP tests  had relatively low powers in this data set. 

Figure \ref{real_data_analysis} clearly shows the superiority of the ball divergence tests BD-${\ell_1}$, BD-${\exp}$ and BD-${\log}$ in the case of Lightning-2 data. These three tests had
much higher powers compared to the rest for samples with a combined sample size larger than 20. Among the other methods, the NN test had the best overall performance. BF, MMD, FR, SHP, and BD-${\ell_2}$ tests also had similar powers. 
The BG test performed poorly in this data set.

\section{Conclusion remarks}

This article investigates the high dimensional behavior of some tests based on \text{ball divergence}. Under appropriate regularity conditions, we have established the consistency of these tests in the HDLSS regime. Unlike many existing tests, the ball divergence tests based on the generalized distance functions can discriminate between two high-dimensional distributions differing outside the first two moments. For suitable choices of the distance function, the resulting tests can perform well even when the underlying distributions are heavy-tailed or there are outliers in the data. 
We have proved the minimax rate optimality of the proposed tests over a certain class of alternatives associated with ball divergence.  
This helps us in establishing the consistency of our tests for a broader class of alternatives, including the shrinking alternatives, when the sample sizes also increase with the dimension. Analyzing several simulated and real data sets, we have amply demonstrated that the proposed tests can outperform the start-of-the-art tests in a wide variety of high-dimensional two-sample problems. 

The proposed tests can be generalized to multivariate $k$-sample problems as well. Note that the variance of the probability  measures of a ball corresponding to different distributions $F_1, F_2,\ldots, F_k$ gives us some idea about the difference among the $F_i$'s. So, the expectation of this variance over the balls with random centers and radius can be  used as a generalized ball divergence measure for a $k$-sample problem. One can construct an estimate of this measure and develop a test based on it. The high dimensional behavior of the resulting test can also be investigated following the ideas given in this article.  

Recently, \cite{pan2020bcov} used the notion of ball divergence to come up with a measure of dependence among several Banach-valued random variables. This dependency measure, which is known as the ball covariance, was used for testing independence among those variables. The theory presented in this article can be used to study the high dimensional behavior of that test, particularly when the sample sizes increase with the dimension. The consistency of the proposed test
can also be proved under appropriate regularity conditions. However, it is not clear what is the minimax rate of separation for that problem and whether the test is minimax rate optimal. This can be considered an interesting problem for future research.

\bibliographystyle{apalike}
\small
\bibliography{refs}

\section*{Appendix: Proofs and mathematical details}

\begin{lemmaA}
If ${\bf X}_1,{\bf X}_2,\ldots, {\bf X}_n \stackrel{iid}{\sim}F$ and ${\bf Y}_1,{\bf Y}_2,\ldots,{\bf Y}_m\stackrel{iid}{\sim}G$ are independent, then
$$\E\{T_{n,m}^\rho\} = \frac{1}{6}\left(\frac{1}{n-2}+\frac{1}{m-2}\right)+\frac{1}{m}(p_{0}-p_{1})+\frac{1}{n}(p_{2}-p_{3})+\Theta_\rho^2(F,G), $$
where $\Theta_\rho^2(F,G)$ is the ball divergence measure defined in Section \ref{The proposed test}, and $p_0,p_1,p_2,p_3$ are given by
\begin{align*}
    p_0 & = \P\big\{\rho(Y_1,X_1)\leq \rho(X_2,X_1)\big\},\hspace{20pt}p_1 = \P\big\{\rho(Y_1,X_1)\leq \rho(X_2,X_1); \rho(Y_2,X_1)\leq \rho(X_2,X_1)\big\},\\
    p_2 & = \P\big\{\rho(X_1,Y_1)\leq \rho(Y_2,Y_1)\big\} ~~\mbox{ and }\hspace{5pt}p_3 = \P\big\{\rho(X_1,Y_1)\leq \rho(Y_2,Y_1); \rho(X_2,Y_1)\leq \rho(Y_2,Y_1)\big\}
\end{align*}
\label{estprop}
\end{lemmaA}

\vspace{-0.1in}
\begin{proof}[\bf Proof.]
\vspace{-0.2in}
Note that $T_{n,m}^\rho$ can be written as $T_{n,m}^\rho  =V_1+V_2$, where
$$V_1 = \frac{1}{n(n-1)}\sum_{1\leq i\not= j\leq n}\bigg\{\frac{1}{n-2}\sum_{k=1, k\not= i,j}^n\delta({\bf X}_k,{\bf X}_j,{\bf X}_i)-\frac{1}{m}\sum_{k=1}^{m}\delta({\bf Y}_k,{\bf X}_j,{\bf X}_i)\bigg\}^2,$$
$$V_2 = \frac{1}{m(m-1)}\sum_{1\leq i\not= j\leq m}\bigg\{\frac{1}{n}\sum_{k=1}^n\delta({\bf X}_k,{\bf Y}_{j},{\bf Y}_{i})-\frac{1}{m-2}\sum_{k=1, k\not = i,j}^{m}\delta({\bf Y}_{k},{\bf Y}_{j},{\bf Y}_{i})\bigg\}^2,$$
and $\delta({\bf s,u,v}) = \mathbbm{1}[\rho({\bf s,v})\leq \rho({\bf u,v})]$. Therefore, we have \begin{align*}
    \E\{V_1\} =&  ~\frac{1}{(n-2)^2}\E\Big\{\big(\sum_{k=1, k\not = 1,2}^n\delta({\bf X}_k,{\bf X}_2,{\bf X}_1)\big)^2\Big\}+\frac{1}{m^2}\E\Big\{\big(\sum_{k=1}^{m}\delta({\bf Y}_{k},{\bf X}_2,{\bf X}_1)\big)^2\Big\}\\
    & -\frac{2}{m(n-2)}\E\Big\{\big(\sum_{k=1, k\not = 1,2}^n\delta({\bf X}_k,{\bf X}_2,{\bf X}_1)\big)\big(\sum_{k=1}^{m}\delta({\bf Y}_{k},{\bf X}_2,{\bf {\bf X}}_1)\big)\Big\}\\
    =& ~\frac{1}{(n-2)^2}\E\Big\{\sum_{k=1, k\not = 1,2}^n\delta({\bf X}_k,{\bf X}_2,{\bf X}_1) + \sum_{k,l=1, k,l\not = 1,2, k\not = l}^n\delta({\bf X}_k,{\bf X}_2,{\bf X}_1)\delta({\bf X}_l,{\bf X}_2,{\bf X}_1)\Big\}\\
    & + \frac{1}{m^2}\E\Big\{\sum_{k=1, k\not = 1,2}^m\delta({\bf Y}_{k},{\bf X}_2,{\bf X}_1) + \sum_{k,l=1, k,l\not = 1,2, k\not = l}^m\delta({\bf Y}_{k},{\bf X}_2,{\bf X}_1)\delta({\bf Y}_{l},{\bf X}_2,{\bf X}_1)\Big\}\\
    &-\frac{2}{m(n-2)}\E\Big\{\sum_{k=1, k\not = 1,2}^n\sum_{l=1}^{m}\delta({\bf X}_k,{\bf X}_2,{\bf X}_1)\delta({\bf Y}_{l},{\bf X}_2,{\bf {\bf X}}_1)\Big\}\\
    =& ~\frac{1}{n-2}~\P\Big\{\rho({\bf X}_3,{\bf X}_1)\leq\rho({\bf X}_2,{\bf X}_1)\Big\}\\ & +\frac{n-3}{n-2}~\P\Big\{\rho({\bf X}_3,{\bf X}_1)\leq\rho({\bf X}_2,{\bf X}_1);\rho({\bf X}_4,{\bf X}_1)\leq\rho({\bf X}_2,{\bf X}_1)\Big\}\\
    & + \frac{1}{m} ~\P\Big\{\rho({\bf Y}_1,{\bf X}_1)\leq \rho({\bf X}_2,{\bf X}_1)\Big\}\\ 
    & + \frac{m-1}{m}~\P\Big\{\rho({\bf Y}_1,{\bf X}_1)\leq \rho({\bf X}_2,{\bf X}_1);\rho({\bf Y}_2,{\bf X}_1)\leq \rho({\bf X}_2,{\bf X}_1)\Big\} \\
    &-2~\P\Big\{\rho({\bf X}_3,{\bf X}_1)\leq \rho({\bf X}_2,{\bf X}_1);\rho({\bf Y}_1,{\bf X}_1)\leq \rho({\bf X}_2,{\bf X}_1)\Big\}\\
    =& ~\frac{1}{(n-2)}\Big\{\frac{1}{2} +(n-3)\frac{1}{3}\Big\}+\frac{1}{m}\Big\{p_0+ (m-1)p_1\Big\}-2p_4\\
    =& ~\frac{1}{3}+\frac{1}{6(n-2)}+\frac{1}{m}(p_{0}-p_{1})+p_{1}-2p_{4},
\end{align*}
where  $p_0$ and $p_1$ are as defined before and $p_4 = \P\{\rho({\bf X}_3,{\bf X}_1)\leq \rho({\bf X}_2,{\bf X}_1);\rho({\bf Y}_1,{\bf X}_1)\leq \rho({\bf X}_2,{\bf X}_1)\}$. Similarly one can show that
\begin{align*}
    \E\{V_2\} = \frac{1}{3}+\frac{1}{6(m-2)}+\frac{1}{n}(p_{2}-p_{3})+p_{3}-2p_{5},
\end{align*}
where $p_{2}$ and $p_{3}$ are as defined before and $p_{5} = \P\{\rho({\bf Y}_3,{\bf Y}_1)\leq \rho({\bf Y}_2,{\bf Y}_1);\rho({\bf X}_1,{\bf Y}_1)\leq \rho({\bf Y}_2,{\bf Y}_1)\}$. Hence, we have
$$
\E\{T_{n,m}^\rho\} = \frac{1}{6(n-2)}+\frac{1}{6(m-2)}+\frac{1}{m}(p_{0}-p_{1})+\frac{1}{n}(p_{2}-p_{3})+\frac{2}{3}+p_{1}-2p_{4}+p_{3}-2p_{5}.
$$ 

Now observe that
\begin{equation*}
    \begin{split}
        \Theta_\rho^2(F,G) = &~\int \left\{F(B(u,\rho(v,u)))-G(B(u,\rho(v,u)))\right\}^2\,[dF(u)dF(v)+dG(u)dG(v)]\\
       = & ~\int F^2(B(u,\rho(v,u))dF(u)dF(v)-2\int F(B(u,\rho(v,u))G(B(u,\rho(v,u))dF(u)dF(v)\\
        & +\int G^2(B(u,\rho(v,u))dF(u)dF(v)+\int F^2(B(u,\rho(v,u))dG(u)dG(v)\\
        & -2\int F(B(u,\rho(v,u))G(B(u,\rho(v,u))dG(u)dG(v)+\int G^2(B(u,\rho(v,u))dG(u)dG(v)\\
        = & ~\frac{1}{3}-2p_4+p_1+p_3-2p_5+\frac{1}{3}
        = ~p_{1}+p_{3}-2p_{4}-2p_{5}  +\frac{2}{3}.
    \end{split}
\end{equation*}

Hence, we have $~\E\{T_{n,m}^\rho\} = \frac{1}{6}\left(\frac{1}{n-2}+\frac{1}{m-2}\right)+\frac{1}{m}(p_{0}-p_{1})+\frac{1}{n}(p_{5}-p_{3})+\Theta_\rho^2(F,G).$ 
\end{proof}

\begin{lemmaA}
If ${\bf X}_1,{\bf X}_2,\ldots, {\bf X}_n \stackrel{iid}{\sim}F$ and ${\bf Y}_1,{\bf Y}_2,\ldots,{\bf Y}_m\stackrel{iid}{\sim}G$ are independent, then
$$ Var(T_{n,m}^\rho) \leq C_1 \Theta_\rho^2(F,G)\Bigg(\frac{1}{n}+\frac{1}{m}\Bigg)+C_2\Bigg(\frac{1}{n}+\frac{1}{m}\Bigg)^2, $$
\noindent
where the constants $C_1$ and $C_2$ are independent of the dimensions $d$. 
\label{varbound}
\end{lemmaA}

\begin{proof}[\bf Proof]
Note that $\Theta_\rho^2(F,G)$ can be written as
$\Theta_\rho^2(F,G)=A_1+A_2$,
where
\begin{equation*}
    \addtolength{\jot}{-0.5em}
\begin{split}
        A_1 &= \int \int \{F({\mathbb B}({\bf u},\rho({\bf v},{\bf u}))-G({\mathbb B}({\bf u},\rho({\bf v},{\bf u}))\}^2 dF({\bf u})dF({\bf v})\\
        &~~~~~~~~~~~~~~~~~~~~~~~= \E\Big(F(B({\bf {\bf X}}_1,\rho({\bf X}_2,{\bf X}_1)))-G(B({\bf X}_1,\rho({\bf X}_2,{\bf X}_1)))\Big)_,^2 \mbox{ and}
\end{split}
\end{equation*}
\begin{equation*}
    \addtolength{\jot}{-0.5em}
\begin{split}
        A_2 &= \int \int \{F({\mathbb B}({\bf u},\rho({\bf v},{\bf u}))-G({\mathbb B}({\bf u},\rho({\bf v},{\bf u}))\}^2 dG({\bf u})dG({\bf v}) 
        \\&~~~~~~~~~~~~~~~~~~~~~~~= \E\Big(F(B({\bf Y}_1,\rho({\bf Y}_2,{\bf Y}_1)))-G(B({\bf Y}_1,\rho({\bf Y}_2,{\bf Y}_1)))\Big)^2.
    \end{split}
\end{equation*}
It can be verified that $V_1$ (as defined in the proof of Lemma A.1) can be expressed as
\begin{equation*}
    \begin{split}
        V_1 & = \frac{1}{n(n-1)}\sum_{1\leq i\not=j\leq n}\Big\{\frac{1}{n-2}\sum_{k=1, k\not=i,j}^n\delta({\bf X}_k,{\bf X}_j,{\bf X}_i)-\frac{1}{m}\sum_{k=1}^m\delta({\bf Y}_k,{\bf X}_j,{\bf X}_i)\Big\}^2\\
        & = \frac{1}{n(n-1)(n-2)^2m^2}\sum_{i\not= j = 1}^n\sum_{u, u'\not=i,j}^n \sum_{v,v'=1}^m\Big\{\delta({\bf X}_{u},{\bf X}_j,{\bf X}_i)\delta({\bf X}_{u'},{\bf X}_j,{\bf X}_i)\\
        & + \delta({\bf Y}_v,{\bf X}_j,{\bf X}_i)\delta({\bf Y}_{v'},{\bf X}_j,{\bf X}_i)-\delta({\bf X}_{u'},{\bf X}_j,{\bf X}_i)\delta(Y_v,{\bf X}_j,{\bf X}_i)-\delta({\bf X}_{u},{\bf X}_j,{\bf X}_i)\delta({\bf Y}_2,{\bf X}_j,{\bf X}_i)\Big\}\\
        & = \frac{1}{n(n-1)(n-2)^2m^2}\sum_{i\not= j = 1}^n\sum_{u, u'\not=i,j}^n \sum_{v,v'=1}^m \psi_{A_1}({\bf X}_i,{\bf X}_j,{\bf X}_u,{\bf X}_{u'};{\bf Y}_{v},{\bf Y}_{v'}), \mbox{ say}.
    \end{split}
\end{equation*}

Clearly, $V_1$ can be written as a linear combination of U-statistics of different degrees. Let $\hat U_{A_1}^{(4,2)}$ be the U-statistic with the kernel $\psi_{A_1}({\bf X}_i,{\bf X}_j,{\bf X}_u,{\bf X}_{u'};{\bf Y}_{v},{\bf Y}_{v'})$, which has the same degree as $V_1$. So, it determines the order of $Var(V_1)$. More specifically, we get
$$        V_1 = \frac{1}{(n-2)^2m^2}\left\{4{n-2\choose 2}{m\choose 2}\hat U_{A_1}^{(4,2)}\right\}+O_P\left(\Big(\frac{1}{n}+\frac{1}{m}\Big)\right) \mbox{ and}$$  $$ Var(V_1) = \frac{\sigma_{1,0}^2(A_1)}{n}+\frac{\sigma_{0,1}^2(A_1)}{m}+C_2\left(\Big(\frac{1}{n}+\frac{1}{m}\Big)^2\right).
  $$
Note that here $\sigma_{1,0}^2(A_1) = Var(\psi_{A_1,1,0}^s({\bf X}))$ and $\sigma_{0,1}^2(A_1) = Var(\psi_{A_1,0,1}^s({\bf X}))$, where $\psi_{A_1,1,0}^s({\bf x})=E(\psi_{A_1}({\bf x},{\bf X}_2,{\bf X}_3,{\bf X}_{4};{\bf Y}_{1},{\bf Y}_{2} )$ and
$\psi_{A_1,0,1}^s({\bf y})=E(\psi_{A_1}({\bf X}_1,{\bf X}_2,{\bf X}_3,{\bf X}_{4};{\bf y},{\bf Y}_{2})$. Since $\psi_{A_1}$ is uniformly bounded, the proportionality constant $C_2$ does not depend on $d$. Now,
\begin{equation*}
    \begin{split}
        \psi_{A_1,1,0}^s({\bf x})  =& ~\E\Big\{\frac{1}{4!2!}\hspace{-0.1in}\sum_{\pi\in\pi(1,2,3,4)}\sum_{\gamma\in\gamma(1,2)}\psi_{A_1}({\bf X}_{\pi(1)},{\bf X}_{\pi(2)},{\bf X}_{\pi(3)},{\bf X}_{\pi(4)};{\bf Y}_{\gamma(1)},{\bf Y}_{\gamma(2)})\mid {\bf X}_1 = {\bf x}\Big\}\\
         = &~\frac{1}{4}\E\Big\{\psi_{A_1}({\bf x},{\bf X}_2,{\bf X}_3,{\bf X}_4;{\bf Y}_1,{\bf Y}_2)+\psi_{A_1}({\bf X}_i,{\bf x},{\bf X}_3,{\bf X}_4;{\bf Y}_1,{\bf Y}_2)\\
        &~~~~~~~~~~~~~~~~~~~~~~~~~~+ \psi_{A_1}({\bf X}_i,{\bf X}_2,{\bf x},{\bf X}_4;{\bf Y}_1,{\bf Y}_2)+\psi_{A_1}({\bf X}_1,{\bf X}_2,{\bf X}_3,{\bf x};{\bf Y}_1,{\bf Y}_2)\Big\}\\
        = & ~\frac{1}{4}\left\{\E\Big(F(B({\bf x},\rho({\bf X}_2,{\bf x})))-G(B({\bf x},\rho({\bf X}_2,{\bf x})))\Big)^2\right\} \\
        & + \frac{1}{4}\left\{\E\Big(F(B({\bf X}_1,\rho({\bf x},{\bf X}_1)))-G(B({\bf X}_1,\rho({\bf x},{\bf X}_1)))\Big)^2\right\}\\
        &+ \frac{1}{2}\E\Big(\delta({\bf x},{\bf X}_2,{\bf X}_1)-G(B({\bf X}_1,\rho({\bf X}_2,{\bf X}_1)))\Big)\\
        &~~~~~~~~~~~~~~~~~~~~~~~~~~\times \Big(F(B({\bf X}_1,\rho({\bf X}_2,{\bf X}_1)))-G(B({\bf X}_1,\rho({\bf X}_2,{\bf X}_1)))\Big)\\
        & = g_1({\bf x})+g_2({\bf x})+g_3({\bf x}),~\mbox{ say}.
    \end{split}
\end{equation*}
So, we have $\sigma^2_{1,0}=
\E\big\{\psi_{A_1,1,0}^s({\bf X}) - \E(\psi_{A_1,1,0}^s({\bf X}))\big\}^2  = \E\big\{\psi_{A_1,1,0}^s({\bf X}) - A_1\big\}^2= \E\big\{g_1({\bf X}) +g_2({\bf X})+g_3({\bf X})- A_1\big\}^2$. So, using the inequality, $E(\sum_{i=1}^{p}Z_1)^2 \le p~E(\sum_{i=1}^{p}Z_i^2)$ and the fact that 
$0 \le g_1({\bf x}), g_2({\bf x}) \le 1/4$ for all ${\bf x}$, we get 
$$\sigma_{1,0}^2 \leq 4\bigg\{\E g_1^2({\bf X}) + \E g_2^2({\bf X}) + \E g_3^2({\bf X}) + A_1^2\bigg\} \leq 4\bigg\{\frac{1}{4}\E g_1({\bf X})+ \frac{1}{4}\E g_2({\bf X})+ \E g_3^2({\bf X})+A_1\bigg\}.$$
Now note that $\E g_1({\bf X})=\E g_2({\bf X})=\frac{1}{4}A_1$. Also, using Cauchy-Schwartz inequality on $g_3({\bf X})$, we get
$$\E g_3^2({\bf X}) \leq \frac{1}{4}\E \Big(\delta({\bf X},{\bf X}_2,{\bf X}_1)-G(B({\bf X}_1,\rho({\bf X}_2,{\bf X}_1)))\Big)^2 A_1 \leq \frac{1}{4}A_1.$$
Hence, we have
$\sigma_{1,0}^2(A_1) \leq\frac{11}{2} A_1.$
Similarly, we can also show that $\sigma_{0,1}^2(A_1)\leq \frac{9}{2} A_1$. Combining these, we get 
$$Var(V_1)\leq \frac{11}{2} A_1\Big(\frac{1}{n}+\frac{1}{m}\Big)+C_2\left(\Big(\frac{1}{n}+\frac{1}{m}\Big)^2\right).$$

Using the same set of arguments, we also have
$$Var(V_2)\leq \frac{11}{2}A_2\Big(\frac{1}{n}+\frac{1}{m}\Big)+C_2\left(\Big(\frac{1}{n}+\frac{1}{m}\Big)^2\right).$$

This completes the proof.
\end{proof}

\begin{lemmaA}
Consider a random permutation $\pi$ of $\{1,2,\ldots, n+m\}$. If $T_{n,m,\pi}^\rho$ denotes the permuted test statistic (the  permutation analog of $T_{n,m}^{\rho}$), given the pooled sample $\mathcal{U}= \{{\bf U}_1,{\bf U}_2,\ldots,{\bf U}_{n+m}\}$, the conditional expectation $T_{n,m,\pi}^\rho$ is given by
$$\E\{T_{n,m,\pi}^\rho\mid \mathcal{U}\} = \frac{1}{6}\left(\frac{1}{n}+\frac{1}{m}+\frac{1}{n-2}+\frac{1}{m-2}\right).
\vspace{-0.2in}$$
\label{condexppermstat}
\end{lemmaA}

\noindent
\begin{proof}[\bf Proof]
For any random permutation $\pi$ of $\{1,2,\ldots, n+m\}$, $T_{n,m,\pi}^\rho$ can be written as
\begin{equation*}
    \begin{split}
        T_{n,m,\pi}^\rho  = \hspace{5pt}& \frac{1}{n(n-1)}\sum_{1\leq i\not= j\leq n}\bigg\{\frac{1}{n-2}\sum_{k=1, k\not= i,j}^n\delta({\bf U}_{\pi(k)},{\bf U}_{\pi(j)},{\bf U}_{\pi(i)})
        \\&~~~~~~~~~~~~~~~~~~~~~~~~~~~~~~~~~~~~~~~~~~~~~-\frac{1}{m}\sum_{k=n+1}^{n+m}\delta({\bf U}_{\pi(k)},{\bf U}_{\pi(j)},{\bf U}_{\pi(i)})\bigg\}^2\\
        & + \frac{1}{m(m-1)}\sum_{n+1\leq i\not= j\leq n+m}\bigg\{\frac{1}{n}\sum_{k=1}^n\delta({\bf U}_{\pi(k)},{\bf U}_{\pi(j)},{\bf U}_{\pi(i)})\\
        &~~~~~~~~~~~~~~~~~~~~~~~~~~~~~~~~~~~~~~~~~~~~~-\frac{1}{m-2}\sum_{k=n+1, k\not = i,j}^{n+m}\delta({\bf U}_{\pi(k)},{\bf U}_{\pi(j)},{\bf U}_{\pi(i)})\bigg\}^2,
    \end{split}
\end{equation*}
So, the conditional expectation of $T_{n,m,\pi}^\rho$ for any given $\mathcal{U}$ can be expressed as 
\begin{equation*}
    \begin{split}
        \E\{T_{n,m,\pi}^\rho\mid \mathcal{U}\} = & ~\E\Bigg\{\bigg\{\frac{1}{n-2}\sum_{k=3}^n\delta({\bf U}_{\pi(k)},{\bf U}_{\pi(2)},{\bf U}_{\pi(1)})
        \\
        &~~~~~~~~~~~~~~~~~~~~~~~~~~~~~~~~~~~~~~~-\frac{1}{m}\sum_{k=n+1}^{n+m}\delta({\bf U}_{\pi(k)},{\bf U}_{\pi(2)},{\bf U}_{\pi(1)})\bigg\}^2\Big|~~ \mathcal{U}\Bigg\}\\
        & + \E\Bigg\{\bigg\{\frac{1}{n}\sum_{k=1}^n\delta({\bf U}_{\pi(k)},{\bf U}_{\pi(2)},{\bf U}_{\pi(1)})\\
        &~~~~~~~~~~~~~~~~~~~~~~~~~~~~~~~~~~~~~~~-\frac{1}{m-2}\sum_{k=n+1, k\not = i,j}^{n+m}\delta({\bf U}_{\pi(k)},{\bf U}_{\pi(2)},{\bf U}_{\pi(1)})\bigg\}^2\Big|~~ \mathcal{U}\Bigg\}.\\
    \end{split}
\end{equation*}
Now note that
\begin{equation*}
    \begin{split}
        & \E\left\{\bigg\{\frac{1}{n-2}\sum_{k=3}^n\delta({\bf U}_{\pi(k)},{\bf U}_{\pi(2)},{\bf U}_{\pi(1)})-\frac{1}{m}\sum_{k=n+1}^{n+m}\delta({\bf U}_{\pi(k)},{\bf U}_{\pi(2)},{\bf U}_{\pi(1)})\bigg\}^2\Big|~~ \mathcal{U}\right\}\\
        =~& \frac{1}{(n-2)^2}\E\Big\{\sum_{k=3}^n \delta({\bf U}_{\pi(k)},{\bf U}_{\pi(2)},{\bf U}_{\pi(1)})\mid\mathcal{U}\Big\}\\
        & ~ +\frac{1}{(n-2)^2}\sum_{k,l=3, k\not=l}^n \E\Big\{\delta({\bf U}_{\pi(k)},{\bf U}_{\pi(2)},{\bf U}_{\pi(1)})\delta({\bf U}_{\pi(l)},{\bf U}_{\pi(2)},{\bf U}_{\pi(1)})\mid\mathcal{U}\Big\}\\
        & ~ +\frac{1}{m^2}\sum_{k=n+1}^{n+m}\E\big\{\delta({\bf U}_{\pi(k)},{\bf U}_{\pi(2)},{\bf U}_{\pi(1)})\mid\mathcal{U}\Big\}\\
        & ~ +\frac{1}{m^2}\sum_{k,l=n+1, k\not=l}^{n+m}\E\Big\{\delta({\bf U}_{\pi(k)},{\bf U}_{\pi(2)},{\bf U}_{\pi(1)})\delta({\bf U}_{\pi(l)},{\bf U}_{\pi(2)},{\bf U}_{\pi(1)})\mid\mathcal{U}\Big\}\\
        & ~-
        \frac{2}{m(n-2)}\sum_{k=3}^n\sum_{l=n+1}^{n+m}\E\Big\{\delta({\bf U}_{\pi(k)},{\bf U}_{\pi(2)},{\bf U}_{\pi(1)})\delta({\bf U}_{\pi(l)},{\bf U}_{\pi(2)},{\bf U}_{\pi(1)})\mid\mathcal{U}\Big\}\\
        =~& \frac{(n-2)q_1}{(n-2)^2}+\frac{(n-2)(n-3)q_2}{(n-2)^2}+\frac{mq_1}{m^2}+\frac{m(m-1)q_2}{m^2}-\frac{2m(n-2)q_2}{m(n-2)}\\
        =~& (q_1-q_2)\Big(\frac{1}{n-2}+\frac{1}{m}\Big), \mbox{ where}
    \end{split}
\end{equation*}
$q_1= \E\{\delta({\bf U}_{\pi(1)},{\bf U}_{\pi(2)},{\bf U}_{\pi(3)})\mid \mathcal{U}\}$, $q_2=\E\{\delta({\bf U}_{\pi(1)},{\bf U}_{\pi(2)},{\bf U}_{\pi(3)})\delta({\bf U}_{\pi(4)},{\bf U}_{\pi(2)},{\bf U}_{\pi(3)})\mid\mathcal{U}\}$. Similarly, we can also show that
\begin{align*}
    &\E\left\{\bigg\{\frac{1}{n}\sum_{k=1}^n\delta({\bf U}_{\pi(k)},{\bf U}_{\pi(2)},{\bf U}_{\pi(1)})-\frac{1}{m-2}\sum_{k=n+1, k\not = i,j}^{n+m}\delta({\bf U}_{\pi(k)},{\bf U}_{\pi(2)},{\bf U}_{\pi(1)})\bigg\}^2\Big|~~ \mathcal{U}\right\}\\ &~~~~~~~~~~~~~~= (q_1-q_2)\Big(\frac{1}{m-2}+\frac{1}{n}\Big).
\end{align*}
But given the pooled sample $\mathcal{U}$, the random variables $\{{\bf U}_{\pi(i)}\}_{i=1}^{n+m}$ are exchangeable. So, we must have $q_1 = 1/2$ and $q_2 = 1/3$. Hence, we have
$$\E\{T_{n,m,\pi}^\rho\mid \mathcal{U}\} = \frac{1}{6}\left(\frac{1}{n}+\frac{1}{m}+\frac{1}{n-2}+\frac{1}{m-2}\right)
\vspace{-0.2in}$$
\end{proof}



\noindent
\begin{proof}[\bf Proof of Lemma \ref{permutation-cut-off}]
Here, we are interested in the quantiles of the conditional distribution of $T_{n,m,\pi}^\rho$ given the pooled sample $\mathcal{U}$.
Since $T_{n,m,\pi}^\rho$ is non-negative, using Markov's inequality on the conditional random variable, we get
$$\P\left\{T_{n,m,\pi}^\rho\geq \frac{1}{\alpha} \E\{T_{n,m,\pi}^\rho\mid \mathcal{U}\}~\Big|~\mathcal{U}\right\} \leq \alpha.$$
Therefore, from the definition of the quantile $c_{1-\alpha}$, we have 
$c_{1-\alpha} \leq \frac{1}{\alpha} \E\{T_{n,m,\pi}^\rho\mid \mathcal{U}\},$
which holds with probability one. From Lemma A.\ref{condexppermstat},  we also have
$$\E\{T_{n,m,\pi}^\rho\mid \mathcal{U}\} = \frac{1}{6}\left(\frac{1}{n}+\frac{1}{m}+\frac{1}{n-2}+\frac{1}{m-2}\right) \leq \frac{2}{3 (\min\{n,m\}-2)}.$$
This completes the proof.
\end{proof}

\noindent
\begin{proof}[\bf Proof of Theorem \ref{constlargesam}]
In view of Lemma A.1 and Lemma A.2, as $\min\{n,m\}$ grows to infinity, $T_{n,m}^\rho$ converges in probability to $\Theta_\rho^2(F, G)$. So, if $\Theta^2_{\rho}(F,G)>0$, under $H_1:F\not=G$, $T_{n,m}^\rho$ converges in probability to a positive number. On the other hand, Lemma \ref{permutation-cut-off} shows that the cut-off value of the permutation test $c_{1-\alpha}$ goes to zero almost surely. Therefore, the power of the permutation test converges to one as $\min\{n,m\}$ grows to infinity.
\end{proof}


\noindent
\begin{proof}[\bf Proof of Lemma~\ref{randper}]
To prove this lemma, we shall use the idea of Corollary 6.1 of \cite{kim2021comparing}. First, let us define
$$H(t)= \frac{1}{N!}\left\{\sum_{\pi\in \mathcal{S}_N}\mathbbm{1}[T_{n,m,\pi_i}\leq t]\right\}\hspace{20pt}\text{ and }\hspace{20pt}H_B(t) = \frac{1}{B}\left\{\sum_{i=1}^B\mathbbm{1}[T_{n,m,\pi_i}\leq t]\right\}.$$
Here $H$ and $H_B$ are distribution functions conditioned on the observed pooled data $\mathcal{U}$. Now, we have
\begin{equation*}
    \begin{split}
        |p_{n,m}-p_{n,m,B}| & = \bigg|\frac{1}{N!}\Big\{\sum_{\pi\in \mathcal{S}_N}\mathbbm{1}[T_{n,m,\pi_i}\geq T_{n,m}]\Big\}-\frac{1}{B+1}\Big\{\sum_{i=1}^B\mathbbm{1}[T_{n,m,\pi_i}\geq T_{n,m}]+1\Big\}\bigg|\\
        & = \bigg|\frac{1}{N!}\Big\{\sum_{\pi\in \mathcal{S}_N}\mathbbm{1}[T_{n,m,\pi_i}\leq T_{n,m}]\Big\}-\frac{1}{B+1}\Big\{\sum_{i=1}^B\mathbbm{1}[T_{n,m,\pi_i}\leq T_{n,m}]+1\Big\}\bigg|\\
        &  = \bigg|H(T_{n,m})-\frac{B}{B+1}H_B(T_{n,m})-\frac{1}{B+1}\bigg|\\
        & \leq |H(T_{n,m})-H_B(T_{n,m})| + \bigg|\frac{H_B(T_{n,m})-1}{B+1}\bigg| \leq \sup_{t\in\R}|H(t)-H_B(t)|+\frac{1}{B+1}.
    \end{split}
\end{equation*}

Conditioned on the pooled data $\mathcal{U}$, the Dvoretzky-Keifer-Wolfwitz inequality \citep[see, e.g.,][]{massart1990tight} gives us
$\P\{\sup_{t\in\R}|H(t)-H_B(t)|>\epsilon\mid\mathcal{U}\}\leq 2e^{-2B\epsilon^2}.$ So, conditioned on $\mathcal{U}$, as $B$ grows to infinity,  $p_{n,m,B}$ converges almost surely to $p_{n,m}$. 
\end{proof}

\noindent
\begin{proof}[\bf Proof of Lemma \ref{inter-point-dist-accu}]
Let ${\bf W}=(W^{(1)},W^{(2)},\ldots,W^{(d)})$ denote any one of the random vectors ${\bf X}_1-{\bf X}_2,{\bf Y}_1-{\bf Y}_2$ or ${\bf X}_1-{\bf Y}_1$, where ${\bf X}_1,{\bf X}_2\sim F$,  ${\bf Y}_1,{\bf Y}_2\sim G$ and they are independent. Now,
\begin{align*}
        Var\Big(\frac{1}{d}|\|{\bf W}\|^2\Big) &= \frac{1}{d^2}\sum_{i=1}^d\sum_{j=1}^d Cov\Big((W^{(i)})^2,(W^{(j)})^2\Big)\\&= \frac{1}{d^2}\Big\{\sum_{i=1}^d Var\Big((W^{(i)})^2\Big)+\sum_{i=1}^d\sum_{j=1,j\not=i}^d Cov\Big((W^{(i)})^2,(W^{(j)})^2\Big)\Big\}.
\end{align*}
Under (A1), we can find a uniform bound $C$ for the $Var((W^{(i)})^{2})$'s. So, using (A1) and (A2), we get 
        $$Var\Big(\frac{1}{d}|\|{\bf W}\|^2\Big) \leq \frac{C}{d}+\frac{1}{d^2}\sum_{i=1}^d\sum_{j=1,j\not=i}^d Cov\Big((W^{(i)})^2,(W^{(j)})^2\Big) \to 0 \text{ as } d \rightarrow \infty .$$
This establishes that $\frac{1}{d}\Big|\|W\|^2-\E(\|W\|^2)\Big| \stackrel{P}{\rightarrow}{0}$ as $d \rightarrow \infty$. Now one can show that under (A3), $\frac{1}{d}E(\|W\|^2)$ 
converges to $2\sigma_F^2$, $2\sigma_F^2$ and $\sigma_F^2+\sigma_G^2+\nu^2$ for ${\bf W}={\bf X}_1-{\bf X}_2,{\bf Y}_1-{\bf Y}_2$ and ${\bf X}_1-{\bf Y}_1$, respectively. So, as $d \rightarrow \infty$, we have
$$d^{-1/2}\|{\bf X}_1-{\bf X}_2\|\stackrel{P}{\to}\sigma_F\sqrt{2},~ d^{-1/2}\|{\bf Y}_1-{\bf Y}_2\|\stackrel{P}{\to}\sigma_G\sqrt{2} \mbox{ and } d^{-1/2}\|{\bf X}_1-{\bf Y}_1\|\stackrel{P}{\to}\sqrt{\sigma_F^2+\sigma_G^2+\nu^2}.$$
\end{proof}




\begin{lemmaA}
\label{HDLSS}
Suppose that ${\bf X}_1,{\bf X}_2 {\sim}F$, ${\bf Y}_1,{\bf Y}_2 {\sim}G$ and they are independent. For a distance function $\rho$,
assume that  $\rho({\bf X}_1,{\bf X}_2)\stackrel{P}{\to} \theta_1$,  $\rho({\bf Y}_1,{\bf Y}_2)\stackrel{P}{\to}\theta_2$ and $\rho({\bf X}_1,{\bf Y}_2)\stackrel{P}{\to}\theta_3$ as $d \rightarrow \infty$.
If $\theta_3 >   \min\{\theta_1,\theta_2\}$, then $P(T_{n,m}^\rho>1/3) \rightarrow 1$ as $d$ diverges to infinity. 
\end{lemmaA}

\noindent
\begin{proof}[\bf Proof]
Note that $T_{n,m}^\rho$ involves the terms $\delta({\bf U}_k,{\bf U}_j,{\bf U}_i)$'s for different choices of ${\bf U}_i,{\bf U}_j,{\bf U}_k$ from the pooled sample. So, the behavior of $T_{n,m}^\rho$ can be studied using the convergence of the  $\delta({\bf U}_k,{\bf U}_j,{\bf U}_i)$'s. 

\vspace{0.05in}
First, consider the case
$\min\{\theta_1,\theta_2\}<\theta_3<\max\{\theta_1,\theta_2\}$.
Let us assume that $\theta_1> \theta_3> \theta_2$. In such a situation, we have
$\lim_{d\to\infty}\P[\rho({\bf X}_2,{\bf Y}_1)\leq \rho({\bf Y}_2,{\bf Y}_1)] = 0$ and $\lim_{d\to\infty}\P[\rho({\bf Y}_1,{\bf X}_1)\leq \rho({\bf X}_2,{\bf X}_1)] = 1.$
Hence, we get
\begin{equation*}
    \begin{split}
        V_1 =\hspace{5pt} & \frac{1}{n(n-1)}\sum_{1\leq i\not= j\leq n}\bigg\{\frac{1}{n-2}\sum_{k=1, k\not= i,j}^n\delta({\bf X}_k,{\bf X}_j,{\bf X}_i)-\frac{1}{m}\sum_{k=1}^{m}\delta({\bf Y}_{k},{\bf X}_j,{\bf X}_i)\bigg\}^2\\
         \stackrel{P}{\to} \hspace{5pt} & \frac{1}{n(n-1)}\sum_{1\leq i\not= j\leq n}\bigg\{\frac{1}{n-2}\sum_{k=1, k\not= i,j}^n\delta({\bf X}_k,{\bf X}_j,{\bf X}_i)-1\bigg\}^2\\
         = \hspace{5pt} & \frac{1}{n(n-1)}\sum_{1\leq i\not= j\leq n}\bigg\{\frac{-1}{n-2}\sum_{k=1, k\not= i,j}^n\mathbbm{1}[\rho({\bf X}_k,{\bf X}_i)>\rho({\bf X}_j,{\bf X}_i)]\bigg\}^2\\
         =\hspace{5pt} & \frac{1}{n(n-1)(n-2)^2}\bigg\{\sum_{1\leq i\not=j\not=k\leq n}\mathbbm{1}[\rho({\bf X}_k,{\bf X}_i)>\rho({\bf X}_j,{\bf X}_i)]\\
         & \hspace{5pt}+ \sum_{1\leq i\not=j\leq n}\sum_{k=1, k\not = i,j}^n\sum_{l=1, l\not = i,j,k}^n\mathbbm{1}[\rho({\bf X}_k,{\bf X}_i)>\rho({\bf X}_j,{\bf X}_i)]\mathbbm{1}[\rho({\bf X}_l,{\bf X}_i)>\rho({\bf X}_j,{\bf X}_i)]\bigg\}\\
         = \hspace{5pt} & \frac{1}{n(n-1)(n-2)^2}\bigg\{{n \choose 1}{n-1\choose 2}+2{n\choose 1}{n-1\choose 3}\bigg\} = \frac{1}{3}+\frac{1}{6(n-2)}.
    \end{split}
\end{equation*}

Similarly, as $d$ diverges to infinity, we have 
$$V_2\stackrel{P}{\to}\frac{1}{m(m-1)}\sum_{n+1\leq i\not= j\leq n+m}\bigg\{\frac{1}{m-2}\sum_{k=1, k\not= i,j}^{m}\delta({\bf Y}_{k},{\bf Y}_{j},{\bf Y}_{i})\bigg\}^2 = \frac{1}{3}+\frac{1}{6(m-2)}.$$

Thus, $T_{n,m}^\rho \stackrel{P}{\rightarrow}\frac{2}{3}+\frac{1}{6}\left(\frac{1}{n-2}+\frac{1}{m-2}\right)$. The same result holds for $\theta_1<\theta_3<\theta_2$ as well.

\vspace{0.05in}
Now consider the case, $\theta_3 = \max\{\theta_1,\theta_2\}$. Assume that $\theta_2<\theta_1=\theta_3$. In this case, the convergence of $\P[\rho({\bf Y}_1,{\bf X}_1)\leq \rho({\bf X}_2,{\bf X}_1)]$ is not clear, but $\P[\rho({\bf X}_1,{\bf Y}_1)\leq \rho({\bf Y}_2,{\bf Y}_1)]$ converges to zero as $d$ diverges to infinity. Hence, $V_2$ converges to $1/3+1/6(m-2)$ in probability, and $V_1$ converges in probability to a non-negative random variable. Therefore,  $P(T_{n,m}^\rho > 1/3)$ converges to one. Similar arguments can be given for $\theta_1<\theta_2=\theta_3$ as well.

\vspace{0.05in}
Finally, consider the case $\theta_3>\max\{\theta_1,\theta_2\}$. In this case, we have
$\lim_{d\to\infty}\P[\rho({\bf X}_2,{\bf Y}_1)\leq \rho({\bf Y}_2,{\bf Y}_1)] = 0$ and $\lim_{d\to\infty}\P[\rho({\bf Y}_1,{\bf X}_1)\leq \rho({\bf X}_2,{\bf X}_1)] = 0.$
Hence as $d$ diverges to infinity, $V_1$ and $V_2$ converge in probability to $1/3+1/6(n-2)$ and $1/3+1/6(m-2)$, respectively. Thus, $T_{n,m}^\rho$ converges in probability to $2/3+1/6\left(1/(n-2)+1/(m-2)\right)$. 

\vspace{0.05in}
These three cases together imply that if $\theta_1>\min\{\theta_1,\theta_2\}$, $P(T_{n,m}^\rho>1/3) \rightarrow 1$ as $d \rightarrow \infty.$ 
\end{proof}

\noindent
\begin{proof}[\bf Proof of Lemma \ref{more-than-1/3}]
Here, we have $d^{-1/2}\|{\bf X}_1-{\bf X}_2\|\stackrel{P}{\to}\sigma_F\sqrt{2}$, $d^{-1/2}\|{\bf Y}_1-{\bf Y}_2\|\stackrel{P}{\to}\sigma_G\sqrt{2}$ and $d^{-1/2}\|{\bf X}_1-{\bf Y}_1\|\stackrel{P}{\to}\sqrt{\sigma_F^2+\sigma_G^2+\nu^2}$ as $d \rightarrow \infty$ (see Lemma 3.1). Let these limiting values be denoted by $\theta_1$, $\theta_2$, and $\theta_3$, respectively. If $\nu^2+(\sigma_F-\sigma_G)^2>0$, one can check that $\theta_3>\sqrt{\theta_1\theta_2} \ge \min\{\theta_1,\theta_2\}$. Hence the proof follows from Lemma A.4.
\end{proof}


\noindent
\begin{proof}[\bf Proof of Theorem \ref{HDLSS-L2}]
It follows from Lemma A.\ref{HDLSS} that under the condition $\nu^2+(\sigma_F-\sigma_G)^2>0$, $P(T_{n,m}^{\ell_2}>1/3)$ converges to $1$ as $d$ tends to infinity. We have also seen that the cut-off of the permutation test $c_{1-\alpha}$ has an upper bound ${2}/\{{3\alpha(\min\{n,m\}-2)}\}$, which does not depend on the dimension $d$. Therefore, if $\min\{n,m\}\geq 2+2/\alpha$, the test based on $T_{n,m}^{\ell_2}$ rejects $H_0$ with probability tending to $1$ as $d$ grows to infinity.  
\end{proof}

\noindent
\begin{proof}[\bf Proof of Lemma \ref{part-wise-gdist}]
This lemma is taken from \cite{sarkar2018some}. The proof can be found on page 5 (see Lemma 1) of that article.
\end{proof}

\noindent
\begin{proof}[\bf Proof of Theorem \ref{HDLSS-gdist}]
We use a sub-sequence argument to prove this theorem. Let $\{d_k\}$ be an arbitrary sub-sequence of the sequence of natural numbers. Under (A4) and $\liminf_{d\to\infty}e_{h,\psi}(F,G)>0$, there exists a further subsequence $\{d_k'\}$ such that $\lim_{d_k'\to\infty}e_{h,\psi}(F,G)>0$, and the corresponding limits of the three terms in $e_{h,\psi}(F,G)$ exist. Let $\theta_1$, $\theta_2$ and $\theta_3$ be the limiting values of
 $d^{-1}\sum_{q=1}^d \{\psi(|X_1^{(q)}-X_2^{(q)}|^2), d^{-1}\sum_{q=1}^d \{\psi(|Y_1^{(q)}-Y_2^{(q)}|^2)$ and $d^{-1}\sum_{q=1}^d \{\psi(|X_1^{(q)}-Y_1^{(q)}|^2)$, respectively, along the sub-sequence $\{d_k'\}$. Since $\lim_{d_k'\to\infty}e_{h,\psi}(F,G)>0$, we have  $2\theta_3>\theta_1+\theta_2$. Hence, using Lemma A.\ref{HDLSS}, we get $P(T_{n,m}^{h,\psi}>1/3) \rightarrow 1$ as $d_k' \rightarrow \infty$. Since $\{d_k'\}$ is the sub-sequence of an arbitrary sequence $\{d_k\}$, we can conclude that $P(T_{n,m}^{h,\psi}>1/3) \rightarrow 1$ as $d \rightarrow \infty$. Now, using arguments similar to those in the proof of Theorem \ref{HDLSS-L2}, one can establish the consistency of the level $\alpha$ test when $\min\{n,m\}\geq 2+2/\alpha$.  
\end{proof}

\noindent
\begin{lemmaA}
If ${\bf X}_1,{\bf X}_2 \stackrel{iid}{\sim} F$ and ${\bf Y}_1,{\bf Y}_2 \stackrel{iid}{\sim} G$ are independent random vectors, then
$$\Theta_{\ell_2}^2(F,G)\geq \left\{\P\{\|{\bf X}_1-{\bf Y}_1\|\leq \|{\bf Y}_2-{\bf Y}_1\|\}-1/2\right\}^2+\left\{\P\{\|{\bf Y}_1-{\bf X}_1\|\leq \|{\bf X}_2-{\bf X}_1\|\}-1/2\right\}^2.\vspace{-0.2in}
$$
\label{BDLB}
\end{lemmaA}

\noindent
\begin{proof}[\bf Proof]
We know that for any random variable $Z$, $(\E\{Z\})^2\leq \E\{Z^2\}$.
Using this fact twice, we get
\begin{equation*}
    \begin{split}
        \Theta_{\ell_2}^2(F,G) & = \int\left\{F(B(u,\ell_2(v,u)))-G(B(u,\ell_2(v,u)))\right\}^2\,(dF(u)dF(v)+dG(u)dG(v))\\
        & \geq  \left\{\int F(B(u,\ell_2(v,u)))-G(B(u,\ell_2(v,u)))\,dF(u)dF(v)\right\}^2\\
        &~~~~~~~~~~+\left\{\int F(B(u,\ell_2(v,u)))-G(B(u,\ell_2(v,u)))\,dG(u)dG(v)\right\}^2\\
        & = \left\{\P\{\|{\bf Y}_1-{\bf X}_1\|\leq \|{\bf X}_2-{\bf X}_1\|\}-1/2\right\}^2+
        \left\{\P\{\|{\bf X}_1-{\bf Y}_1\|\leq \|{\bf Y}_2-{\bf Y}_1\|\}-1/2\right\}^2
    \end{split}
\end{equation*}

The last equality follows from the fact that if ${\bf X}_1,{\bf X}_2,{\bf X}_3 \stackrel{iid}{\sim} F_0$, where $F_0$ is a continuous distribution, we have $\int F_0(B(u,\ell_2(v,u)))\,dF_0(u)dF_0(v) = \P\{\|{\bf X}_3-{\bf X}_1\|\leq\|{\bf X}_2-{\bf X}_1\|\} = \frac{1}{2}$.
\end{proof}

\noindent
\begin{lemmaA}
Consider two $d$-dimensional random variables ${\bf X} = (\xi_1,0,0,\ldots,0)^\top$ and ${\bf Y}=(\xi_2,0,0,\ldots,0)^\top$, where $\xi_1\sim N(\mu_1,1)$, $\xi_2\sim N(\mu_2,1)$, and they are independent. Let $P_1$ and $P_2$  denote the distributions of ${\bf X}$ and ${\bf Y}$, respectively. 
If $\mu_1=cn^{-1/2}$ and $\mu_2 = -cm^{-1/2}$ for some $c>0$, then there exists a constant $C>0$ independent of the dimension $d$ such that
$\Theta^2(P_1,P_2)\geq C\big(\frac{1}{\sqrt{n}}+\frac{1}{\sqrt{m}}\big)^2.$
This lower bound is tight up to a constant factor.
\label{BDLBD}
\end{lemmaA}

\noindent
\begin{proof}[\bf Proof]
Let $\xi_{11},\xi_{12} \stackrel{iid}{\sim} N(\mu_1,1)$ and $\xi_{21},\xi_{22} \stackrel{iid}{\sim} N(\mu_2,1)$ be independent random variables. In view of Lemma A.\ref{BDLB}, for  ${\bf X}_1=(\xi_{11},0,\ldots,0)^\top,{\bf X}_2=(\xi_{12},0,\ldots,0)^\top\sim P_1$,  ${\bf Y}_1=(\xi_{21},0,\ldots,0)^\top$ and ${\bf Y}_2=(\xi_{22},0,\ldots,0)^\top\sim P_2$, it is enough to prove that 
$$\left[\P\{\|{\bf X}_1-{\bf Y}_1\|\leq \|{\bf Y}_2-{\bf Y}_1\|\}\hspace{-0.02in}-\hspace{-0.02in}1/2\right]^2\hspace{-0.02in}+\hspace{-0.02in}\left[\P\{\|{\bf Y}_1-{\bf X}_1\|\leq \|{\bf X}_2-{\bf X}_1\|\}\hspace{-0.02in}-\hspace{-0.02in}1/2\right]^2\geq C\Big(\frac{1}{\sqrt{n}}+\frac{1}{\sqrt{m}}\Big)^2,$$
We derive the lower bounds for these two terms separately. Note that
\begin{equation*}
    \begin{split}
        \bigg|\P\{\|{\bf X}_1-{\bf Y}_1\|\leq \|{\bf Y}_2-{\bf Y}_1\|\}-1/2\bigg| & = \bigg|\P\{|\xi_{11}-\xi_{21}|\leq |\xi_{22}-\xi_{22}|\}-1/2\bigg|\\& = \bigg|\P\{|\xi_{11}-\xi_{21}|^2- |\xi_{22}-\xi_{21}|^2\leq 0\}-1/2\bigg|\\
        & = \bigg|\P\{(\xi_{11}+\xi_{22}-2\xi_{21})(\xi_{11}-\xi_{22})\leq 0\}-1/2\bigg|.
        \end{split}
\end{equation*}
So, taking $T_1=\xi_{11}-\xi_{22}$ and $S_1 = \xi_{11}+\xi_{22}-2\xi_{21}$, we get
$$
    \bigg|\P\{\|{\bf X}_1-{\bf Y}_1\|\leq \|{\bf Y}_2-{\bf Y}_1\|\}-1/2\bigg|     = \bigg|\P\{S_1T_1\leq 0\}-1/2\bigg| = \bigg|\E\{\P\{S_1T_1\leq 0\mid S_1\}-1/2\}\bigg|.$$
Here, $T_1$ and $S_1$ jointly follow a bivariate normal distribution with  $E(T_1)=c(\frac{1}{\sqrt{n}}+\frac{1}{\sqrt{m}}\big)$, $E(S_1)=c(\frac{1}{\sqrt{n}}+\frac{1}{\sqrt{m}}\big)$, $Var(T_1)=2$, $Var(S_1)=6$ and $Cov(T_1,S_1)=0$. Therefore,
\begin{align*}
    \bigg|\E\big\{\P\{S_1T_1\leq 0\mid S_1\}-1/2\}\bigg| &=\bigg|\E\bigg\{\Phi\bigg(-c\Big(\frac{1}{\sqrt{n}}+\frac{1}{\sqrt{m}}\Big)\frac{S_1}{|S_1|\sqrt{2}}\bigg)-1/2\bigg\}\bigg|\\
    &\geq \frac{c}{\sqrt{2}}\Big(\frac{1}{\sqrt{n}}+\frac{1}{\sqrt{m}}\Big)\phi(\sqrt{2}c),
\end{align*} 
where $\phi(\cdot)$ and $\Phi(\cdot)$ denote the density function and the distribution function of the standard normal variate, respectively. The last inequality is obtained by using the mean value theorem and the fact that $\phi(t)$ is decreasing in $|t|$. So, we get
$$\Big|\P\{\|{\bf X}_1-{\bf Y}_1\|\leq \|{\bf Y}_2-{\bf Y}_1\|\}-1/2\Big|\geq \frac{c}{\sqrt{2}}\Big(\frac{1}{\sqrt{n}}+\frac{1}{\sqrt{m}}\Big)\phi(\sqrt{2}c).$$
Similarly, we can derive the same lower bound for $\Big|\P\{\|{\bf Y}_1-{\bf X}_1\|\leq \|{\bf X}_2-{\bf X}_1\|\}-1/2\Big|$ as well. So, we can find a constant $C>0$ independent of the dimension $d$ such that
$$\Theta_{\ell_2}^2(P_1,P_2)\geq C\Big(\frac{1}{\sqrt{n}}+\frac{1}{\sqrt{m}}\Big)^2.$$

To show that this lower bound is tight, notice that
$$\Theta^2(P_1,P_2) \leq \sup_{A}|P_1(A)-P_2(A)| \leq KL(P_1,P_2)
= \frac{c^2}{2}(\mu_1-\mu_2)^2
= \frac{c^2}{2}\Big(\frac{1}{\sqrt{n}}+\frac{1}{\sqrt{m}}\Big)^2,$$
where $KL(\cdot,\cdot)$ denotes the Kullback-Leibler divergence between two probability measures.
Here, the first inequality follows trivially, and the second one is known as Pinsker's inequality \citep[see Lemma 2.5 in][]{tsybakov2009introduction}.  Hence the lower bound is tight up to a constant term.
\end{proof}

\noindent
\begin{proof}[\bf Proof of Theorem \ref{minimax-lower-bound}]
The minimax lower bound can be obtained based on the standard application of Neyman-Pearson lemma \citep[see][]{baraud2002,kim2020robust}. Let the distributions of the sample under the null and alternative hypothesis be denoted as $Q_0$ and $Q_1$ respectively. Then following our notations, we have
\begin{equation*}
    \begin{split}
        R_{n,m,d}(\epsilon) & \geq 1-\alpha-\sup_{A}|Q_0(A)-Q_1(A)|\geq 1-\alpha-\sqrt{\frac{1}{2}KL(Q_0,Q_1)},
    \end{split}
\end{equation*}
where the second inequality is obtained from Pinsker's inequality \citep[see][]{tsybakov2009introduction}. {Now suppose that $P_1$ and $P_2$ are the  distributions corresponding to ${\bf X}=(\xi_1,0,\ldots,0)^{\top}$ and ${\bf Y}=(\xi_2,0,\ldots,0)^{\top}$, where $\xi_1$ and $\xi_2$ differ only in their means, which are
$$\mu_1 = \frac{\sqrt{2}(1-\alpha-\zeta)}{\sqrt{n}}\hspace{20pt}\mbox{and}\hspace{20pt}\mu_2 = -\frac{\sqrt{2}(1-\alpha-\zeta)}{\sqrt{m}},$$
respectively.} Define $k(\alpha,\zeta):=(1-\alpha-\zeta)^2\Big(\phi\big(\sqrt{2}(1-\alpha-\zeta)\big)\Big)^2$. Then by Lemma A.\ref{BDLB}, $(P_1, P_2)  \in \mathcal{F}(c\lambda({n,m}))$ for all $0<c<k(\alpha,\zeta)$. We also have,
$$KL(Q_0,Q_1) = \frac{n}{2}\mu_1^2+\frac{m}{2}\mu_2^2 = 2(1-\alpha-\zeta)^2.$$
Therefore,
$R_{n,m,d}(c\lambda({n,m}))\geq \zeta$
for all $0<c<k(\alpha,\zeta)$. Since $\zeta$ and $k(\alpha,\zeta)$ do not depend on $n,m$ and $d$, this trivially satisfies the condition $\liminf\limits_{n,m,d\rightarrow \infty}R_{n,m,d}(c\lambda({n,m}))\geq \zeta$ for all $0<c<k(\alpha,\zeta)$.
\end{proof}


\noindent
\begin{proof}[\bf Proof of Theorem \ref{minimax-upper-bound}]
Here we want to show that for every positive $\alpha$ and $\zeta$, there exists a constant $K(\alpha,\zeta)$ such that
$\limsup_{n,m,d\to\infty}\sup_{(F,G)\in \mathcal{F}(c\lambda({n,m}))}\P_{F,G}^{n,m}\{T_{n,m}\leq c_{1-\alpha}\}\leq \zeta$ for all $c>K(\alpha,\zeta)$.
Let us first choose a constant $K_1$ such that 
$$ K_1\Big(\frac{1}{\sqrt{n}}+\frac{1}{\sqrt{m}}\Big)^2 \geq  \frac{1}{\alpha}\E\{T_{n,m}^\pi\mid \mathcal{U}\} = \frac{1}{6\alpha}\left(\frac{1}{n}+\frac{1}{m}+\frac{1}{n-2}+\frac{1}{m-2}\right).$$
Now, take any $(F,G)\in \mathcal{F}(c\lambda(n,m))$ such that $c>K_1$. Using the fact that $c_{1-\alpha} \le \frac{1}{\alpha}\E\{T_{n,m}^\pi\mid \mathcal{U}\}$ (see the proof of Lemma 2.1), we get
\begin{equation*}
    \begin{split}
        \P_{F,G}^{n,m}\{T_{n,m}\leq c_{1-\alpha}\} & \leq\P_{F,G}^{n,m}\{T_{n,m} \leq\frac{1}{\alpha}\E\{T_{n,m}^\pi\mid \mathcal{U}\}\}\\
&= \P_{F,G}^{n,m}\{-T_{n,m}+\E_{F,G}\{T_{n,m}\}\geq\E_{F,G}\{T_{n,m}\}-\frac{1}{\alpha}\E\{T_{n,m}^\pi\mid \mathcal{U}\}\}\}.
    \end{split}
\end{equation*}
Since $\E_{F,G}\{T_{n,m}\}\geq \Theta_\rho^2(F,G)\geq c \lambda(n,m) \ge K_1 \lambda(n,m) \ge \frac{1}{\alpha}\E\{T_{n,m,\pi}^\rho\mid\mathcal{U}\}$, using the Chebyshev's inequality, one gets
\begin{equation*}
   \begin{split}
& \P_{F,G}^{n,m}\{T_{n,m}  \leq c_{1-\alpha}\}  \leq \P_{F,G}^{n,m}\{-T_{n,m}+\E_{F,G}\{T_{n,m}\}\geq\E_{F,G}\{T_{n,m}\}-\frac{1}{\alpha}\E\{T_{n,m}^\pi\mid \mathcal{U}\}\}\}\\
  & \hspace{0.2in}       \leq \frac{Var_{F,G}(T_{n,m})}{\left(\E_{F,G}\{T_{n,m}\}-\frac{1}{\alpha}\E\{T_{n,m}^\pi\mid \mathcal{U}\}\right)^2}\\
       & \hspace{0.2in} \leq \frac{C_1 \Theta_\rho^2(F,G)\Big(\frac{1}{{n}}+\frac{1}{{m}}\Big)+C_2\Big(\frac{1}{{n}}+\frac{1}{{m}}\Big)^2}{\left(\frac{1}{6}\left(\frac{1}{n-2}+\frac{1}{m-2}\right)+\frac{1}{m}(p_{0}-p_{1})+\frac{1}{n}(p_{2}-p_{3})+\Theta_\rho^2(F,G)-\frac{1}{6\alpha}\Big(\frac{1}{n}+\frac{1}{m}+\frac{1}{n-2}+\frac{1}{m-2}\Big)\right)^2}\\
     & \hspace{0.2in} \leq \frac{C_1 \Theta_\rho^2(F,G)\Big(\frac{1}{{n}}+\frac{1}{{m}}\Big)+C_2\Big(\frac{1}{{n}}+\frac{1}{{m}}\Big)^2}{\left(\Theta_\rho^2(F,G)-\frac{1}{6\alpha}\Big(\frac{1}{n}+\frac{1}{m}+\frac{1}{n-2}+\frac{1}{m-2}\Big)\right)^2},
    \end{split}
\end{equation*}

This implies
$\limsup_{n,m,d\to\infty}\sup_{(F,G)\in\mathcal{F}(c\lambda(n,m))}\P_{F,G}^{n,m}\{T_{n,m}\leq c_{1-\alpha}\} \leq ({C_1 c+C_2})/{\left(c-\frac{1}{3\alpha}\right)^2}.$ 
One can notice that this upper bound is a decreasing function in $c$, and as $c$ grows to infinity, it goes to zero. Hence, for any $0<\zeta<1-\alpha$, there exists a constant $K_2>0$ such that this upper bound is smaller than $\zeta$. Now let $K(\alpha,\zeta) = \max\{K_1,K_2\}$. Then for $c > K(\alpha,\zeta)$ the maximum type II error rate is asymptotically upper bounded by $\zeta$. This establishes the theorem.
\end{proof}


\noindent
\begin{proof}[\bf Proof of Theorem \ref{High-Dimension-L2}]
Notice that if $F$ and $G$ are such that $\lim_{d\to\infty}\Theta_{\ell_2}^2(F,G)/\lambda({n,m})=\infty$, then from Theorem \ref{minimax-upper-bound}.  we have
$\lim_{d\to\infty}\P_{F,G}^{n,m}\{T_{n,m}\leq c_{1-\alpha}\} = 0.$
Hence the power of the test converges to 1.
\end{proof}


\noindent
\begin{proof}[\bf Proof of Theorem \ref{High-Dimenion-gdist}]
Using similar arguments as in the proofs of Theorems 4.1 and 4.2, one can show that if $h$ and $\psi$ are strictly increasing functions, then for testing $H_0:\Theta^2_{\varphi_{h,\psi}}(F, G) = 0$ against $H_1:\Theta^2_{\varphi_{h,\psi}}(F, G)>\epsilon$, the minimax separation rate is $\lambda({n,m}) = (1/\sqrt{n}+1/\sqrt{m})^2$, and the permutation test based on $T_{n,m}^{h,\psi}$ is minimax rate optimal. Hence, one gets a similar conclusion as in Theorem \ref{High-Dimension-L2}.  
\end{proof}


\noindent
\begin{proof}[\bf Proof of Proposition \ref{NSA-prop}]
Here, we use Lemma A.\ref{BDLB} to establish the condition of Theorem \ref{High-Dimension-L2} for different $\beta$.  Assume that ${\bf X}_1,{\bf X}_2\sim F = \prod_{i=1}^d \mathcal{N}_1(1/d^\beta,1)$,  ${\bf Y}_1,{\bf Y}_2\sim G= \prod_{i=1}^d \mathcal{N}_1(-1/d^\beta,1)$, and they are independent. Then
\begin{equation*}
    \begin{split}
        \Theta^2_{\ell_2}(F,G) & \ge \Big[\P\big(\|{\bf X}_1-{\bf Y}_1\|\leq \|{\bf Y}_2-{\bf Y}_1\|\big)-1/2\Big]^2+\Big[\P\big(\|{\bf Y}_1-{\bf X}_1\|\leq \|{\bf X}_2-{\bf X}_1\|\big)-1/2\Big]^2 \\
        & = \Big[\P\big(\|{\bf X}_1-{\bf Y}_1\|^2-\|{\bf Y}_2-{\bf Y}_1\|^2\leq 0\big)-1/2\Big]^2\\
        & ~~~~~~~~~~+\Big[\P\big(\|{\bf Y}_1-{\bf X}_1\|^2-\|{\bf X}_2-{\bf X}_1\|^2\leq 0\big)-1/2\Big]^2\\
        & = \bigg[\P\bigg(\sum_{i=1}^d T_iS_i\leq 0\bigg)-\frac{1}{2}\bigg]^2+\bigg[\P\bigg(\sum_{i=1}^d T'_iS'_i\leq 0\bigg)-\frac{1}{2}\bigg]^2,
    \end{split}
\end{equation*}
where $T_i = X_{1i}-Y_{2i}$, $S_i = X_{1i}+Y_{2i}-2Y_{1i}$, $T_i' = Y_{1i}-X_{2i}$ and $S_i'= Y_{1i}+X_{2i}-2X_{1i}$ ($i=1,2,\ldots,d$). Clearly, $T_i, S_i$ are independent, and so are $T_i', S_i'$. Here $S_1, S_2,\ldots, S_d\stackrel{iid}{\sim}N(\frac{2}{d^\beta},6)\big)$, and $S_i^{'}$ has the same distribution as $-S_i$ for all $i=1,2,\ldots,d$. Now,
\begin{equation*}
    \begin{split}
        & \Big[\P\big(\|{\bf X}_1-{\bf Y}_1\|\leq \|{\bf Y}_2-{\bf Y}_1\|\big)-1/2\Big]^2+\Big[\P\big(\|{\bf Y}_1-{\bf X}_1\|\leq \|{\bf X}_2-{\bf X}_1\|\big)-1/2\Big]^2 \\
        & ~~~~~~= \left[\E\Bigg\{\Phi\Bigg(-\frac{2}{d^\beta}\frac{\sum_{i=1}^d S_i}{\sqrt{2\sum_{i=1}^d S_i^2}}\Bigg)\Bigg\}-\frac{1}{2}\right]^2+\left[\E\Bigg\{\Phi\Bigg(\frac{2}{d^\beta}\frac{\sum_{i=1}^d S_i^\prime}{\sqrt{2\sum_{i=1}^d {S_i^\prime}^2}}\Bigg)\Bigg\}-\frac{1}{2}\right]^2.\\
        & ~~~~~~= 2\left[\E\Bigg\{\Phi\Bigg(-\frac{2}{d^\beta}\frac{\sum_{i=1}^d S_i}{\sqrt{2\sum_{i=1}^d S_i^2}}\Bigg)\Bigg\}-\frac{1}{2}\right]^2 
    \end{split}
\end{equation*}
Hence, studying the behavior of $Z(\beta)={\big(2\sum_{i=1}^dS_i}\big)/{\big(d^{\beta}\sqrt{2\sum_{i=1}^d S_i^2}\big)}$ for different values of $\beta$ will yield the conditions for the consistency of our test. 

Note that $\sum_{i=1}^dS_i/d^{\beta+1/2}\sim N(2/d^{2\beta-1/2},6/d^{2\beta})$. Hence for $\beta<1/4$, $\frac{1}{d^{\beta+1/2}}\sum_{i=1}^dS_i \stackrel{P}{\rightarrow} \infty$  and $\sum_{i=1}^d S_i^2/d \stackrel{P}{\rightarrow} 6$. So, for $\beta<1/4$, $Z(\beta)\stackrel{P}{\rightarrow} \infty$. For $\beta=1/4$, $\sum_{i=1}^dS_i/d^{\beta+1/2} \stackrel{P}{\rightarrow} 2$. So, $Z(\beta) \stackrel{P}{\rightarrow} 2/\sqrt{3}$. Therefore, for $\beta\leq 1/4$, we have 
$$\liminf_{d\to\infty}\left[\E\Bigg\{\Phi\Bigg(-\frac{2}{d^\beta}\frac{\sum_{i=1}^d S_i}{\sqrt{2\sum_{i=1}^d S_i^2}}\Bigg)\Bigg\}-\frac{1}{2}\right]^2>0,$$
which in turn implies that $\liminf_{d\to\infty}\Theta_{\ell_2}^2(F,G)>0$. This proves Proposition \ref{NSA-prop}{(a)}. 

\vspace{0.05in}
For $1/4<\beta<1/2$, notice that $d^{2\beta-1/2}\sum_{i=1}^dS_i/d^{\beta+1/2}\sim N(2,6d^{2\beta-1})$. So, as $d$ tends to infinity,  $d^{2\beta-1/2}\sum_{i=1}^dS_i/d^{\beta+1/2} \stackrel{P}{\rightarrow} 2$. Now, if we take $n\asymp m\asymp d^\gamma$, to match this convergence rate so that $\Theta_{\ell_2}^2(F,G)/\lambda({n,m})$ diverges to infinity, we require the following
\begin{align*}
\lim_{d\to\infty}d^{\gamma}&\left[\E\Bigg\{\Phi\bigg(-\frac{2}{d^{2\beta-1/2}}\frac{d^{\beta-1}\sum_{i=1}^dS_i}{\sqrt{2\sum_{i=1}^d S_i^2/d}}\bigg)\Bigg\}-\frac{1}{2}\right]^2
= \infty.
\end{align*}

This is possible when $\gamma>4\beta-1$. Note that for $\beta=1/2$,  $\sum_{i=1}^dS_i/d^{1/2}$ forms a tight sequence. In this case, we need
\begin{align*}
\lim_{d\to\infty}d^{\gamma} &\left[\E\Bigg\{\Phi\bigg(-\frac{2}{d^{\beta}}\frac{d^{-1/2}\sum_{i=1}^dS_i}{\sqrt{2\sum_{i=1}^d S_i^2/d}}\bigg)\Bigg\}-\frac{1}{2}\right]^2
= \infty,
\end{align*}
which is satisfied when $\gamma>1 = 4\beta-1$. This proves Proposition \ref{NSA-prop}{(b)}.

\vspace{0.05in}
Now for $\beta>1/2$, we have $\sum_{i=1}^dS_i/d^{1/2}\sim N(2/d^{\beta-1/2},6)$, and hence it is a tight sequence of random variables. In this scenario we require
\begin{align*}\lim_{d\to\infty}d^{\gamma}
& \left[\E\Bigg\{\Phi\bigg(-\frac{2}{d^{\beta}}\frac{d^{-1/2}\sum_{i=1}^dS_i}{\sqrt{2\sum_{i=1}^d S_i^2/d}}\bigg)\Bigg\}-\frac{1}{2}\right]^2
= \infty,
\end{align*}
which is satisfied if $\gamma>2\beta$. Also notice that when $\beta>1/2$ and $\gamma<2\beta-1$, the Kullback-Leibler Divergence {($KL(Q_1,Q_0) \asymp d^{\gamma-2\beta+1}$)} converges to zero with increasing dimensions. Hence in this scenario, the asymptotic type II error rate of any test remains bounded below by $1-\alpha$, i.e., no tests have asymptotic power more than the nominal level $\alpha$. This completes the proof of Proposition \ref{NSA-prop}{(c)}.
\end{proof}

\end{document}